\documentclass{article}
\usepackage{amssymb,amsmath,amsthm,framed,algorithmic,graphicx}
\usepackage{fullpage}
\usepackage{bbm}
\usepackage{authblk}
\usepackage{wrapfig}

\RequirePackage[OT1]{fontenc}
\RequirePackage[numbers]{natbib}
\usepackage{enumerate}
\usepackage[utf8]{inputenc} 
\usepackage{natbib} 
\usepackage{amsmath,amssymb,amsthm}  
\usepackage[pdftex,bookmarks,colorlinks,breaklinks]{hyperref}  
\usepackage{url}
\usepackage{multirow}
\usepackage{verbatim}
\usepackage{graphicx}
\hypersetup{linkcolor=blue,citecolor=blue,filecolor=cyan,urlcolor=blue} 

\newtheorem{theorem}{Theorem}[section]
\newtheorem{lemma}[theorem]{Lemma}

\newtheorem{definition}[theorem]{Definition}

\DeclareFontFamily{U}{mathx}{\hyphenchar\font45}
\DeclareFontShape{U}{mathx}{m}{n}{<-> mathx10}{}
\DeclareSymbolFont{mathx}{U}{mathx}{m}{n}
\DeclareMathAccent{\widebar}{0}{mathx}{"73}

\usepackage[all]{xy}

\def\eps{\varepsilon}
\def\wt{\widetilde}
\def\wh{\widehat}
\def\wb{\widebar}

\def\E{\mathbb{E}}
\def\P{\mathbb{P}}

\def\S{\mathcal{S}}
\def\M{\mathcal{M}}

\def\Cov{\textrm{Cov}}

\def\L{\mathcal{L}}
\def\R{\mathbb{R}}

\usepackage{fancyhdr}
\def\dsst{\displaystyle} 

\def\Ord{O}
\def\Xsp{\mathcal{M}}
\def\metric{\rho}
\def\osim{\mathcal{S}}
\def\lsim{\widehat{\mathcal{S}}}
\def\patht{{t_0}}
\def\MDS{\Phi}

\def\ls{ATLAS} 
\def\csp{\mathcal{A}} 
\def\icx{z} 

\title{ATLAS: A geometric approach to learning high-dimensional stochastic systems near manifolds}
\author{Miles Crosskey$^{1}$, Mauro Maggioni$^{1,2,3}$}
\affil{Department of Mathematics$^{(1)}$, Electrical and Computer Engineering$^{(2)}$ and Computer Science$^{(3)}$, Duke University, Durham, NC, 27708}

\begin{document}

\maketitle

\begin{abstract}
\noindent When simulating multiscale stochastic differential equations (SDEs) in high-dimensions, separation of timescales, stochastic noise and high-dimensionality can make simulations prohibitively expensive. The computational cost is dictated by microscale properties and interactions of many variables, while the behavior of interest often occurs at the macroscale level and at large time scales, often characterized by few important, but unknown, degrees of freedom. For many problems bridging the gap between the microscale and macroscale by direct simulation is computationally infeasible. In this work we propose a novel approach to automatically learn a reduced model with an associated fast macroscale simulator. Our unsupervised learning algorithm uses short parallelizable microscale simulations to learn provably accurate macroscale SDE models, which are continuous in space and time. The learning algorithm takes as input: the microscale simulator, a local distance function, and a homogenization spatial or temporal scale, which is the smallest time scale of interest in the reduced system. The learned macroscale model can then be used for fast computation and storage of long simulations. We prove guarantees that related the number of short paths requested from the microscale simulator to the accuracy of the learned macroscale simulator. We discuss various examples, both low- and high-dimensional, as well as results about the accuracy of the fast simulators we construct, and its dependency on the number of short paths requested from the microscale simulator.
\end{abstract}
 

\section{
Introduction
}

\indent 
High-dimensional dynamical systems arise in a wide variety of applications, from the study of macromolecules in biology to finance, to multi-agent systems, to climate modeling.
In many cases these systems are stochastic by nature, or are well-approximated by stochastic processes, for example as a consequence of slow-fast scale phenomena in the system.
Simulations typically require significant amounts of computation, for several reasons. First of all each time step of the numerical scheme is often expensive because of the large dimensionality of the space, and the large number of interactions that need to be computed. Secondly, fast timescales and/or stochasticity may force each time step to be extremely small in order to have the requested accuracy. Finally, large-time behavior of the system may be dominated by rare transition events between stable regions, requiring very long paths to understand large-time dynamics.
A large amount of research spanning multiple fields tackles the problems above.

Suppose we are given a high dimensional stochastic simulator, and we are interested in the large-time behavior of the system, but are faced with the problem of prohibitively expensive costs to run long simulations. 
What could be computable in a highly parallel fashion is an ensemble of short paths \cite{larson2002folding}. We therefore ask: what can be learned from ensembles of short paths? Several crucial problems to be addressed include: where in state space such short paths should be started? how many paths should be run locally? for how long? how does the local accuracy depend on these parameters? and once these local paths are constructed, and perhaps local simulators approximating the local dynamics are constructed, how can they be stitched together to produce a global simulation scheme? What can be guaranteed about the accuracy of such a global scheme, for large times? Some examples, among many, in this direction are Markov State Models \cite{pande2010everything,bowman2009progress,larson2002folding}, milestoning \cite{milestoning}, and several other techniques (e.g. \cite{Noe:2009:19011,Huisinga:2004:419,Deuflhard:2000:39,Shardlow00aperturbation,Kevrekidis:2003:715,vanden2003fast,kevrekidis,Kevrekidis:2003:715,Antoulas01asurvey,HuisingaSchuetteStuart2003} and references therein). Many of these methods are based on discretizations of the state space into regions, and measure transitions between such regions, others on biasing the potential to speed up exploration, yet others estimate local statistics and use them to coarse grain the system. Our method is related to some of these existing approaches, but uniquely combines them with ideas from machine learning, stochastic processes and dimension reduction, and introduces several novel key elements that combined lead to an accurate yet low-dimensional estimate of the generator of the diffusion process (rather than estimating discretized transition densities), with guarantees on the large time accuracy of the simulator we construct.  Our construction can be seen as a higher order approach compared to Markov state models, since we fit a linear reduced model, continuous in space and time, to each region, and smoothly glue these models together. Because of this, we are also able to approximate the original dynamics by a process which is continuous in time and space.

The philosophy of reducing a high-dimensional system to a low-dimensional surrogate is well-established as enabling the simulation of complex, large, high-dimensional systems, and more methods have been proposed than we can possibly discuss here. These include model reduction \cite{Moore:PCALinearSystems,Antoulas01asurvey,HuisingaSchuetteStuart2003}, homogenization of PDE's \cite{HornungHomogenization,gilbert1998comparison}, coarse-grained dynamics of high-dimensional systems \cite{kevrekidis,Kevrekidis:2003:715}, and multiscale modeling \cite{Majda:MathFrameworkStochasticClimate,Shardlow00aperturbation,Kevrekidis:2003:715,vanden2003fast}. We refer the reader to \cite{HMM} for a summary of the motivations and applications of several of these techniques, and to the references therein.

We take concepts from {\em{manifold learning}} \cite{GMRA, brand2002charting, LTSA, Saul_tgfl} in order to learn an underlying low-dimensional manifold around which most trajectories concentrate with high probability. We approximate the macroscale manifold with linear low-dimensional subspaces locally, which we call charts. These charts enable us to learn local properties of the system in low dimensional Euclidean space. Geometric Multi-Resolution Analysis (GMRA), introduced in \cite{GMRA} uses this concept to approximate high dimensional distributions on manifolds.
These techniques perform the {\em{model reduction}} step, mapping the high-dimensional system from $\mathbb{R}^D$ down to $d$-dimensions, yielding a small set of coordinates describing the effective small number of degrees of freedom of the system.
 
We combine this dimension reduction step with {\em{homogenization theory}} \cite{multi-meth, gilbert1998comparison,vanden2003fast} and learn a local low-dimensional approximation of the system at a certain time-(or space-)scale $t_0$. We note that this approximation is not necessarily accurate at timescales shorter than $t_0$. Short-time events may be complex, high-dimensional, highly stochastic or deterministic but chaotic, and we do not seek a simulator reproducing these fine-scale behaviors. We are interested though in the ``net effect'' and implications of these phenomena at timescale $t_0$ and larger. Locally we fit a simple reduced system, e.g. a constant coefficient SDE to each chart. If the macroscale simulator is well approximated by a smooth SDE, then constant coefficient SDEs will approximate the system well locally. This smooth SDE is approximating the original simulator above a certain timescale.

In order to obtain a global simulator, we add a last crucial ingredient: we construct an approximation to the transition maps between charts, generating a numerical approximation to a manifold {\em{atlas}}. Learning such transition maps between charts is necessary to allow us to smoothly combine simulators on distinct charts into one global simulator on the atlas. The simulator we construct we then call the \ls, and we show that under appropriate conditions it captures long term statistics of the dynamics of the original system.

Finally, we note that accurate samples from the stationary distribution is a valuable tool in studying many dynamical systems. Reduced large-time models for complex high dimensional dynamical systems is sometimes obtained using so called ``reaction coordinates'', a set of global low dimensional coordinates describing the important states of the system (see e.g. \cite{Boaz:DiffusionMapsSpectralClustering,Qi:2010:6979,Cho:2006:586,Hanggi:1990:251,Das:2006:9885,Szabo:1980:4350,Clementi:LowDimensionaFreeEnergyLandscapesProteinFolding,Peters:2006:054108,Socci:1996:5860,Berezhkovskii:2005:014503,milestoning} and references therein). Several of the techniques (but not all, e.g. notably the string method of \cite{Vanden-Eijnden:2009:194103,E:2007:164103,Ren:2005:134109}) need samples from long simulations of the system and/or the stationary distribution. For example diffusion maps and their extensions to the study of molecular dynamics data (see \cite{DiffusionMaps,Coifman:2008:842,RZMC:ReactionCoordinatesLocalScaling}) require  many samples from the stationary distribution to guarantee accuracy. These reaction coordinates allow further analysis of dynamical systems by easily identifying stable states, and, most importantly, transitions paths between such states and reaction coordinates parametrizing such transitions. In fact, part of the motivation for this work was observing that the slowest part of running diffusion maps on such complicated high dimensional systems was obtaining the samples from the stationary distribution.

The paper is organized as follows: in section \ref{s:main} we describe at high level our construction, algorithm, and informally state the main result on the accuracy of the \ls\ for large times; then we illustrate the algorithm on simple examples. In section \ref{sec_alg} we discuss the algorithm in detail. In section \ref{sec_err} we state and prove our main result. In section \ref{sec_exs} we present a wide range of examples. We conclude with a  discussion in section \ref{sec_disc}.


\section{
Construction and Main Results
}
\label{s:main}

The geometric assumption underlying our construction is that the dynamics of the Markovian stochastic dynamical system of interest $(Y_t)_{t\ge0}$ in $\R^D$  is concentrated on or near an intrinsically low-dimensional manifold $\Xsp$ of dimension $d$, with $d\ll D$. We refer to $\Xsp$ as the {\em{effective state space}} of the system, as opposed to the full state space $\R^D$.
This type of model may be appropriate in a wide variety of situations:
\begin{itemize}
\item[(i)] the system has $d$ degrees of freedom, and is therefore constrained (under suitable smoothness assumptions) to a $d$-dimensional manifold $\Xsp$;
\item[(ii)] as in (i), but possibly with small deterministic or stochastic violations of those constraints (perhaps at a fast scale), but such that the trajectories stay close to $\Xsp$ at all times.
\end{itemize}
In these cases it makes sense to approximate $\Xsp$ by an efficient low-dimensional approximation $\csp$, such as a union of $d$-dimensional linear affine sets (charts) \cite{CM:MGM2,MMS:NoisyDictionaryLearning}, and the dynamics of $Y_t$ by surrogate dynamics on the atlas $\csp$. Learning dynamics on $\csp$ reduces the problem from learning a high-dimensional global simulator to a low-dimensional local simulator, together with appropriate transitions between local simulators in different charts. We also gain computational efficiency by using the structure $\csp$: long paths may be more quickly stored and simulated in lower dimensions. We will make assumptions about the geometry of $\Xsp$ and on the underlying macroscale simulator, in order to prove large time accuracy results for the \ls. We expect this approach to be valid much more generally, and this hope is supported by our numerical experiments.

While in this paper we consider a special class of stochastic dynamical systems, those well-approximated by low-dimensional SDEs such as those leading to advection-diffusion equations along a manifold, the framework can be significantly extended, as we briefly discuss later in section \ref{sec_disc}, and this will be subject of future work.

\subsection{Main Ideas and Steps}
Our construction takes as input:
\begin{itemize}
\item a dense enough sample of $\Xsp$, or a way of sampling $\Xsp$ in a rather uniform way (both of these statements will be quantified later, see section \ref{e:netconstruction});
\item  a simulator $\osim$ for the stochastic dynamical system $(Y_t)_{t\ge0}$, which may be started upon request at any specified initial condition and run for a specified amount of time;
\item a distance function $\rho$ to be used to for measuring distances between pairs of data points returned by the simulator;
\item a spatial homogenization parameter $\delta$;  
\item the dimension $d$ of the effective state space, and a confidence parameter $\tau$.
\end{itemize}
We note here that the homogenization scale $\delta$ can also be given as a temporal scale $\patht$, and the two are related by natural scalings in the underlying dynamical system. Given a time $\patht$, running paths of length $\patht$ and examining the average distance traveled by such paths reveals a corresponding natural spatial scale $\delta(\patht)$ (in fact this is done in example \ref{ex_fcn}). Inversely, given $\delta$, one could choose $\patht$ so that the average distance traveled by paths is approximiately $\delta$. We later discuss the accuracy of the simulator, which is a function of the parameter $\delta$. 

We remark that while $d$ is here considered as a parameter for the algorithm, in fact there is a lot of work on estimating the intrinsic dimension of high-dimensional data sets that would be applicable here. In particular, the Multiscale SVD techniques of \cite{LMR:MGM1,MM:MultiscaleDimensionalityEstimationAAAI} have strong guarantees, are robust with respect to noise, and are computationally efficient. See also \cite{MMS:NoisyDictionaryLearning} for finite sample guarantees on the approximation of manifolds by local affine approximate tangent spaces. We will mention again the problem of estimating $d$ when we construct the local charts in section \ref{sec_LMDS}.
 
The confidence parameter $\tau$ sets the probability of success of the algorithm (at least $1-2e^{-\tau^2}$), and is related to the number of sample paths one must use to approximate the local parameters of the simulator. 


Our construction then proceeds in a few steps:
\begin{itemize}
\item[(i)] {\bf net construction}: find a well-distributed set of points $\Gamma = \{y_k\}$ in $\Xsp$, having a granularity parameter $\delta$, the finest resolution of interest; 
\item[(ii)] {\bf learning the atlas}: learn local charts $C_k$ near $y_k$ obtained by mapping  $\Xsp$ locally to $d$-dimensional Euclidean domains, and learn transition maps for changing coordinates between a nearby pair of charts;
\item[(iii)] {\bf learning the simulator}: run $p=p(\delta,\tau)$ paths for time $\patht=\patht(\delta)$ from each $y_k$ and map them to the coordinate chart $C_k$. Use these low dimensional representations to estimate a simple simulator on each chart $C_k$.
\end{itemize}

\subsubsection{Net construction}
\label{e:netconstruction}
The first stage is to produce a $\delta$-net $\Gamma = \{y_k\}$, which is a set of points $\{y_k\}$ in $\Xsp$ such that no two points are closer than $\delta$, and every point in $\Xsp$ is at least $\delta$ close to some $y_k$. With abuse of notation, the range of $k$ will also be denoted by $\Gamma$, so we may also write the net as $\{y_k\}_{k\in\Gamma}$. 
We say that two points $y_k$ and $y_j$ are connected, or $k \sim j$, if $y_k$ and $y_j$ are within $2\delta$. 
We shall construct a reduced simulator at each node and the neighboring connections will determine the switching between simulators at adjacent nodes.
See section \ref{sec_dnet} for the details.

In real world examples, the space $\Xsp$ may be unknown. 
In this case we assume that we have the ability to sample from $\Xsp$, and we generate enough samples $\{x_i\} \subset \Xsp$ such that balls of radius $r\ll\delta$ cover $\Xsp$. 
This first round of sampling should ideally have the following properties: it can be generated by a fast exploration method (e.g. see the recent work \cite{zheng2013rapid} for molecular dynamics, and references therein - this problem by itself is  subject of much research); its samples do not require a significant number of calls to the simulator, or long runs of the simulators; different points may be sampled independently so that the process may be parallelized.
We can then downsample these $\{x_i\}$ to obtain the desired net $\Gamma$.
These considerations depend on the sampling measure. If this is simply the canonical volume measure on $\mathcal{M}$, it is easy to see that $O(\delta^{-d}\log(1/\delta))$ samples suffice, with high-probability, to obtain the desired $\delta$-net. Similarly for measures that have uniformly lower and upper bounded density with respect to the volume measure. We remark that this sampling is therefore independent of the dynamics of interest, has in general nothing to do with the stationary measure of the process we will seek to approximate: it may therefore be significantly easier to construct this sampling mechanism rather than one adapted to the dynamics.

Finally, it is important to remark that the algorithm we present is easily modified to run in exploratory mode: the fast simulator runs on the currently explored region of space, and whenever configurations outside the explored region of space are encountered (an event that is quickly detectable using the data structures we employ), new charts and local simulators may be added on-the-fly with minimal computational work. This is subject of current work and will be detailed in a forthcoming publication.

\subsubsection{Learning charts and corresponding maps}

The first step in learning the charts is to generate a set of landmarks $A_k\subset\Xsp$ for each $y_k$ in the net $\Gamma\subset\Xsp$: a set of points well-spread in $d$ directions on $\Xsp$, at distance about $\delta$ from $y_k$. In our setting where $\M$ is $d$-dimensional, we can sample $p\geq d$ paths from the simulator of $Y_t$, starting at $y_k$ and run until time $\patht=\patht(\delta)$. As long as the diffusion $Y_t$ is nondegenerate on the tangent plane, the projections of the end points of these paths will span the tangent space to $\Xsp$ at $y_k$.

Next we learn a mapping $\MDS_k$ from a neighborhood of $y_k$ to a coordinate chart $C_k\subseteq\mathbb{R}^d$ for each $y_k$. In order that neighboring coordinate charts overlap on a region of size $\delta$,  we learn $\MDS_k$ from $L_k = \bigcup_{j\sim k} A_j$, the union of neighboring landmarks. In this way $\MDS_k$ will be defined as a map from $B_{2\delta}(y_k)\subseteq\Xsp$ to $C_k\subseteq\mathbb{R}^d$. The overlap between neighboring charts will allow us to smoothly transition the simulator from one chart to the next. Each mapping $\MDS_k$ is constructed using Landmark Multi-Dimensional Scaling (LMDS) on $L_k$, minimizing distortion of pairwise distances between the landmarks $L_k$ (see section \ref{sec_LMDS}).

For any $k\sim j$, $L_k$ and $L_j$ have the landmarks $A_k \cup A_j$ in common; thus the charts $C_k$ and $C_j$ overlap on $A_k \cup A_j$. These landmarks span the local charts, and are the points used to learn the transition maps between neighboring charts. The affine transition map $S_{k,j}$ is chosen as the ``best'' linear mapping from $\MDS_k(A_k\cup A_j)$ to $\MDS_j(A_k\cup A_j)$ described in section \ref{sec_pinv}. Figure \ref{fig_charts} shows a cartoon version of the points used to learn the atlas.

\begin{figure}[ht]
\begin{center}
\includegraphics[width=4in]{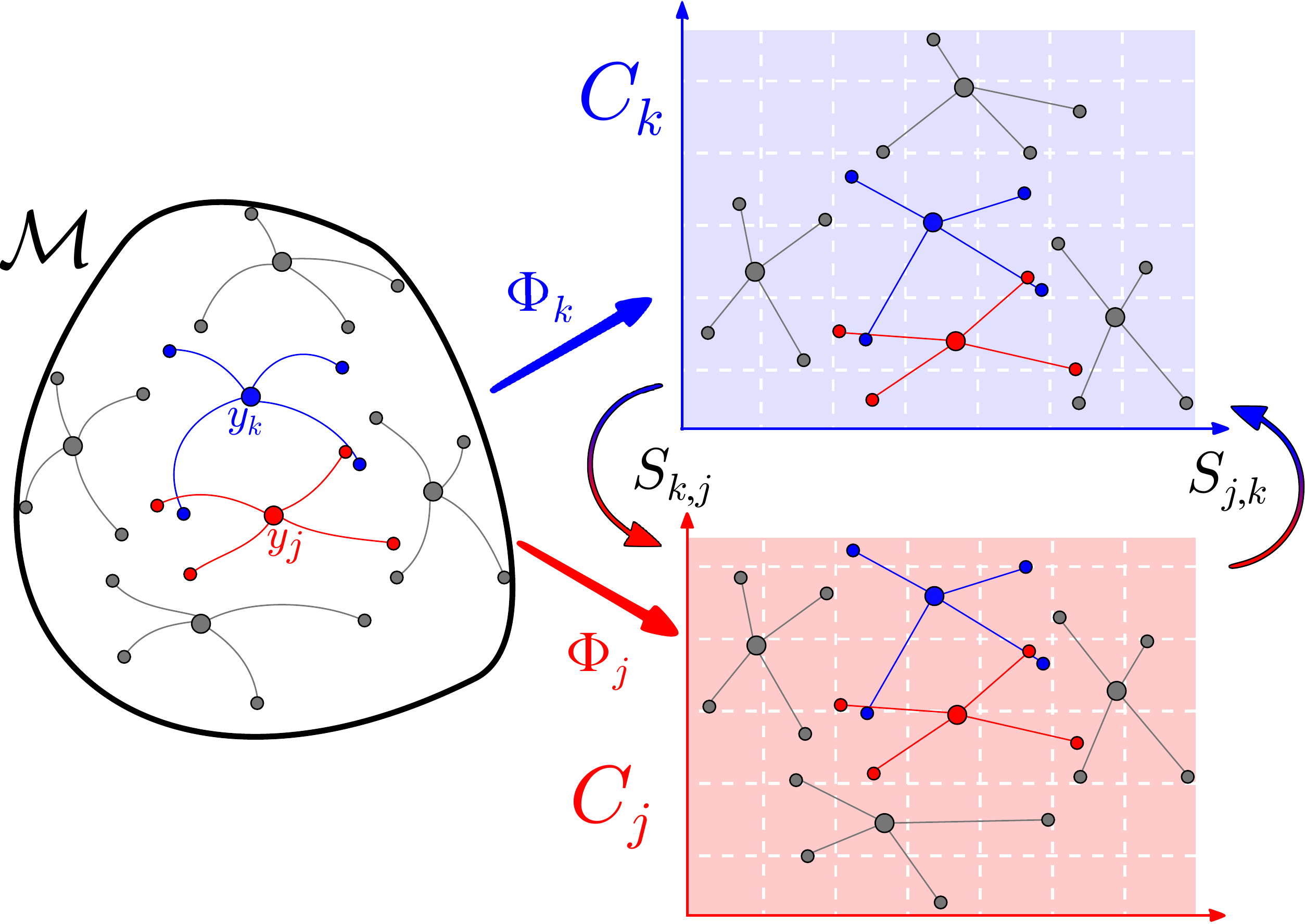}
\caption[This figure depicts the landmarks used to learn the overlapping charts.]{This figure depicts $m=4$ samples per net point being used to learn the charts. Large circles represent net points (or projections of net points) and small circles represent path end points (or projections of the path endpoints). The LMDS mappings $\MDS_k,\MDS_j$ use all the circles to learn the chart, while the transition maps $S_{k,j},S_{j,k}$ use only the colored circles.}\label{fig_charts}
\end{center}
\end{figure}  

\subsubsection{Learning the Simulator}
Once the charts are known, we learn an approximation to the simulator on each chart. For each $y_k\in\Gamma$, we run $p=p(\delta,\tau)$ paths via the original simulator $\osim$ up to time $\patht=\patht(\delta)$ starting from $y_k$. Next we project the samples to $C_k$ in order to estimate local simulation parameters.
In this paper we use constant coefficient SDEs to model the simulator on each chart:
\begin{align}
 d\wb X_t = \wb b_k dt + \wb \sigma_k dB_t\,,
\end{align}
for some $\wb b_k\in\R^d$ and some positive definite $\wb \sigma_k\in\R^{d\times d}$.
The solution to this constant coefficient SDE is a Gaussian with mean $\wb b_k\,\patht$ and covariance $\wb \sigma_k \wb \sigma_k ^T \, \patht$. 
Therefore, we estimate $\wb b_k$ and $\wb \sigma_k$ by imposing that these statistics match the sample mean and sample covariance of the endpoints of the $p$ paths run with $\mathcal{S}$. Finite sample bounds for these empirical values determine how large $p$ should be in order to achieve a desired accuracy ($\delta$) with the requested confidence ($\tau$).

This step is trivially parallelizable, both in $k$ (the chart in which the learning takes place) and within each chart (each of the $p$ paths may be run independently). At the end of this process we have obtained the family of parameters $(\wb b_k,\wb \sigma_k)_{k\in\Gamma}$ for a family of simulators $(\lsim_k)_{k\in\Gamma}$. 

The local simulators $(\lsim_k)_{k\in\Gamma}$ are extended to a global simulator $\lsim$ on $\csp$ using the transition maps between charts. This is done by alternating between steps from $(\lsim_k)_{k\in\Gamma}$, and transition map operations -- this is somewhat delicate, and detailed in section \ref{sec_lsim}. 

The choice of the local SDE's and the estimator of its parameters is one of the simplest possible, however we will see a collection of these simple simulators combine to reproduce much more complicated systems. Naturally the ideas may be extended to richer families of local SDE's, for which appropriate estimators based on the statistics of local trajectories may be constructed: this is subject of current research.

\subsection{Theoretical guarantees}
We present here a simplified version of the main result, Theorem \ref{thm_main}. 
Suppose the given stochastic dynamical system $Y_t$ is driven by an SDE on a $d$-dimensional manifold $\M$ with volume measure $\mu$ of the form
\begin{align}
 dY_t = b(Y_t)dt + \sigma(Y_t)dB_t \label{eqn_ygeneral}
\end{align}
with $b, \sigma$ Lipschitz functions, and $\sigma$ uniformly nondegenerate on the tangent bundle $T(\M)$. Let $q$ be the density of the stationary distribution of $Y_t$ on $\M$, and $\hat{q}$ be the density of a probability measure on $\csp$ defined later in equation \eqref{eqn_qhatdef} and computed by running the \ls\ for large time. Let $G$ be the inverse mapping from $\csp$ to $\M$ defined in section \ref{sec_prelim}. 
If the number of sample paths collected at each of $\Ord(\delta^{-d})$ starting points is at least $\Ord((d+\tau^2)/\delta^4)$, then with probability at least $1-2e^{-\tau^2}$
\begin{align}
 ||q - G_*\wh q||_{L^1(\M)} < c \delta \ln(1/\delta)
 \label{e:mainsimple}
\end{align}
for some constant $c$ depending on geometric properties of $\M$, the Lipschitz constants of the drift $b$ and diffusion $\sigma$, and the lower bound on singular values of $\sigma$ along the tangent plane.
Here $G_*\wh q$ is the pushforward of the measure $\wh q$ from $\csp$ to $\M$ as defined in equation \ref{eqn_gstar}.
\vspace{0.1in}

One can think of $Y_t$ as the underlying homogenized system which we are trying to learn. This result guarantees that the ATLAS process learned only from short paths of $Y_t$, actually behaves closely to $Y_t$ for large times.
Note that if the microscale simulator does not satisfy these conditions, it is possible the system is well-approximated by a macroscale simulator of the form \eqref{eqn_ygeneral} satisfying the conditions of the theorem on the timescale $\patht$ (the time sample paths are run) and above; in this case the error in approximating the original simulator by $Y_t$ is simply added to the right hand side of \eqref{e:mainsimple}.

Our results could be re-interpreted in the context of adaptive MCMC as follows. Assume we wish to sample from a probability distribution $q$ on a $d$-dimensional manifold $\Xsp$, that is the stationary distribution of a process $Y_t$ as above (with the assumptions stated in Theorem \ref{thm_main}), then if we have access to a local simulator of $Y_t$, we can construct an efficient sampler for an approximation to $q$. Contrary to Riemann manifold Hamiltonian Monte Carlo \cite{GirolamiCalderheadRiemannManifoldLangevin} we do not need to know the parameter space of the underlying statistical model, which would correspond to a parametrization of $\Xsp$, but we learn it through many short simulations of the dynamics, nor do we need sophisticated numerical integrators. These ideas are being developed further in a forthcoming publication.

\subsection{Examples}
Here we present some examples showcasing the usefulness of the \ls. The examples shown here have Brownian motion in a potential well, although the theorem guarantees accuracy for any simulator of the form \eqref{eqn_ygeneral}. Further examples will be discussed in section \ref{sec_exs}.

\subsubsection{Brownian Motion on a Manifold}

Given a $d$-dimensional smooth compact manifold $\mathcal{M}$, one may construct the potential
\begin{equation*}
U_{\varepsilon}(x) = \frac 1{\varepsilon}\mathrm{dist}(x,\mathcal M)^2
\end{equation*}
and consider the It\^o diffusion in $\mathbb{R}^D$ given by
\begin{equation}
 dY_t = -\nabla U_{\varepsilon} dt + dB_t \label{eqn_yeps}
\end{equation}

If one simulates \eqref{eqn_yeps} numerically for small $\varepsilon$, the timesteps must be at least as small as $\Ord(\varepsilon)$ (using Euler-Maruyama, since we only need the weak convergence of the scheme). We can view this thin potential around the manifold as our microscale interactions which forces our choice of timestep. What we are interested in is the macroscale behavior determined by the manifold $\M$.

For $\varepsilon\rightarrow 0$ this process converges to the canonical Brownian motion on the manifold $\mathcal M$ \cite{Emery:StochasticCalculusManifolds}. For $\varepsilon$ sufficiently small (compared to the curvature of $\M$) $Y_t$ is well-approximated locally by (the low-dimensional) Brownian motion on $\mathcal M$, and the stationary distribution of $Y_t$ is close to that of Brownian motion on $\mathcal M$. Our results apply to this setting, yielding an efficient $d$-dimensional simulator for $Y_t$, without a priori knowledge of $\mathcal M$.

\subsubsection{One Dimensional Example}\label{ex_onedim}
In this numerical example, we start with Brownian motion in a simple double well, and add a high frequency term to the potential to get $U(x)$:
\begin{align}
 U(x) =  16x^2(x-1)^2 + \frac16\cos(A\pi x)
\end{align}
with $A=100$.
The high frequency term gives the Lipschitz constant $L\sim 10^2$, forcing the forward Euler scheme to use time steps on the order of $L^{-2}\sim10^{-4}$ in order to just achieve stability (using a higher order method would not solve these problems as higher derivatives of $U$ will be even larger; implicit schemes would allow for a larger step size, at the expense of larger computational complexity for each step). The first term in $U$ is much smoother, and homogenization theory (e.g. \cite{pavliotis2007parameter} and references therein) suggests that the system is well-approximated by a smoother system with Lipschitz constant $l\sim 10$ or less (determined by the quartic term in $U$), at least for $A$ going to $\infty$ (and suitable renormalization). Our construction yields an approximation of the given system, for $A$ fixed, at a spatial scale larger than $\delta$ (corresponding to a timescale larger than $\delta^2$). This is the target system we wish to approximate. Running \ls\ with $\delta = 0.1\sim l^{-1}$, we obtain a smoothed version of the potential homogenizing the high frequency term (see Figure \ref{fig_wells}).  

\begin{figure}[ht]
\begin{center}
\includegraphics[width=4.5in]{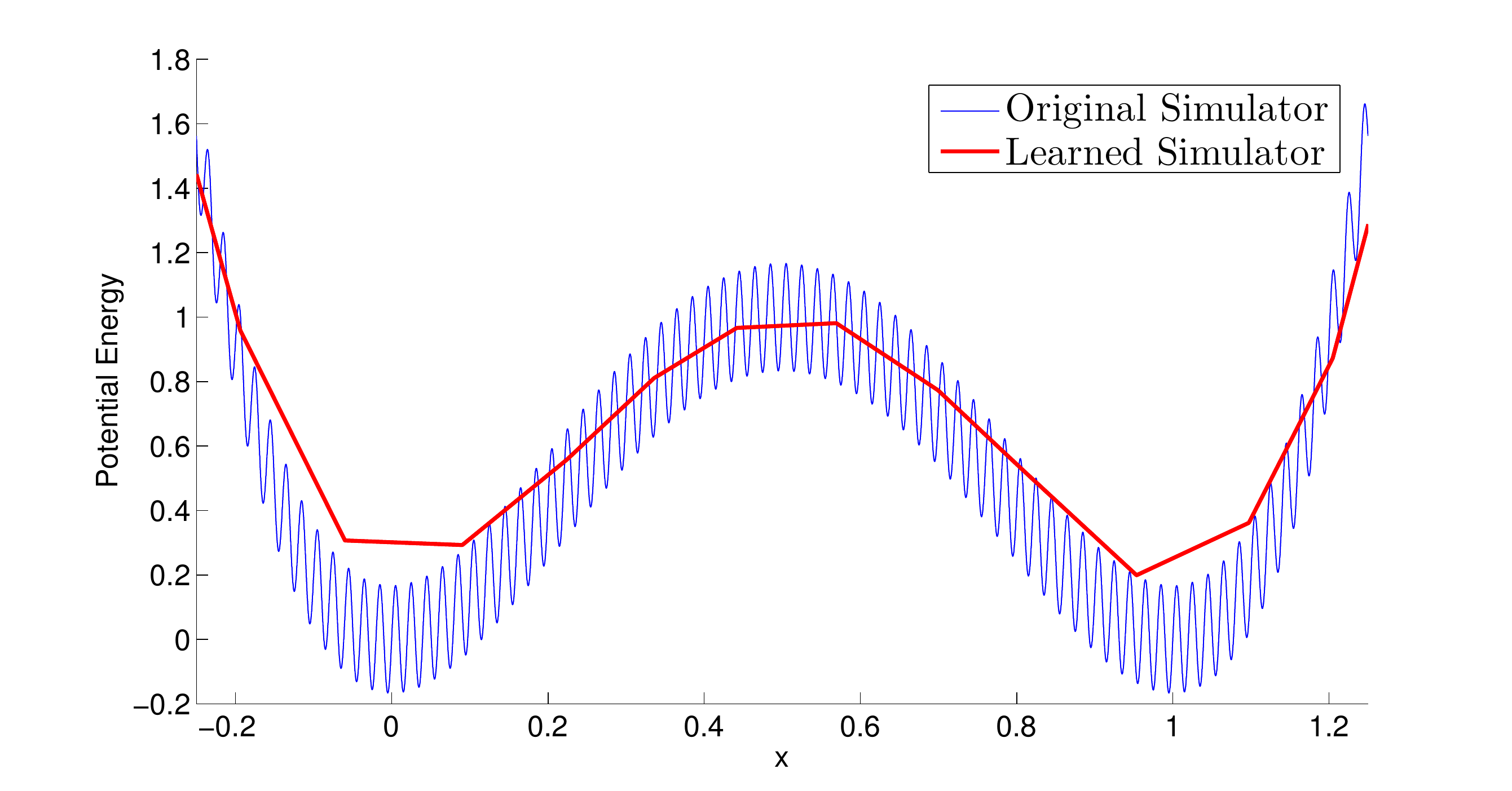}
\caption[Comparing stiff potential $U$, and effective potential $\wh U$ for the \ls.]{Original stiff potential $U$ (shown in blue), and effective potential $\wh U$ (shown in red) for the \ls, learned from short trajectories of the original simulator.}\label{fig_wells}
\end{center}
\end{figure}
The \ls\ takes time steps which are over $10^2$ times larger than the original system, and thus long paths can be simulated $10^2$ times faster, as if the system had a potential smooth at scale $O(1)$. Note that increasing the frequency of the oscillating term (thereby increasing $L$) does not affect the speed of the \ls, but only the speed of constructing the \ls\ (since $\S$ would have to use smaller time steps). This means that our algorithm allows for a decoupling of the microscale complexity (which is handled in parallel) from the macroscale complexity. A histogram of the approximate stationary distributions are shown in Figure \ref{fig_bar} comparing the \ls\ and the original system. See section \ref{ex_dwwr} for more details about the experiment, and Figures \ref{f:Verrorhf}, \ref{f:simerrorhf} for the errors in true effective potential vs. estimated effective potential, and the error in approximating the time evolution of the original system by the \ls\ for a multiscale choice of times, as well as the different in transition rates between the wells.

\begin{figure}[ht]
\begin{center}
\includegraphics[width=6in]{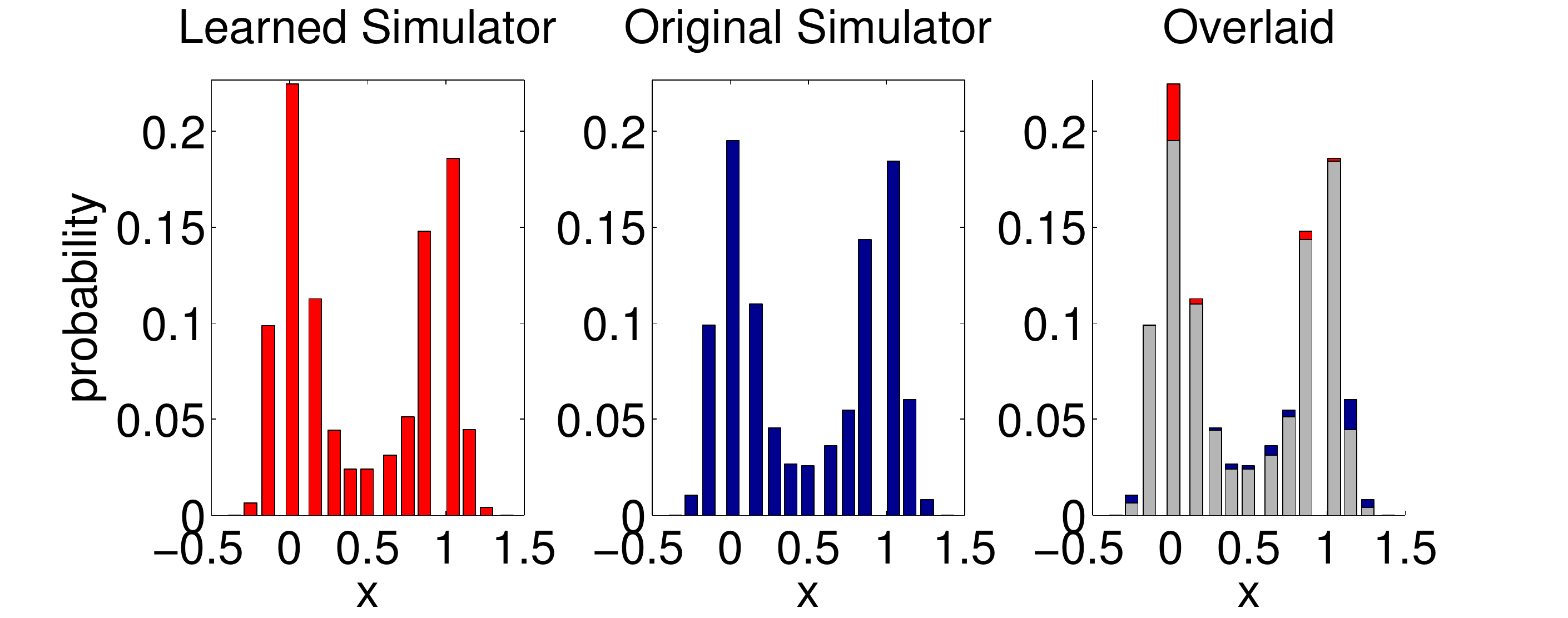}
\caption[Comparing original stationary distribution with \ls\ stationary distribution.]{Comparison between the (approximate) stationary distributions between the original simulator and the \ls\ with $10^5$ samples for example \ref{ex_onedim} (for this number of samples the error is expected to be $O(\sqrt{10^{-5}})\ll\delta$).}\label{fig_bar}
\end{center}
\end{figure}


\section{
ATLAS Algorithm
} \label{sec_alg}
In this section we present the algorithm in detail, since the main result will state properties of the output of the algorithm itself. Pseudo-code is presented in figure \ref{f:algomain}.
We start by discussing the algorithms used during the learning phase. We discuss afterwards the full details of the simulator learning phase and the simulation phase. 
The algorithm has access to a simulator $\mathcal{S}$ of the original system, and takes as input several parameters:
\begin{itemize}
 \item[$\cdot$] $\delta$: We will assume this parameter is given to us and represents the homogenization scale, and it will be related to the desired accuracy of the simulator via \eqref{e:mainbound} in Theorem \ref{thm_main}.
 \item[$\cdot$] $\{x_j\}$: A set of points on $\Xsp$ that is dense enough that a $\delta$-net for $\Xsp$ may be extracted from it. Alternatively, a way of sampling points on $\Xsp$ with respect to a measure comparable to volume measure on $\Xsp$.
 \item[$\cdot$] $\patht$: This represents the time short paths will be simulated for. In most examples in this paper, we choose $\patht = \delta^2$. In practice, one should choose $\patht$ so that sample paths at time $\patht$ are an expected distance $\delta$ from the starting location.
 \item[$\cdot$] $m$: The number of landmarks for each net point for learning the chart and transition maps. $m$ should be at least $d$, and we choose it of order $d$ to minimize sampling and computational complexity.
 \item[$\cdot$] $p$: The number of sample paths computed for each point in the net. $p$ should be $\Ord(\delta^{-4})$.
 \item[$\cdot$] $\Delta t$: time step of the \ls. It should be $\Ord(\delta/\ln(1/\delta))$. In the examples we used $\delta/5$. 
\end{itemize}
These choices of parameters are informed by the results and proofs in section \ref{sec_err}. We will see that for these choices of parameters, the \ls\ produces paths whose stationary distribution has error $\Ord(\delta \ln(1/\delta))$.

\subsection{Net construction} \label{sec_dnet}
In a metric space $(\Xsp,\rho)$ we define a $\delta$-net of points as follows:
\begin{definition}[$\delta$-net]
\label{def:net}
A $\delta$-net for a metric space $(\Xsp,\metric)$ is a set of points $\{y_k\}_{k\in\Gamma}$ such that $\forall k_1,k_2\in\Gamma \ \rho(y_{k_1},y_{k_2}) \geq \delta$, and at the same time  $\forall x\in\Xsp \, \exists k\in\Gamma \, \rho(x,y_{k}) \leq \delta$.
\end{definition}
In view of our purposes, the first property ensures that the net points are not too close together: this is essential so we do not waste time exploring regions of the space that we have already explored, do not construct many more local simulators than needed, and do not switch between charts significantly more often than necessary. The second property ensures that $\{B_\delta(y_k)\}_{k\in\Gamma}$ is a cover for $\Xsp$, guaranteeing that we explore the whole space $\Xsp$. We will connect nearby net points: if $d(y_{k_1},y_{k_2})\le 2\delta$ we say that $y_{k_1}$ and $y_{k_2}$ are neighbors and we write $k_1 \sim k_2$.

\subsubsection{Computational cost}
Algorithms for efficiently constructing $\delta$-nets in metric spaces satisfying a doubling condition exist and are non-trivial, for example by constructing a data structure called cover trees (see \cite{beygelzimer2006cover}), which run in $\Ord(C^d  n \log(n) D)$ time, where $d$ is the intrinsic dimension (e.g. doubling dimension), $C$ a constant that depends on the curvature of $\Xsp$, $n$ is the number of points in $\Xsp$, and $D$ is the cost of computing the distance between a pair of points in $\Xsp$. These data structures are especially useful for both finding near points to any given point, and for constructing nets of points at multiple resolutions, and they may be run in an online fashion.

A slower, simpler algorithm for constructing a net is to add points one a time if they are farther than $\delta$ from any point already in the net; finish when no more points can be added. For simplicity, this is the algorithm we have used in examples presented in this paper. 

If a way of sampling points from $\Xsp$ with respect to a measure $\mu$ with uniformly lower and upper bounded density with respect to the volume measure on $\Xsp$, then it is easy to see that $O(\delta^{-d}\log(1/\delta))$ samples suffice, with high-probability, to obtain enough samples from which a $\delta$-net may be extracted \cite{MMS:NoisyDictionaryLearning}.

\subsection{Dimension Reduction: Landmark Multidimensional Scaling (LMDS)} \label{sec_LMDS}
LMDS takes as input a set of landmarks $L\subset \Xsp$ and a set of other points $Z\subset \Xsp$, and constructs a map $\MDS:L\cup Z\rightarrow\R^d$ embedding $L,Z$ into $\R^d$. LMDS computes all pairwise distances between points in $L$, and returns low dimensional coordinates which minimize the distortion given by
\begin{align}
 \sum_{l_i,l_j\in L} \left(\metric(l_i,l_j)^2 - ||\MDS(l_i) - \MDS(l_j)||^2_{\R^d}\right)^2
\end{align}
over all possible mappings $\MDS$.  This is Multiscale Dimensional Scaling (MDS, \cite{Borg:MDS}).
At this point $\Phi$ is only defined on $L$: to extend it to $L\cup Z$, LMDS proceeds by computing the distances between each point in $L$ and each point in $Z$,  and for each point $z \in Z$ assigns coordinates $\MDS(z)$ which minimize
\begin{align}
 \sum_{l_i\in L} \left(\metric(l_i,z)^2 - ||\MDS(l_i) - \MDS(z)||^2_{\R^d}\right)^2
\end{align}
over all possible choices $d$ dimensional vectors $\MDS(z)$. 
For a full description of the algorithm, see \cite{de2004sparse}.  If the distance $\metric$ is Euclidean, the algorithm reduces to principal component analysis (PCA). 

If the dimension $d$ is unknown, one could learn $d$ at this stage from observing the eigenvalues of the squared distance matrix obtained during MDS. Eigenvalues which are of order $\delta^2$ correspond to directions along the manifold, and eigenvalues which are of order $\delta^4$ or lower correspond to curvature (or noise). Thus, one could learn $d$ by choosing a cutoff threshold depending upon $\delta$ (in fact this is done in example \ref{ex_fcn}). For an extensive analysis on how to detect intrinsic dimensionality of data sets see \cite{LMR:MGM1,CM:MGM2}, and its use in the context of high-dimensional stochastic systems in view of global nonlinear dimension reduction and reaction coordinates see \cite{RZMC:ReactionCoordinatesLocalScaling,ZRMC:PolymerReversal}.

\subsubsection{Computational cost}
The computational cost of this algorithm is $\Ord((|L|^2+|L|\cdot|Z|)D)$, where $D$ is the cost of evaluating $\metric$ at a pair of points. This cost comes directly from the number of distances computed. The point of LMDS compared to MDS (the case $|Z|=0$) is that in the case of interest where $|L|\ll|Z|$, the cost is linear in $|L\cup Z|$ for LMDS instead of quadratic as in MDS \cite{de2004sparse}. The cost of computing $d$ eigenvectors on a matrix of size $|L|\times|L|$ is $\Ord(d|L|^2)$, which is negligible compared to $|L|^2D$.

\subsection{Least-squares switching maps} \label{sec_pinv}
We will use the pseudoinverse (see \cite{penrose1956best}) to solve a least squares problem of finding the best linear transition map. If $X$ and $Y$ are each $2m\times d$ matrices (with chart images as rows, and with mean zero columns), then the matrix $T = X^\dagger Y$ minimizes $||XT - Y||_2$ over all $d\times d$ matrices. 

Fix $y_k\in\Gamma$. In the construction algorithm that follows, for each connection $j\sim k$, we take a set of common landmarks $L_{k,j} = A_k \cup A_j$ and let $X = \MDS_k(L_{k,j})$ and $Y = \MDS_j(L_{k,j})$, with the $d$-dimensional vectors involved being the rows of these matrices. Since the mean of the rows of $X$ (resp. $Y$) is not zero, we subtract from each row the mean $\mu_{k,j}$ (resp. $\mu_{j,k}$) of the rows of $X$ (resp. $Y$). The charts $C_k$ and $C_j$ represent overlapping areas on $\Xsp$, and so there will exist a matrix $T_{k,j}=X^\dagger Y$ which has small error. See again figure \ref{fig_charts} for a detailed picture.

To simplify notation, we will combine the mean shifting and matrix multiplication into a single operator $S_{k,j}$
\begin{align}
S_{k,j}(x) = (x - \mu_{k,j})T_{k,j} + \mu_{j,k} \label{eqn_switchingmap}
\end{align}

To decide when to apply the switching maps, we will need to know when the simulator is in the region between two charts. To do this, we will calculate the distances to the \emph{chart centers}, which we call $c_{k,j} = \Phi_k(y_j)$, i.e. the images of the net points via the low dimensional mappings.

\subsubsection{Computational cost}
The cost of computing the pseudoinverse is $\Ord(2md^2)$ since we must compute the singular value decomposition of $X$, and $2m > d$. The cost of applying the switching map is $\Ord(d^2)$.

\begin{figure}[tb!]
\centering
    \textbf{Learning Phase}\par
\begin{framed}
\begin{algorithmic}
 \item[] $\lsim = $ {\bf construction\_phase}($\{x_j\}, \metric,\osim$)
 \item[] \hspace{.1in} $\{y_k\}_{k \in \Gamma} \leftarrow \delta-$net($\{x_j\}$)
 \item[] \hspace{.1in} {\bf for} $k \in \Gamma$
 \item[] \hspace{.1in} \hspace{.3in} $\%$ \emph{create} $m+1$ \emph{landmarks for LMDS} around $y_k$
 \item[] \hspace{.1in} \hspace{.3in} $\{a_{k,l}\}_{l=1..m} = \osim(y_k,m,\patht)$
 \item[] \hspace{.1in} \hspace{.3in} $A_k = y_k \cup \{a_{k,l}\}$
 \item[] \hspace{.1in} {\bf end}
 \item[] \hspace{.1in} {\bf for} $k \in \Gamma$
 \item[] \hspace{.1in} \hspace{.3in} $\%$ \emph{simulate} $p$ \emph{paths for estimating drift and diffusion coefficients around $y_k$}
 \item[] \hspace{.1in} \hspace{.3in} $\{x_{k,l}\}_{l=1..p} = \osim(y_k,p,\patht)$
 \item[] \hspace{.1in} \hspace{.3in} $L_k = \bigcup_{i\sim k} A_i$
 \item[] \hspace{.1in} \hspace{.3in} $[L_k',\{x_{k,l}'\}_{l=1..p}] = $ LMDS($L_k,\{x_{k,l}\}_{l=1..p},\metric$)
 \item[] \hspace{.1in} \hspace{.3in} $\{\lsim.c_{k,j}\} \leftarrow \bigcup_{j \sim k} y_j$ in $L_k'$
 \item[] \hspace{.1in} \hspace{.3in} shift coordinates so $c_{k,k}=0$
 \item[] \hspace{.1in} \hspace{.3in} $\lsim.\wb{b}_k \leftarrow  \sum_l x_{k,l}'/p\patht $
 \item[] \hspace{.1in} \hspace{.3in} $\lsim.\wb{\sigma}_k \leftarrow $ (Cov($\{x_{k,l}'\})/\patht)^{1/2}$
 \item[] \hspace{.1in} \hspace{.3in} $\%$ \emph{compute switching maps}
 \item[] \hspace{.1in} \hspace{.3in} {\bf for} $j \sim k$, $j<k$
 \item[] \hspace{.1in} \hspace{.6in} $L_{k,j} = A_k \bigcup A_j$
 \item[] \hspace{.1in} \hspace{.6in} $L_{k,j}' = L_{k,j} $ in $L_k'$ coordinates
 \item[] \hspace{.1in} \hspace{.6in} $L_{j,k}' = L_{j,k} $ in $L_j'$ coordinates
 \item[] \hspace{.1in} \hspace{.6in} $\lsim.\mu_{k,j} \leftarrow \E [L_{k,j}']$
 \item[] \hspace{.1in} \hspace{.6in} $\lsim.\mu_{j,k} \leftarrow \E [L_{j,k}']$
 \item[] \hspace{.1in} \hspace{.6in} $\lsim.T_{k,j} \leftarrow (L_{k,j}' - \mu_{k,j})^\dagger (L_{j,k}' - \mu_{j,k})$
 \item[] \hspace{.1in} \hspace{.3in} {\bf end}
 \item[] \hspace{.1in} {\bf end}
\end{algorithmic}
\end{framed}
\caption[Pseudocode for constructing the \ls.]{Main algorithm for constructing the \ls: it constructs the $\delta$-net, computes chart embeddings, learns chart simulators from short sample paths, and transition maps.}
\label{f:algomain}
\end{figure}

\subsection{Learning Phase} \label{learning_phase}
The first part of the \ls\ algorithm is the learning phase, in which we use the sample paths to learn local chart coordinates, local simulators and transition maps. In this part of the algorithm, we store all the information necessary for the global simulator in $\lsim$, and describe in the next section what it means to run this simulator. We will use the notation $\lsim.var$ to denote the variable $var$ within the simulator $\lsim$. Recall that $\{x_j\}$ is a given, dense enough sample of $\Xsp$ to produce a net at scale $\delta$. Let $\osim(y,p,\patht)$ denote running $p$ paths of the simulator starting at $y$ for time $\patht$. We treat all points on the charts (resulting from LMDS) as row vectors. 
We compute, as described in the algorithm in Figure \ref{f:algomain}, for each chart $k\in\Gamma$, a drift $\lsim.\overline b_k$ and a diffusion coefficient $\lsim.\overline\sigma_k$,  effectively approximating the dynamics in the chart by that of a It\^o diffusion with constant drift and diffusion coefficient in the chart space (i.e. after mapping to Euclidean space using the LMDS map), estimated from the $p$ paths from $\mathcal{S}$ observed in the chart.

\subsubsection{Computational cost}
Each net point has order $2^d$ connections at most, by the properties of the $\delta$-net. Let $S$ denote the cost of running one simulation of length $\patht$. Thus, the computational cost of the construction phase, for each chart, is of order
\begin{equation}
\underbrace{mS}_{\text{landmark simulation}}+\underbrace{pS}_{\text{path simulation}}+\underbrace{2^dmpD}_{\text{LMDS}}=(m+p)S+2^dmpD
\end{equation}
We note that the term $2^dmpD$ can be decreased to $dpD$ since a $d$-dimensional plane may be estimated with only $O(d)$ points. Instead of using all $cm$ points as landmarks, one could choose a (e.g. random) subset of these landmarks for the initial embedding (although all these landmarks will be needed later for computing $T$).
All these steps are easily parallelized, so the per-chart cost above is also a per-processor cost if enough processors are available. Finally, observe that there at most $O(\delta^{-d})$ such charts (this follows from the property of the $\delta$-net, which ensures that balls of radius $\delta/2$ centered at net points are disjoint).

\subsection{The learned simulator} \label{sec_lsim}
The second part of our construction is to actually define the \ls, i.e. the reduced simulator of the system, given the parameters learned in the first stage. In other words we must describe what a single step of time $\Delta t$ looks like starting at a location $x$ in chart $i$. Figure \ref{f:learnedsimulator} contains the pseudocode for the algorithm implementing the strategy we now discuss (written assuming $x$ is a row vector).
Given the position at time $t$ is $x$ in chart $i$, the position at time $t+\Delta t$ is determined by first choosing $i'$ so that $\lsim.c_{i',j}$ is closest to $x$, among all possible $\{\lsim.c_{k,j}\}_k$. If $i'\neq i$, then the coordinates of $x$ (in chart $i$) are changed to coordinates in chart $i'$ by applying the switching map $T_{i,i'}$. Now that $x$ is in the coordinates of the chart $i'$, a forward stochastic Euler step is taken using the drift and diffusion coefficients in chart $i'$. Finally, this Euler step is confined to the local chart by applying a ``wall function'' $W$; details will be discuss in the next section.

\begin{figure}[hb]
\centering
    \textbf{Learned Simulator}\par
\begin{framed}
\begin{algorithmic}
 \item[] $(x,i') = $ {\bf simulator\_step}$(x,i,\lsim)$
 \item[] \hspace{.1in} $\%$ \emph{select new coordinate chart}
 \item[] \hspace{.1in} $i' = \text{argmin}_j||x - \lsim.c_{i,j}||^2_{\R^d}$
 \item[] \hspace{.1in} {\bf if} $i' \not = i$
 \item[] \hspace{.1in} \hspace{.3in} $x \leftarrow (x - \lsim.\mu_{i,i'})\lsim.T_{i,i'} + \lsim.\mu_{i',i}$
 \item[] \hspace{.1in} {\bf end}
 \item[]
 \item[] \hspace{.1in} $\%$ \emph{forward Euler step}
 \item[] \hspace{.1in} $\eta \sim \mathcal{N}(0,I_d)$
 \item[] \hspace{.1in} $x \leftarrow x + \lsim.\wb{b}_{i'} \Delta t + \eta \lsim.\wb{\sigma}_{i'}  \sqrt{\Delta t}$
 \item[]
 \item[] \hspace{.1in} $\%$ \emph{prevent escape from local chart}
 \item[] \hspace{.1in} {\bf if} $|x| > 3\delta/2$
 \item[] \hspace{.1in} \hspace{.3in} $x \leftarrow W(x):=\frac{x}{|x|}\left(2\delta - \frac{\delta}{2} \exp\left(3-\frac{2}{\delta}|x|\right)\right)$
 \item[] \hspace{.1in} {\bf end}
\end{algorithmic}
\end{framed}
\caption[Pseudocode for a single step of the \ls.]{Algorithm for running the \ls, by combining local diffusions and linear transition map between charts.}
\label{f:learnedsimulator}
\end{figure}

\clearpage

\subsubsection{Computational cost}
The \ls\ runs in $d$ dimensions, and does not require calls to the original simulator, so the running time now only depends on the local complexity of the homogenized problem. The number of simulation steps required to approach stationarity still depends upon the time it takes to converge to equilibrium, but so too did the original simulator. 

If $c \approx 2^d$ is the maximum number of connections each net point has, the computational cost of each time step of the \ls\ is of order
\begin{equation}
\underbrace{d2^d}_{\text{distance computation}}+\underbrace{d^2}_{\text{forward step}} = (2^d+d)d\,.
\end{equation}
%
%


\section{
Theoretical Results and Guarantees
} \label{sec_err}
In this section, we first introduce the minimum amount of material to precisely state the Theorem in section \ref{sec_thmstate}. Then we introduce the necessary mathematical objects to state the Lemmata used during the main proof in section \ref{sec_prelim}. In section \ref{sec_proofs}, we prove Theorem \ref{thm_main} and the Lemmata used.

\subsection{Theorem Statement}\label{sec_thmstate}

Let $\{y_k\}_{k\in\Gamma}$ denote the set of net points.
For each $k$ we have the mapping $\MDS_k$ from $\M$ to $\R^d$ given by MDS.
We will assume the projection to the tangent plane at $y_k$ is invertible on a ball of radius $2\delta$, which implies $\MDS_k$ will also be invertible on a ball of radius $2\delta$ (this is an intermediate step in the proof of Lemma \ref{lemm_Lb-Lt}).
Let $\csp = B_{2\delta}(0) \times \Gamma$, equipped with the transition maps $\{\MDS_j(\MDS_k^{-1})\}$, denote the atlas.
We shift coordinates in each chart so that $c_{k,k} = \MDS_k(y_k) = 0$ (where $c_{k,j}$ was defined in section \ref{sec_pinv}), so $\csp$ contains all points in each chart within $2\delta$ of $c_{k,k}$. 
Next we state some definitions in order to mathematically define a step of the \ls. If $i\sim j$, let $S_{i,j}(x) = (x - \mu_{i,j})T_{i,j} + \mu_{j,i}$ be the transition map between charts $i$ and $j$ defined in equation \ref{eqn_switchingmap}.

\begin{wrapfigure}{r}{5cm}
\includegraphics[width=5cm]{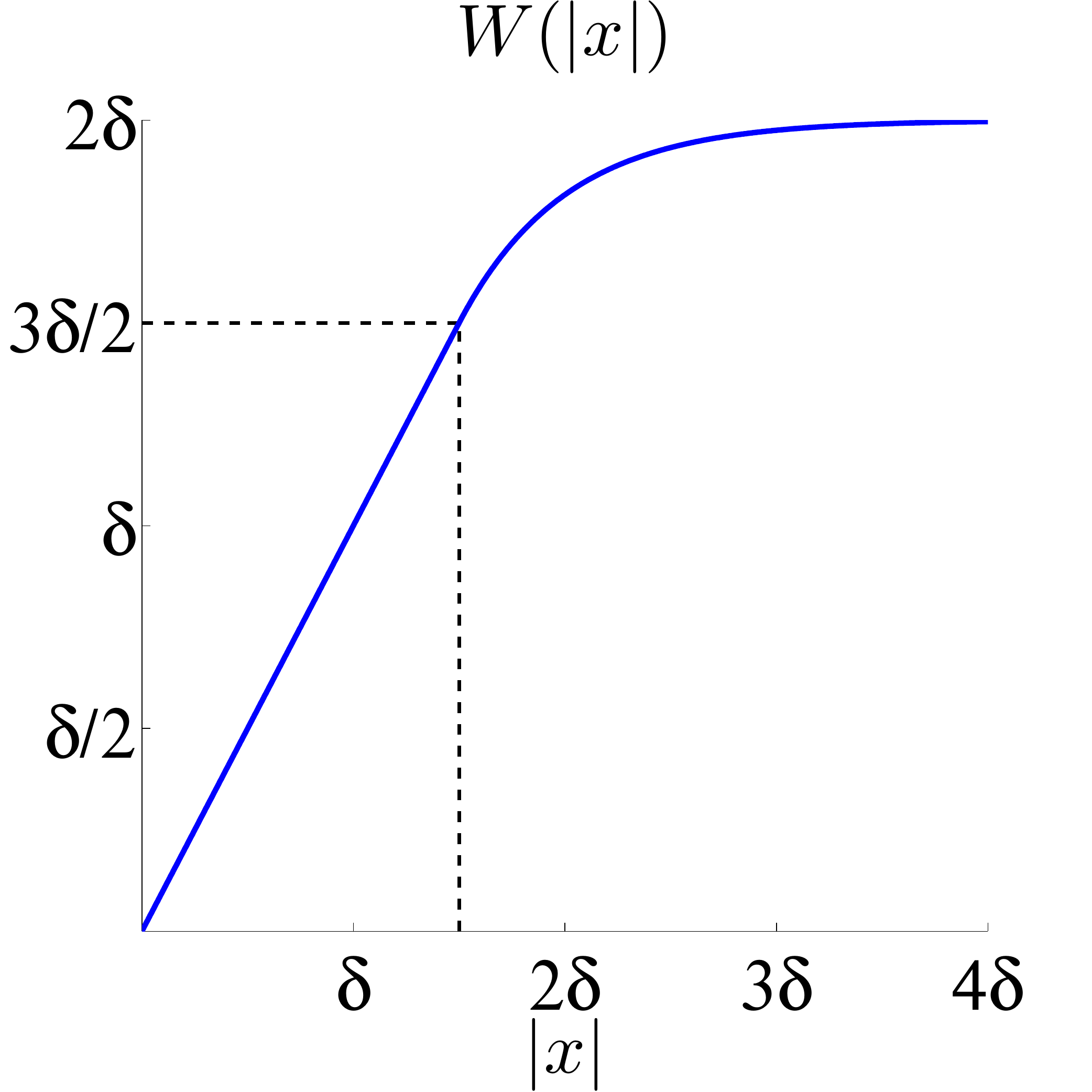}
\caption{The wall function $W$.}\label{fig_wall}
\end{wrapfigure}  

Let $W(x)$ be the wall function (which confines the simulator to $B_{2\delta}(0)$) defined by
\begin{equation}
   W(x) = \left\{
     \begin{array}{cr}
       x &  |x| \leq \frac{3\delta}{2}\\
 & \\
       \frac{2\delta x}{|x|} - \frac{\delta x}{2|x|} \exp\left(3-\frac{2}{\delta}|x|\right) &  |x| > \frac{3\delta}{2}
     \end{array}
   \right.
\label{e:W}   
\end{equation} 
There are other possible choices for $W$, but the main ingredients are: $W$ is $C^2$, invertible, equal to the identity on a ball of radius $3\delta/2$, and takes $\R^d \rightarrow B_{2\delta}(0)$.

Let $B_t$ be a standard Brownian motion in $\R^d$. The update rule for the \ls\ starting at $(x_0,i_0)$ is
\begin{align}
 i_{k+1} &= \text{argmin}_j\big\lvert x_k-c_{i_k,j}\big\rvert \, \label{eqn_ik} \\
 x_{k+1} &= W\big(S_{i_k,i_{k+1}}(x_k) + \wb b_{i_{k+1}} \Delta t  +  \wb \sigma_{i_{k+1}} B_{\Delta t}\big)
\end{align}

Define the \ls\ process $Z_k \in \csp$ starting at $z_0 = (x_0,i_0)$ to be $(x_k,i_k)$. See the algorithm in figure \ref{f:learnedsimulator}. We show in Lemma \ref{lemm_Zerg} that under the conditions of Theorem \ref{thm_main}, $Z_k$ has a unique stationary distribution $\mu$ on $\csp$.

We will often wish to refer to the chart index associated with the following time step, starting from an $(x,i) \in \csp$; we call this next chart index $i'(x,i)$, leaving off the $(x,i)$ where it is obvious which initial condition is being talked about:
\begin{align}
i'(x,i) = \text{argmin}_j\big\lvert x-c_{i,j}\big\rvert \label{eqn_ip}
\end{align}

Define for each $i$ the continuous time process $\wh X_t^x \in B_{2\delta}(0)$ by
\begin{align}
\wh X_t^x = W\big(S_{i,i'}(x) + \wb b_{i'} t  +  \wb \sigma_{i'} B_{t} \big) \label{eqn_xihat}
\end{align}

We will often be referring to initial conditions of the form $\icx=(x,i) \in \csp$, and for ease of notation we use $\wh X_t^\icx$ to keep track of which $\wh X$ process we are referring to, and what the starting location is. Keep in mind that $\wh X_t$ lives in the local tangent plane.

Define a measure $\wh q$ on $\csp$ for each measurable set $E \subset \csp$
\begin{align}
 \wh q(E) = \int_\csp \int_0^{\Delta t} \wh \P \big[ (\wh X_t^\icx,i') \in E \big] dt\, \mathrm{d}\mu(\icx) \,,
 \label{eqn_qhatdef}
\end{align}
where $\mu$ is the stationary distribution of $Z_k$ on $\Xsp$.
By equation \ref{eqn_xihat}, one can see that $\wh q$ is absolutely continuous w.r.t. Lebesgue measure (defined by the Lebesgue measures on each chart), and we will denote its Radon-Nikodym derivative, with some abuse of notation, $\wh q(x)$. We show in Theorem \ref{thm_main} that $\wh q$ is close to $q$, the stationary distribution of our original process $Y_t$.

Samples from $\wh q$ may be generated by the \ls\ by running $N$ steps of size $\Delta t$, and then one step of size $\delta t \sim \text{ Unif}(0,\Delta t)$.
Samples of $\wh X_{\delta t}$ are distributed according to $\int_0^{\delta t} \wh P_{t} du$, as seen by integrating over the joint density, $\wh P_u du$.
Since $Z$ is ergodic by lemma \ref{lemm_Zerg}, $Z_N \rightarrow \mu$, thus there is an $N$ for which our sampling plan ($N$ steps plus a $\delta t$ step) will approximate $\wh q$ sufficiently well.
Choosing $N$ is not easy and depends on the problem, although this is a difficulty with the original simulator as well; in practice, one should choose $N$ large enough that simulations reach a large fraction of the charts in the simulator.
The reason for the random final time step $\delta t$ (and the integral in equation \ref{eqn_qhatdef}) is to average over positions visited in between applications of the switching map.
 
Lastly, define a mapping back to the original space $G:\csp \rightarrow \Xsp$ by $G(x,i) = \MDS_i^{-1}(x)$ for each $(x,i) \in \csp$. Define the push forward map $G_*$ on $\wh q$ by
\begin{align}
G_* \wh q(y) = \sum_{x:G(x)=y} \wh q(x) dx \label{eqn_gstar}
\end{align}
The set $\{x:G(x)=y\}$ is finite since the $\MDS_i$'s are one-to-one. This mapping $G_*$ will allow us to compare $q$ and $\wh q$ as densities on $\Xsp$.

\begin{theorem} \label{thm_main}
Let $(Y_t)_{t\ge0}$ be an It\^o diffusion on a smooth compact connected $d$-dimensional manifold $\M$ with no boundary:
\begin{align}
dY_t = b(Y_t)dt + \sigma(Y_t)dB_t \label{eqn_Ydef}\,,
\end{align}
with $b,\sigma$ Lipschitz, and $\sigma$ uniformly elliptic, i.e. there exists $\lambda>0$ such that for all $x\in\M$ and $v\in T_x(\M)$ $v^T\sigma(x)\sigma(x)^Tv\ge\lambda|v|^2$.
Let $\delta$ be small enough so that for every $x$ the orthogonal projection $\M\rightarrow T_x(\M)$ is invertible on a ball of radius $2\delta$ on $T_x(\M)$.
Let $q$ be density (with respect to volume measure on $\M$) of the unique stationary distribution of $Y_t$. Let $\wh q$ be the density on $\csp$ generated by the \ls, as defined above in \eqref{eqn_qhatdef}. There exists constants $c_1,c_2$ such that if the number of sample paths satisfies $p > c_1(\tau,\M)/\delta^4$ then with probability at least $1-2\exp(-\tau^2)$,
 \begin{align}
  ||q - G_*\wh q||_{L^1(\M)} \leq c_2 \delta \ln(1/\delta)
  \label{e:mainbound}
 \end{align}
\end{theorem}
Some remarks are in order:
\begin{itemize}
\item[(i)] Lipschitz coefficients guarantee existence and uniqueness of solutions to \eqref{eqn_Ydef}; strong ellipticity (together with smoothness and connectivity of $\M$) guarantees that the process $Y_t$ has a unique stationary distribution with smooth density (this latter assumption may be weakened to include hypo-elliptic systems).
\item[(ii)] From a computational perspective, we note that as $G$ appears in the statement of the Theorem, it would be very useful to compute this mapping $G$. This is in general hard in arbitrary metric spaces, although one can always use the simple approximation $\wh G(x,i) = y_i$ which approximates $G$ at scale $\delta$.
\item[(iii)] The dimension of the state space of the system does not appear in the Theorem - only the intrinsic dimension of $\Xsp$ crucially affects the sampling requirements of the construction.
\item[(iv)] The ATLAS simulator is random, since it depends on the random paths collected to construct ATLAS, and so $\wh q$ is random. The result states that with high probability $\wh q$ is close to $q$ in the sense of \eqref{e:mainbound}.
\item[(v)] In practice we are interested in applying our construction and the Theorem to processes $Y_t$ that do not satisfy the assumptions above. This is possible when it is the case that above a certain timescale $t_0$, $Y_t$ may be well approximated by a process $\wt Y_t$, representing $Y_t$ {\em{homogenized at time scale $\patht$}}, which does satisfy the assumptions of the Theorem. In algorithmic terms, the microscale simulator (for $Y_t$) may not satisfy the conditions imposed above, but it may be well-approximated at timescales larger than $\patht$ by a simulator (for $\wt Y_t$) of the form \eqref{eqn_ygeneral} satisfying the conditions of the theorem on the timescale $\patht$. Our algorithm may then be applied to such a microscale simulator since it operates at time scales greater $\patht$ (its input is a set of paths of length $\patht$), and cannot therefore tell the difference between a microscale simulator for $Y_t$ and one for its homogenization $\wt Y_t$. In this case the error in approximating the original process $Y_t$ by $\wt Y_t$ may be added to the right hand side of \eqref{e:mainbound}.
\end{itemize}

We now turn to the proof of these results, after which we present many examples in section \ref{sec_exs}.

\subsection{Algorithmic Complexity}
Suppose that a single call to the original simulator of length $\patht$ costs $S$. The total number of points in the net $\Gamma$ contains $\Ord(\delta^{-d})$ points with constant depending on the volume of the manifold $\M$. From each point we will see that we must choose $p=\Ord(\delta^{-4})$ to estimate the parameters of the \ls\ simulator to within accuracy $\delta$. Assuming the expensive part of the construction algorithm is running the simulations, the total cost of construction is $\Ord(S\delta^{-d-4})$.

During each call to the \ls, we will see from the proof that we must choose our timestep $\Delta t = \patht / \log(1/\delta)$. The $\log(1/\delta)$ term may be neglected, and we will do so in this discussion. At each call to the simulator, we must compute the distance to each neighbor, of which there are $\Ord(2^d)$. Assuming each distance costs $\Ord(d)$ flops, the total cost of running the \ls\ for time $\patht$ is $\Ord(d2^d)$. 

Comparing the running time of the original simulator $S$ with the \ls\ amounts to comparing the cost of $S$ to $d2^d$. The benefit of the \ls\ over the original one then clearly depends upon how expensive the original simulator was, which can depend on many factors: length of the timestep, cost of evaluating functions, ambient dimension, etc. One thing is clear - the cost of $S$ depends on microscale properties, while $d2^d$ does not. Since the cost $S$ varies for each problem, in our examples in section \ref{sec_exs} we compare the \ls\ cost to the original simulator cost by comparing the size of the timestep alone. 

The running time is the main benefit of the \ls, but not the only one. Another advantage is that long paths of the simulator can be stored using only $d$ dimensions rather than the ambient dimension where $\Xsp$ lives. Yet another benefit is that some postprocessing has already been done, for example consider the question ``how long does $Y_t$ spend near the state $Y^*$?''. After running the original simulator one would have to compute a distance to $Y^*$ for many data points. After running the \ls\ answering this question requires only computing distances from the few net points near $Y^*$ to $Y^*$ to obtain the result with accuracy $2\delta$, and the \ls\ already knows which parts of the paths are in charts near $Y^*$. Finally, ATLAS organizes the effective state space in terms of simple local low-dimensional models, enabling novel, interactive visualizations of the dynamical system, that can run on any (portable) device with a web browser, thanks to the compression achieved by the \ls: this is subject of ongoing work.

\subsection{Preliminaries for Proofs} \label{sec_prelim}
Before diving into the proof of Theorem \ref{thm_main}, we need several definitions. In figure \ref{sec_notdef} we list a set of pointers to various definitions and notations used throughout.
Let $X_t = \MDS_i(Y_t) \in B_{2\delta}(0)$.
Let $\tau$ be the time when $X_t$ first hits the boundary of $B_{2\delta}(0)$.
Let $\MDS_{i,k}$ denote the $k^{th}$ coordinate of $\MDS_i$. Given that $Y_t$ satisfies \eqref{eqn_Ydef}, a straightforward application of It\^o's formula shows that $X_t=\MDS_i(Y_t)$  solves, for $t \leq \tau$, the It\^o SDE
\begin{align}
& dX_t = b_i(X_t)dt + \sigma_i(X_t)dB_t \label{eqn_Xdef} \\
& (b_i)_k(x) := \nabla \MDS_{i,k}(\MDS_i^{-1}(x)) \cdot b(\MDS_i^{-1}(x))  + \frac{1}{2} \sum_{j,l} \frac{\partial \MDS_{i,k}}{\partial x_j \partial x_l}(\MDS_i^{-1}(x))(\sigma \sigma^T)_{j,l}(\MDS_i^{-1}(x)) \label{eqn_bidef} \\
&(\sigma_i)_{k,j}(x) := \sum_l \frac{\partial \MDS_{i,k}}{\partial x_l}(\MDS_i^{-1}(x))  \sigma_{l,j}(\MDS_i^{-1}(x)) \label{eqn_sidef}
\end{align}

Let $\L$ be the generator (see \cite{oksendal}) for $Y_t$ on $\Xsp$, and, for $\icx=(x,i)\in\csp$, $\L_\icx$ be the generator of $X_t^\icx$. $\L$ is uniformly elliptic on $\Xsp$, since $Y_t$ is uniformly elliptic on $\M$.

We will be comparing a number of simulators to bridge the gap between the simulation scheme $\lsim$ and the true simulator $\osim$. We will do this by introducing intermediate processes in the local charts, and we will keep track of which chart the ATLAS process $Z_k$ is in separately. For this reason we consider the following processes on $B_{2\delta}(0)$, started at $\icx = (x,i)$:
 \begin{align}
  \wb X^\icx_t &= \MDS_{i'}\left(\MDS_{i}^{-1}\left(x\right)\right) + \wb b_{i'} t + \wb \sigma_{i'} B_t \label{eqn_Xbdef}\\
 \wt X^\icx_t &= S_{i,i'}(x) + \wb b_{i'} t + \wb \sigma_{i'} B_t \label{eqn_Xtdef} \\
 \wh X^\icx_t &= W\big(\wt X^\icx_t\big) \label{eqn_Xhdef}
 \end{align}
where $i'$ is defined  as in \eqref{eqn_ip}. The processes $\wb X_t, \wt X_t$ are natural stepping stones from $X_t$ to $\wh X_t$: $\wb X_t$ is the process which differs from $X_t$ locally only in that it uses the learned drift and diffusion coefficients. $\wt X_t$ differs from $\wb X_t$ only in that it uses the learned transition map $S_{i,i'}$ rather than the true transition map $\MDS_{i'}\circ \MDS_{i}^{-1}$. Finally $\wh X_t$ differs from $\wt X_t$ only because of the application of the local wall function $W$. Given any initial condition $\icx = (x,i)$, the three processes $\wb X^\icx_t,\wt X^\icx_t,\wh X^\icx_t$ are solutions of SDE's with generators $\wb \L_\icx, \wt \L_\icx, \wh \L_\icx$, respectively, on chart $i'$. These generators clearly depend on the chart $i'$, but as we will see in Lemma \ref{lemm_Lb-Lt}, they will also depend on $x$, and we keep track of this by putting $\icx$ as a subscript on the generators. We will prove that these generators are close to $\L_\icx$ for all $\icx \in \csp$. Then we will show, using ideas from \cite{jonm_conv}, that this is enough to imply bounds on the stationary distributions. 

\subsection{Proofs} \label{sec_proofs}

Before we begin the proof of Theorem \ref{thm_main}, we state some Lemmata which we will prove later, in order to keep the details until the end, while first showing the main ideas of the proof.

\begin{lemma} \label{lemm_Zerg}
 The process $Z_k$ is ergodic with stationary distribution $\mu$.
\end{lemma}

\begin{lemma} \label{lemm_ltc}
For any smooth test function $F:\csp\times B_{2\delta}(0) \rightarrow \R$ and initial condition $Z_0$,
\begin{align}
 \frac{1}{n}\sum_{k=1}^n \int_0^{\Delta t} F\big(Z_k,\wh{X}_t^{Z_k}\big)dt \rightarrow  \E_\mu\left[\E\left[\dsst\int_0^{\Delta t} F \big(z,\wh{X}^z_t\big) dt \right]\right]
\end{align}
a.s. as $n \rightarrow \infty$. Here $\E_\mu$ means taking the expectation over the initial condition $z \sim \mu$, and $\E$ is taking the expectation over the transition probabilities of $\wh X^z_t$.
\end{lemma}
The randomness in the statements of the following three Lemmata is that of the paths collected to learn the drift and diffusion coefficients used by $\wb X_t$ defined in \eqref{eqn_Xbdef}.

\begin{lemma} \label{lemm_L-Lb}
There exists a constant $C$ such that for any smooth test function $f:B_{2\delta}(0)\rightarrow\R$, and initial condition $z \in \csp$, with probability at least $1-4\exp(-\tau^2)$
\begin{align} 
\E \left[ \frac{1}{\Delta t}\dsst\int_0^{\Delta t} \big(\L_z- \wb{\L}_z\big)f\big(\wh X^z_t\big) dt \right] \leq C \delta \tau \sqrt{\ln(1/\delta)}||f||_{C^2} \label{eqn_L-Lb}
\end{align}
\end{lemma}

\begin{lemma} \label{lemm_Lb-Lt}
There exists a constant $C$ such that for any smooth test function $f$, and initial condition $z \in \csp$, with probability at least $1-4\exp(-\tau^2)$
\begin{align} 
\E \left[ \frac{1}{\Delta t}\dsst\int_0^{\Delta t} \big(\wb\L_z- \wt{\L}_z\big)f\big(\wh{X}^z_t\big) dt \right] \leq C \delta \ln(1/\delta) ||f||_{C^2} \label{eqn_Lb-Lt}
\end{align}
\end{lemma}

\begin{lemma} \label{lemm_Lt-Lh}
There exists a constant $C$ such that for any smooth test function $f$, and initial condition $z \in \csp$, with probability at least $1-4\exp(-\tau^2)$
\begin{align} 
\E \left[ \frac{1}{\Delta t}\dsst\int_0^{\Delta t} \big(\wh{\L}_z- \wt{\L}_z\big)f\big(\wh{X}^z_t\big) dt \right] \leq C \delta ||f||_{C^2} \label{eqn_Lt-Lh}
\end{align}
\end{lemma}

\begin{proof}[Proof of Theorem \ref{thm_main}]
This proof follows the ideas and techniques from \cite{jonm_conv} for proving large time convergence of numerical schemes, and we refer the reader to this reference for an overview of the previous substantial work in the area. Here we adapt part of the arguments contained in that paper to our setting. By assumption, the operator $\L$ is uniformly elliptic on $\Xsp$ (for more details on how to define these operators on manifolds see \cite{stroock}, \cite{hsu}). Let $\phi:\Xsp \rightarrow \R$ be a smooth test function on $\Xsp$ and define the average $\wb \phi$ by
\begin{align}
 \wb \phi = \int_\Xsp \phi(y)q(y)dy
\end{align}
By construction $(\phi - \wb \phi) \perp $ Null($\L^*$), and by the Fredh\"olm alternative there exists a unique solution $\psi$ to the Poisson equation $\L\psi = \phi - \wb{\phi}$. 
Uniform ellipticity of $\L$ implies, via standard estimates \cite{krylov}, $||\psi||_{C^2} \le C_{\Xsp,\lambda} ||\phi||_{\infty}$.

For ease of notation, we will index everything by $k$, the step of the process $Z_k=(x_k,i_k)\in\csp$. 
Let $\wh \L_k = \wh \L_{Z_{k-1}}$, $\wh X_t^k = \wh X_t^{Z_{k-1}}$, and $\psi_k = \psi \circ \MDS_{i_k}^{-1}$. Also let $\left\{B_t^k\right\}_{k=1}^\infty$ denote independent Brownian motions. The function $\psi_k$ is smooth on $B_{2\delta}(0)$, so by It\^o's formula:
\begin{align}
 \psi_k\big(\wh{X}_{\Delta t}^k\big) - \psi_k\big(\wh{X}_{0}^k\big) = \int_0^{\Delta t} \wh{\L}_{k} \psi_k \big(\wh{X}_{t}^k\big)dt + \int_0^{\Delta t} \nabla \psi_k \big(\wh{X}_{t}^k\big) \wh{\sigma}_{i_k} dB^k_t \label{eqn_itopsi}
\end{align}
By It\^o's isometry, letting $||A(\cdot)||_{F,\infty}:=||\, ||A(x)||_F\, ||_{L^\infty(\M)}$,
\begin{align}
 \E\left[\left(\int_0^{\Delta t} \nabla \psi_k \big(\wh{X}_{t}^k\big) \wh{\sigma}_{i_k} dB^k_t\right)^2\right] \leq \Delta t ||\psi||^2_{C^1} ||\wh\sigma||_{F,\infty}^2 \label{eqn_mtbd}
\end{align}
Define the martingale $M_n$ by
\begin{align}
 M_n = \frac{1}{n \Delta t}\sum_{k=1}^{n} \int_0^{\Delta t}  \nabla \psi_k \big(\wh{X}_{t}^k\big) \wh{\sigma}_{i_k} dB^k_t
\end{align}
When calculating the variance of $M_n$, cross terms vanish by independence. Then from equation \eqref{eqn_mtbd} we obtain the bound
\begin{align}
 \E[M_n^2] \leq \frac{1}{n\Delta t} ||\psi||^2_{C^1} ||\wh\sigma||_{F,\infty}^2
\end{align}
which implies $M_n \rightarrow 0$ a.s. as $n\rightarrow \infty$ by the martingale convergence theorem. 
Summing equation \eqref{eqn_itopsi} and dividing by $n \Delta t$,
\begin{align}
 \frac{1}{n \Delta t} (\psi_k(Z_n) - \psi_k(Z_0)) = M_n + \frac{1}{n \Delta t} \sum_{k=1}^n \int_0^{\Delta t} \wh{\L}_{k} \psi_k \big(\wh{X}_{t}^k\big)dt
\end{align}
Since $\psi$ is bounded, $\psi_k(Z_n)/n \Delta t \rightarrow 0$. Taking $n \rightarrow \infty$ on both sides and using Lemma \ref{lemm_ltc},
\begin{align}
0 &= \E_\mu\left[\E\left[\frac{1}{\Delta t}\dsst\int_0^{\Delta t} \wh\L_z \psi_z \big(\wh{X}^z_t\big) dt \right]\right]\,,
\end{align}
where, if $z=(x,i)$, $\psi_z = \psi \circ \MDS_{i'}^{-1}$ with $i'$ defined as in equation \eqref{eqn_ip}. Using Lemma \ref{lemm_L-Lb}, \ref{lemm_Lb-Lt} and \ref{lemm_Lt-Lh}, we have with probability at least $1-4\exp(-\tau^2)$,
\begin{align}
 \E_\mu\left[\E\left[\frac{1}{\Delta t}\dsst\int_0^{\Delta t} \L_z \psi_z \big(\wh{X}^z_t\big) dt \right]\right] \leq c \delta \tau \ln(1/\delta) ||\phi||_\infty\
\end{align}
Using the limit definition of the generator in \cite{oksendal}, we can see that $\L_z \psi_z = \L \psi \circ \MDS_{i'}^{-1}$. Then since $\L\psi = \phi-\wb\phi$ and $\MDS_{i'}^{-1}(x)=G(z)$,
\begin{align}
 \left|\int_A \phi(G(z)) \wh q(z)dz - \int_\Xsp \phi(y) q(y)dy \right| \leq c \delta \tau \ln(1/\delta) ||\phi||_\infty\,. \label{eqn_phibd}
\end{align}
Since this holds for all smooth $\phi\in\mathbb{L}^\infty(\Xsp)$, we obtain
\begin{align}
 ||q - G_*\wh q||_{\mathbb{L}^1(\Xsp)} \leq c \delta \tau \ln(1/\delta)
\end{align}
\end{proof}

\begin{proof}[Proof of Lemma \ref{lemm_Zerg}]

Suppose first that we use the update rule starting at $Z_0 = (x_0,i_0)$
\begin{align}
 i_{k+1} &= \text{argmin}_j\big\lvert x_k+ \eta u_k-c_{i_k,j}\big\rvert \\
 x_{k+1} &= W\big(S_{i_k,i_{k+1}}(x_k) + \wb b_{i_{k+1}} \Delta t  +  \wb \sigma_{i_{k+1}} B_{\Delta t}\big)\\
 Z^\eta_{k+1} & = (x_{k+1},i_{k+1})
\end{align}
with $\{u_k\}$ being random variables drawn from Lebesgue measure on $B_{2\delta}(0)$, and $\eta>0$. Note that $\eta=0$ in the algorithm detailed in section \ref{sec_alg} and in equation \eqref{eqn_ik}. If $\eta = 0$, the process $Z^0_n$ is not Feller continuous ($\E[Z^0_1]$ does not depend continuously on the initial conditions), a common assumption towards showing that $Z_n$ has a stationary distribution. We start with fixing $\eta>0$ and later will take $\eta \rightarrow 0$ and show that the $\eta$-dependent stationary measures converge to a new stationary measure. We temporarily suspend the use of the superscript $\eta$ to simplify notation, till towards the end of the argument, when it will be relevant again when we let $\eta\rightarrow0$.

 First we show that the process $Z_n$ is Feller continuous. Let $f$ be a bounded function on $\csp$ and $(x,i) \in \csp$. Let $p_j(x,i)$ denote the probability of transitioning to chart $j$ from $i$ starting at $x$. Note that $p_j$ is continuous and bounded provided $\eta > 0$. If $n=1$,
\begin{align}
 \E_{(x,i)}[f(Z_1)] = \sum_{j\sim i} p_j(x,i) \E_{(x,i)}[f(Z_1) \vert i_1 = j]
\end{align}
Since $\wh X_t^{Z_0}$ conditioned on $i_1$ is an It\^o process, it is Feller continuous (see \cite{oksendal}). Thus $\E_{(x,i)}[f(\wh X_{\Delta t}^{Z_0},j) \vert i_1 = j]$ is continuous and bounded and therefore $\E_{(x,i)}[f(Z_1)]$ is continuous and bounded. 
By induction if $u(x,i) = \E_{(x,i)}[f(Z_n)]$ is continuous and bounded then 
\begin{align}
 \E_{(x,i)}[f(Z_{n+1})] &= \E_{(x,i)}\big[\E_{(y,j)}[f(Z_n)] \big\vert Z_1 = y,i_1=j\big]= \E_{(x,i)}[u(Z_1)]\,,
\end{align}
which is continuous and bounded. Thus by induction on $n$, $Z_n$ is Feller continuous for all $\eta > 0$.

Next we show the transition density of $Z_n$ is tight for all $n$. Fix $\eps > 0$. Let $z = Z_{n-1}$. Then $\wt X^z_{\Delta t}$ is Gaussian with mean $\wt b_z\Delta t$ and variance $\wt \sigma_z^{\phantom{l}} \wt \sigma_z^T\Delta t$ (see equation \eqref{eqn_wtX} for their definitions). sup$_{z\in \csp}\wt b_z$ and sup$_{z\in \csp}\wt \sigma_z$ are bounded. Thus there exists an $R$ such that $\P[\wt X^z_{\Delta t} \in B_R(0)] > 1-\eps$ for all $z\in \csp$, and thus $\P[Z_n \in W(B_R(0))] > 1-\eps$. $W(B_R(0))$ is compact which implies the transition density of $Z_n$ is tight.  The transition density is tight and Feller continuous, so by the Krylov-Bogolyubov Theorem \cite{kbthm} there exists an invariant measure.

Next we show that for any $(x,i) ,(y,j) \in \csp$, if $A_y$ is a neighborhood of $y$ in chart $C_j$ then $\P[Z_n \in A_y] > 0$ for large enough $n$. The $\delta -$net is a connected graph since $\M$ is connected. Thus there exists a finite length path $\{i_k\}$ with $i_k \sim i_{k-1}$, $i_0 = i$, $i_n = j$. The probability of such a path occurring is strictly positive since the probability of jumping from $i_{k-1}$ to $i_k$ is strictly positive for all $k$. The density of $Z_n$ is strictly positive on the chart $C_{i_n}$, thus $\P[Z_n \in A_y] > 0$. Because the density of $Z_n$ is positive on any open set in $\csp$ for large enough $n$, the invariant measure is unique, and thus $Z_n$ is ergodic. 

Now let $\mu_\eta$ denote the stationary measure for $Z_n^\eta$ for each $\eta > 0$. The family of measures $\mu_\eta$ is tight, and so there exists a subsequence $\{\mu_{\eta_k}\}_{k=1}^\infty$ which converges in probability to some measure $\mu$ (see \cite{prothm}). It is left to show that $\mu$ is stationary for the process $Z_n^0$.  Let $\eps > 0$ and let $f$ be a bounded function on $\csp$. Let $\mu_k = \mu_{\eta_k}$. Let $P_k$ denote the transition density for the process $Z_n^{\eta_k}$. 

For ease of notation, for functions $g,\nu$ and transition kernels $P$,
\begin{align}
g\nu &= \int_\csp g(x)\nu(x)dx \\
gP\nu &= \int_\csp \int_\csp g(x)P(x,y)\nu(y)dxdy
\end{align}
Then $|f\mu_k - f\mu| \rightarrow 0$ by the bounded convergence theorem.
\begin{align}
 |fP_0\mu - f\mu| \leq |fP_0\mu_k - fP_k\mu_k| + |fP_0\mu - fP_0 \mu_k| + |f\mu_k - f\mu|
\end{align}
The last two terms go to zero as $k \rightarrow \infty$ because $\mu_k$ converges to $\mu$ in probability and $f$ is bounded. It is left to show that $(P_0-P_k)\mu_k \rightarrow 0$. Let $E$ denote the boundary set defined by 
\begin{align*}
 E = \left\{ (x,i) : \, \exists j \text{ with } |x| = |x - c_{i,j}|\right\}
\end{align*}
The chart centers $c_{i,j}$ are a finite set and so $E$ has $\mu$ measure zero. Let $E_k$ denote the set $E$ thickened by $\eta_k$:
\begin{align*}
 E_k = \left\{ (x,i) :\, \exists j \text{ with } \big| |x| - |x - c_{i,j}| \big| < \eta_k \right\}
\end{align*}
Fix $z=(x,i) \in \csp$ and notice that the probability density starting at $z$, $P_k(z,\cdot)$, for any $k$ is of the form 
\begin{align*}
 P_k(z,\cdot)= \sum_{j\sim i} p^k_j(z) \nu_j(\cdot)
\end{align*}
 with $p^k_j(z)$ being the probability of transitioning to chart $j$ from $i$ with $\eta = \eta_k$, and $\nu_j$ independent of $k$. For any $j$, $\nu_j$ is absolutely continuous with respect to Lebesgue measure on $\R^d$, and $\nu_j(E_k) \rightarrow 0$. It follows that
\begin{align*}
 \mu_k(E_k) \leq \sup_{z\in\csp} P_k(z,E_k) \rightarrow 0.
\end{align*}
Since $P$ and $P_k$ agree on the set $E_k^c$, $|fP_k\mu_k - fP\mu| \rightarrow 0$. Thus, $fP\mu = f\mu$ for all test functions $f$, and therefore $P\mu = \mu$.
\end{proof}

\begin{proof}[Proof of Lemma \ref{lemm_ltc}]
First let
\begin{align}
 I_n = \int_0^{\Delta t} F\big(Z_n,\wh{X}_t^{Z_n}\big)dt
\end{align}
Then define a new Markov chain $Q_n = (Z_n,Z_{n+1},I_n)$. Define a family of measures $\nu$ on $\R$ by
\begin{align}
 \nu(z_1,z_2,A) = \P[I_n \in A | Z_n = z_1,Z_{n+1} = z_2]
\end{align}
Let $\gamma$ be a measure on $\csp \times \csp \times \R$ so that
\begin{align}
 \gamma(A) = \int_{\csp \times \csp \times \R} \mathbbm{1}_{A}(r) \nu(z_1,z_2,dr) P (z_1,dz_2) \mu(dz_1)
\end{align}
Where $P$ is the transition density for $Z$. Because $Z_n$ is ergodic, $P^n(\delta_{(x,i)},\cdot) \rightarrow \mu(\cdot)$ weakly. Then by the dominated convergence theorem as $n \rightarrow \infty$,
\begin{align}
 \int_{\csp \times \csp \times \R} \mathbbm{1}_{A}(r) \nu(z_1,z_2,dr) P (z_1,dz_2) P^n(\delta_{(x,i)},dz_1) \rightarrow \gamma(A)
\end{align}
The last statement shows that the density of $Q_n$ converges weakly to $\gamma$, and so $Q$ is ergodic. Pick $\phi(Q_n) = I_n$. Then by Birkhoff's ergodic theorem (see \cite{ergthry}),

\begin{align}
 \frac{1}{n} \sum_{k=1}^n \phi(Q_k) \rightarrow \int \phi d\gamma =  \E_\mu\left[\E\left[\dsst\int_0^{\Delta t} F \big(z,\wh{X}^z_t\big) dt \right]\right]
\end{align}
\end{proof}

\begin{proof}[Proof of Lemma \ref{lemm_L-Lb}]
Choose some $z=(x,i) \in \csp$ and $f \in C^2$. Then the generators $\L_z,\wb \L_z$ are given by
\begin{equation*}
\begin{aligned}
 \L_z f (y) &= \sum_j \big(b_i(y)\big)_j  \frac{\partial f}{\partial y_j}(y) + \frac{1}{2} \sum_j \sum_k \big(\sigma^{\phantom{l}}_i(y)\sigma_i^T(y)\big)_{j,k} \frac{\partial^2 f}{\partial y_j \partial y_k}(y)\\
 \wb \L_z f (y) &= \sum_j \big(\wb b_i\big)_j  \frac{\partial f}{\partial y_j}(y) + \frac{1}{2} \sum_j \sum_k \big(\wb \sigma^{\phantom{l}}_i\wb \sigma_i^T\big)_{j,k} \frac{\partial^2 f}{\partial y_j \partial y_k}(y)
\end{aligned}
\end{equation*}
It suffices to show that $b_i(y)$ is close to $\wb b_i$ and $\sigma^{\phantom{l}}_i(y)\sigma_i^T(y)$ is close to $\wb \sigma^{\phantom{l}}_i\wb \sigma_i^T$ for each $y \in B_{2\delta}(0)$ and all $i$. Let $x_k = x_{i,k}'$ be random draws from $\MDS_i(Y_{\Delta t}^{y_i})$; These are samples of $X_t$ starting at $c_{i,i} = 0$. Then as $p \rightarrow \infty$,
\begin{align*}
 t\wb{b}_i \rightarrow \E\left[X_t\right] \qquad,\qquad
 t\wb{\sigma}^{\phantom{l}}_i\wb{\sigma}_i^T \rightarrow \text{Cov}\left( X_t \right)
\end{align*}
a.s. by the strong law of large numbers. Next in order to use finite sample bounds, we show that the random variables $x_k$ are sub-gaussian with sub-gaussian norm $\patht\kappa(|\sigma|_{F,\infty} + \patht|b|_\infty) $ for some universal constant $\kappa$. To do this, we first show $Y_{\patht}$ is sub-gaussian.

Rewrite the process $Y_s$ by the definition of the It\^o integral. Here we use a uniform partition of $(0,s)$ with $n$ subintervals so $\Delta s = s/n$, $s_j = j\Delta s$, $Y_j = Y_{s_j}$ and $z_j$ are independent standard random normal vectors in $\R^d$. Then $Y_s - Y_0$ can be written
\begin{align}
 Y_s - y_0 = \dsst\lim_{n\rightarrow\infty} \frac1n\sum_{j=0}^{n-1} b(Y_j) \Delta s + \sigma(Y_j) \sqrt{\Delta s} z_j \label{eqn_yito}
\end{align}
Note that we can always think of equations \eqref{eqn_Ydef},\eqref{eqn_yito} as being in $\R^D$ by the Whitney Embedding Theorem. If one is concerned about how to make sense of equation \eqref{eqn_yito} on a manifold in $\R^D$, see \cite{hsu}. Using proposition 5.10 in \cite{ver} on the right hand side of equation \eqref{eqn_yito}, we can see that there is a universal constant $c$ so that for each $n$,
\begin{align}
 \P\left[\left|\frac{1}{n}\sum_{j=0}^{n-1} b(Y_j) \Delta s + \sigma(Y_j) \sqrt{\Delta s} z_j\right| > \alpha\right] \leq \exp\left(\frac{-c\alpha^2}{s (|\sigma|_{F,\infty} + s|b|_\infty)}\right)
\end{align}
Taking $n \rightarrow \infty$ we conclude that the sub-gaussian norm of $Y_{\patht}$ is bounded by $\sqrt{\patht\kappa(|\sigma|_{F,\infty} + \patht|b|_\infty)}$ for a universal constant $\kappa$. Then $X_{\patht}$ is also sub-gaussian with at most the same sub-gaussian norm since $\MDS_i$ is a projection. Then $|\patht\wb b_i - \patht\E[\wb b_i]|$ can be written as a sum of mean zero sub-gaussians and by \cite{ver} there exists a $c_1$ such that,
\begin{align}
 \P[\,|\wb b_i - \E[\wb b_i]| < \varepsilon] &\geq 1-e\cdot\exp\left(-c_1\varepsilon^2 p\patht\right) \label{eqn_bprob}
\end{align}
Again by \cite{ver} bounds on finite sample covariance estimation yields for some $c_2$,
\begin{align}
 \P\Big[\,\big|\big|\wb \sigma^{\phantom{l}}_i \wb\sigma_i^T - \E\left[\wb \sigma^{\phantom{l}}_i \wb\sigma_i^T\right]\big|\big|_2 < \varepsilon \big|\big|\E\left[\wb \sigma^{\phantom{l}}_i \wb\sigma_i^T\right]\big|\big|_2\Big] &\geq 1-2e^{-c_2 d\alpha^2} \label{eqn_sprob}
\end{align}
provided $p > d\alpha^2/\varepsilon^2$. Notice that $\patht$ appears in the bound in equation \eqref{eqn_bprob}, but not in equation \eqref{eqn_sprob}. This is due to the fact that for $\patht\ll 1$, the mean is much smaller than the standard deviation (and thus harder to estimate). Estimating the covariance to within accuracy $\varepsilon$ takes $\Ord(d/\varepsilon^2)$ samples, but estimating the drift to within accuracy $\varepsilon$ takes $\Ord(1/\patht\varepsilon^2)$ samples. Assuming $\patht = \delta^2 \ll  1/d$, the mean will be more difficult to estimate. For simplicity we will assume that the covariance has the same bound as the drift.

Next we must ensure that the probabilistic bound holds for all indices $i \in \Gamma$. Since the volume of $\Xsp$ is fixed, $|\Gamma| = c_3 (1/\delta)^d$ for some $c_3$. 
Next set the accuracy to $\eps = \delta \ln(1/\delta)$, and the confidence $\tau^2 = c_1\varepsilon^2 p\patht - (1 + \ln(c_3) + d\ln(1/d))$. When we take a union bound over $i \in \Gamma$ we have with probability at least $1-2e^{-\tau^2}$,

\begin{align}
  |\wb b_i - &\E[\wb b_i]| < \delta \ln(1/\delta) \text{ and } \big|\big|\wb \sigma^{\phantom{l}}_i \wb\sigma_i^T - \E\left[\wb \sigma^{\phantom{l}}_i \wb\sigma_i^T\right]\big|\big|_2 < \delta \ln(1/\delta) \label{eqn_concineq} \\
 & \text{ if } p > \dsst\frac{c_1}{\patht\delta^2}\left( \frac{\tau^2 + 1 + \ln(c_3)}{\ln(1/\delta)^2} + \frac{d}{\ln(1/\delta)}\right) \label{eqn_pbd} 
\end{align}
We can think of equation \eqref{eqn_pbd} as telling us $p > c_4/\delta^4$ (up to $\ln(1/\delta)$ factors) since $\patht = \delta^2$ and $\tau,d,\ln(1/\delta)$ all behave like $\Ord(1)$ constants.

Since $\MDS_i$ is smooth, $b_i, \sigma_i$ are Lipschitz and bounded because $b,\sigma$ are Lipschitz and bounded by some constant $M$. By the Cauchy-Schwarz inequality and It\^o's isometry,
\begin{align}
 \E[|X_\patht|^2] \leq M^2\patht  + \Ord(\patht^{3/2}) \qquad,\qquad
 \E\left[\int_0^\patht|X_s|^2ds\right] \leq \frac12M^2\patht^2  + \Ord(\patht^{5/2}) \label{eqn_EX}
\end{align}
Let $A = \int_0^\patht b_i(X_s) -b_i(0) ds$ and $B = \int_0^\patht \sigma_i(X_s)dB_s$. Then by Jensen's inequality,
\begin{align}
 |\E[X_\patht]-\patht b(0)|^2 \leq \E[|A|^2] &\leq \E\left[\patht\int_0^\patht |b_i(X_s) - b_i(0)|^2 ds\right] \leq \frac{1}{2}C^2M^2\patht^3 + \Ord(\patht^{7/2}) \label{eqn_Abd}
\end{align}
Dividing by $\patht$ and taking a square root,
\begin{align}
 |\E[\wb b_i] - b_i(0)| \leq \sqrt{\frac{\patht}{2}}MC + \Ord(\patht^{3/4}) \label{eqn_b0bd}
\end{align}
By It\^o's isometry we have
\begin{align}
\E[|B|^2] & \leq M^2\patht  
\label{eqn_Bbd}
\end{align}
Combining equations \eqref{eqn_Bbd} and \eqref{eqn_Abd},
\begin{align}
 |\Cov(A + B) - \Cov(B)|_2 \leq \E[|A|^2] + 2\E[|A|^2]^{1/2} \E[|B|^2]^{1/2} 
 \leq \sqrt{2}CM^2\patht^2  + \Ord(\patht^{9/4}) \label{eqn_CB1}
\end{align}
Using It\^o isometry and the Lipschitz condition on $\sigma_i$, 
\begin{align}
 |\Cov(B) - \patht\sigma^{\phantom{l}}_i(0)\sigma_i^T(0)|_2 
&= \Bigg\vert \E\Bigg[ \left(\int_0^\patht \sigma_i(X_s) - \sigma_i(0) dB_s \right) \left(\int_0^\patht \sigma_i(X_s) dB_s \right)^T  \\
& \phantom{space} + \left(\int_0^\patht \sigma_i(0) dB_s \right)\left(\int_0^\patht \sigma_i(X_s) - \sigma_i(0) dB_s \right)^T    \Bigg] \Bigg\vert \\
& \leq \sqrt{2}KM\patht^{3/2} + \Ord(\patht^{7/4}) \label{eqn_CB2}
\end{align}
Combining equations \eqref{eqn_b0bd} with the concentration inequality \eqref{eqn_concineq} along with $\patht = \delta^2$ implies for some $c_5$ with probability at least $1-2e^{-\tau^2}$,
\begin{align}
 |\wb b_i - b_i(0) | \leq c_5 \delta \ln(1/\delta)
\end{align}

Finally, combining equations \eqref{eqn_CB1}, \eqref{eqn_CB2} with the concentration inequality \eqref{eqn_concineq} yields that for some constant $c_6$ with probability $1-2e^{-\tau^2}$,
\begin{align}
 \big|\big|\wb \sigma^{\phantom{l}}_i \wb\sigma_i^T - \sigma^{\phantom{l}}_i(0) \sigma_i^T(0) \big|\big|_2 \leq c_6 \delta \ln(1/\delta)
\end{align}
The Lipschitz conditions on $b_i$ and $\sigma_i$ yield the result.

\end{proof}

\begin{proof}[Proof of Lemma \ref{lemm_Lb-Lt}]
Fix a starting location $z = (x,i) \in \csp$. We can write an SDE for $\wt X^z_t$ starting at $\MDS_{i'}\left(\MDS_{i}^{-1}\left(x\right)\right)$ in the next chart $i'$ by 
\begin{align}
 d\wt X^z_t &= \wt b_z dt + \wt \sigma_z dB_t \label{eqn_wtX} \\
\wt b_z &= \frac{1}{\Delta t} \big(S_{i,i'}(x) - \MDS_{i'}\left(\MDS_{i}^{-1}\left(x\right)\right)\big) + \wb b_i \notag \\
\wt \sigma_z &= \wb \sigma_{i'} \notag
\end{align}
Writing this equation in this form spreads the transition error out over the course of one timestep of length $\Delta t$. Thus, proving $\wt \L_z$ is close to $\wb \L_z$ reduces to showing that the transition error is sufficiently small after dividing by $\Delta t$ (so that it can be combined in the drift term). Allowing the transition error to affect the drift forces us to have the drift $\wt b$ (and thus $\wt \L$) depend on the starting location $z$.

By the Whitney embedding theorem, $\M$ can be smoothly embedded into $\R^{D}$ for $D \geq 2d$. In $\R^{D}$, the LMDS mapping $\MDS_i$ reduces to principal component analysis, which is simply a projection onto the top $d$ eigenvectors of the covariance matrix of the local landmarks. Thus we can think of $\MDS_i$ as a matrix acting on vectors. To be consistent with the algorithm, vectors will be written as row vectors and the matrix $\MDS_i$ will act on the left. 

Fix $k \in \Gamma$. Let $\Pi_k \in \R^{D\times d}$ denote the projection from $\Xsp$ onto $T_{y_k}$, the tangent plane of $\Xsp$ at $y_k$. Then $\Pi_k$ is invertible on a ball of radius $2\delta$ on $T_{y_k}$ by assumption. Also since $\Xsp$ is smooth, Taylor's theorem tells us that for some $c_1$ and all $x\in \Xsp$ near $y_k$,
\begin{align}
 |x\Pi_k - x| \leq c_1|x-y_k|^2
\end{align}
 Let $L_k = \{l_{k,i}\}$ denote the collection of landmarks associated with the neighbors of $y_k$, and $\mu_k$ denote their mean. The matrix $\MDS_k$ minimizes the squared error on the landmarks given by:
\begin{align}
 \sum_i |l_{k,i} \MDS_k - l_{k,i}|^2 \label{eqn_mdsmin}
\end{align}

Inserting $\Pi_k$ in place of $\Phi_k$ into equation \eqref{eqn_mdsmin} yields a bound of $c^2\delta^4$. The landmarks are well spread through the space by construction and the ellipticity condition. 
Then with high probability (in fact, at least $1-e^{-\tau^2}$ by \cite{Vershynin:NARMT}, since $p\gtrsim d\log d$) $\wt L = (L-\mu)\Pi = \{\wt l_i\}$ must have smallest singular value at least $\delta$, and thus any vector $v$ in the tangent plane can be written $v = a\wt L$ with $a = v \wt L ^\dagger$. The bound on the singular values implies $|a| \leq \delta^{-1} |v|$.  Then using Cauchy-Schwarz,
\begin{align}
 |v\MDS_k - v| \leq |v|c_1\delta\,, \label{eqn_mdsprop1}
\end{align}
which implies, since $\MDS_k$, $\Pi_k$ are projections, $||\MDS_k - \Pi_k||_2 \leq c_1\delta$. Let $j \in \Gamma$ such that $j\sim k$. By a Taylor expansion, $||\Pi_k - \Pi_j|| \leq  c_2\delta$ for some $c_2$. Thus there exists a constant $c_3$ such that

\begin{align}
 ||\MDS_k - \MDS_j||_2 \leq c_3 \delta \label{eqn_mdsprop2}
\end{align}

The properties \eqref{eqn_mdsprop1} \eqref{eqn_mdsprop2} allow us to treat $\MDS_k$ like $\Pi_k$, the projection onto the tangent plane. Also since $||\MDS_k - \Pi_k||_2 \leq c_1\delta$ and $\delta \ll  1$, $\MDS_k$ will be invertible whenever $\Pi_k$ is since $\Pi_k$ has singular values equal to $1$.

Next let $A = A_{k,j} = \{a_{k,i}\}\cup \{a_{j,i}\}$ be the collection of landmarks common to $L_k$ and $L_j$. Let $\mu = \mu_{k,j}$ be the mean of $A$. Now we can write $T_{i,j}$ as
\begin{align}
T_{i,j} = [(A - \mu)\MDS_i]^\dagger(A-\mu)\MDS_j
\end{align}
By definition of the pseudoinverse, $T_{i,j}$ minimizes
\begin{align}
||(A-\mu)\MDS_i T - (A-\mu)\MDS_j||_2   \label{Amin}
\end{align}
over all choices of $T \in \R^{d \times d}$. Choose $T$ to be the restriction of $\MDS_j$ onto chart $C_i$. Then $T \in \R^d \times \R^d$ and 
\begin{align}
||(A-\mu)\MDS_i T - (A-\mu)\MDS_j||_2   
& = ||((A - \mu)\MDS_i - (A - \mu))(\MDS_j - \MDS_i)||
\leq c_1 c_2\delta^3
\end{align}
Since $T$ is a possible choice for $T_{i,j}$,
\begin{align}
||(A-\mu)\MDS_i T_{i,j} - (A-\mu)\MDS_j||_2 \leq c_1 c_3\delta^3
\end{align}
The matrix of landmarks $(A-\mu)$ spans the chart $C_i$, so there is a constant $c_4$ such that for any $x$ in the chart $C_i$,
\begin{align}
 |S_{i,j}(x) - \MDS_j(\MDS_i^{-1}(x))| \leq c_4 \delta^3 \label{eqn_switchbd}
\end{align}

Using $\Delta t = \delta / \ln(1/\delta)$ and equation \eqref{eqn_switchbd}, the result follows.
\end{proof}

\begin{proof}[Proof of Lemma \ref{lemm_Lt-Lh}]
Fix a starting location $z = (x,i) \in \csp$. Then the process $\wt X_t^z$ is the solution of an SDE on chart $i'$ with smooth coefficients. Thus, $\wh X_t^z = W\big(\wt X_t^z\big)$ is also the solution of an SDE on chart $i'$ with smooth coefficients:
\begin{align}
 d\wh X_t &= \wh b(z,\wh X^z_t) dt + \wh \sigma(z,\wh X^z_t) dB_t
\end{align}
Using It\^o's formula on $W(\wt X)$,
\begin{align}
 \wh b_j (z,\wh X^z_t) &= \sum_k \frac{\partial W_j}{\partial x_k}(\wt X^z_t) \wt b_k(z) + \frac12 \sum_k \sum_l \frac{\partial^2 W_j}{\partial x_k \partial x_l}(\wt X^z_t) (\wt \sigma \wt \sigma^T)_{k,l}(z) \\
\wh \sigma_{j,l}(z,\wh X^z_t) &= \sum_k \frac{\partial W_j}{\partial x_k} \wt \sigma_{k,l}(z) \label{eqn_whsdef}
\end{align}
with $\wt b(z)=\wt b_z,\wt \sigma(z) = \wt \sigma_z$ defined as in the proof of Lemma \ref{lemm_Lb-Lt}. Note that since $W$ is invertible, we could replace $\wt X^z_t$ with $W^{-1}(\wh X^z_t)$ so that $\wh b,\wh \sigma$ can be thought of as a function of $z$ and $\wh X^z_t$. 
Direct computation shows that for some $c_1,c_2$,
\begin{align*}
 \sum_k \left(\frac{\partial W_j}{\partial x_k}(\wt X_t)\right)^2 &\leq c_1\qquad,\qquad
 \sum_k \sum_l \left(\frac{\partial^2 W_j}{\partial x_k \partial x_l}(\wt X_t)\right)^2 \leq \frac{c_2}{\delta^2}
\end{align*}
Let $E_t$ denote the set $\left\{t :  |\wt X^z_t| > \dsst\frac{3\delta}{2}\right\}$. By definition of $W$, $\wh b$ and $\wt b$ agree on $E_t^c$. 
The index $i'$ is chosen so that $|x-c_{i,i'}| < \delta$. As the switching map $S_{i,i'}$ makes error $\Ord(\delta^3)$ by equation \ref{eqn_switchbd}, $|\wt X^z_0| \leq \delta + \Ord(\delta^3)$. In order that $|\wt X^z_t|>3\delta/2$, the Brownian motion must push the process at least $\Ord(\delta)$ in time $t$. In other words there are constants $c_3,c_4$ such that,
\begin{align}
\P[E_t] \leq \P[|B_t| > c_3\delta] \leq \exp(-c_4\delta^2/4t) \label{eqn_Etbd}
\end{align}
 Next we can bound the effect of the boundary function $W$ on the drift and diffusion terms:
\begin{align*}
 \E\left[ \int_0^{\Delta t} \left|\wh{b} - \wt{b}\right|^2\big(z,\wh{X}^z_t\big)dt \right] &= \E\left[ \int_0^{\Delta t} \mathbbm{1}_{E_t}\left|\wh{b} - \wt{b}\right|^2\big(z,\wh{X}^z_t\big)dt \right]
\leq \frac{c_5 \Delta t}{\delta^2} \exp\left(-c_4 \frac{\delta^2}{\Delta t} \right)
\end{align*}
for some new constant $c_5$. By equations \eqref{eqn_whsdef}, \eqref{eqn_Etbd} and the fact that $\wh \sigma$ agrees with $\wt \sigma$ on $E_t^c$,
\begin{align}
 \E\left[\int_0^{\Delta t}\left|\left|\wh{\sigma}\wh{\sigma}^T - \wt{\sigma}\wt{\sigma}^T\right|\right|^2_F \big(z,\wh{X}^z_t\big) dt\right]
\leq \Delta t c_6\exp\left(-c_4 \frac{\delta^2}{\Delta t} \right)
\end{align}
The result follows for $\Delta t = \delta/\ln(1/\delta)$.
\end{proof}


\section{
Examples
} \label{sec_exs}

\subsection{Simulator Comparison} \label{sec_simcomp}

\subsubsection{Multiscale Transition Probability Comparisons} \label{sec_probcomp}
In order to see how well the \ls\ works we will need to have a criterion for comparing simulators. Since we are interested in the behavior of the system over multiple timescales, we will simulate 10,000 paths from each simulator and record the positions at times $\{t_k:=2^k\}$. The smallest time scale will be at the size of one step of the original simulator and the largest time scale will be at some time $T$ (example dependent) at which point systems have reached equilibrium.

\begin{figure}[hbp]
\begin{center}
\centering\includegraphics[width=6in]{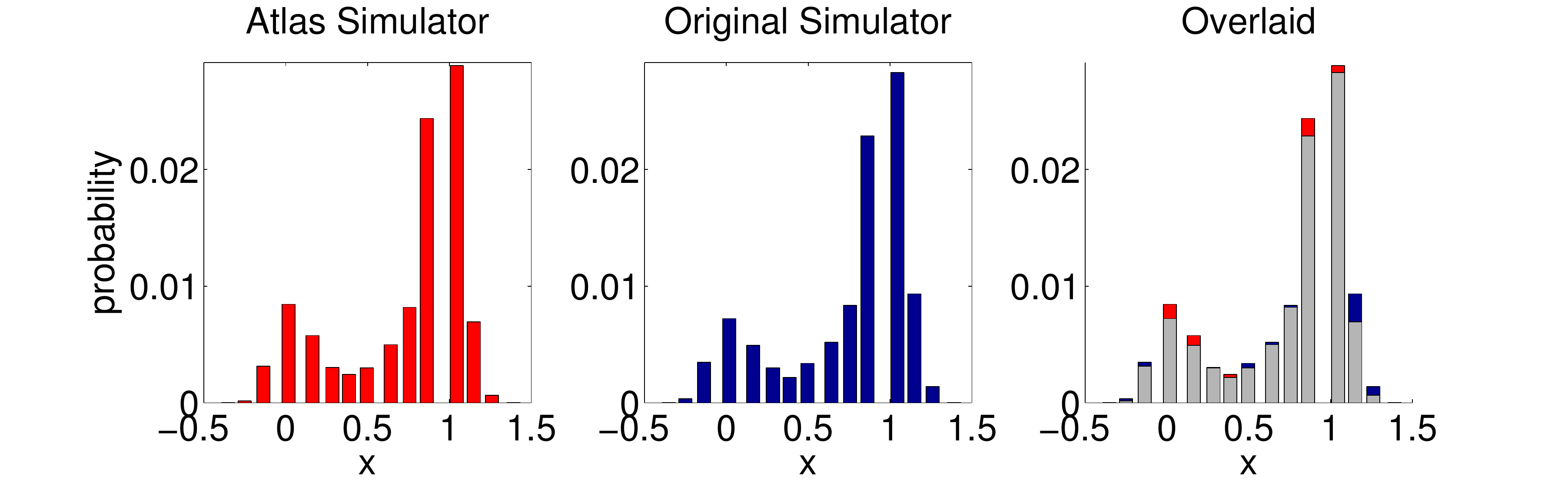}
\caption[Comparing bar graphs for transition probabilities.]{Comparing distributions obtained from two simulators at time $T=0.2$ (orginal and \ls) in example \ref{ex_dwwr}.}\label{fig_bar1}
\end{center}
\end{figure}

\begin{figure}[hbp]
\begin{center}
\centering\includegraphics[width=6in]{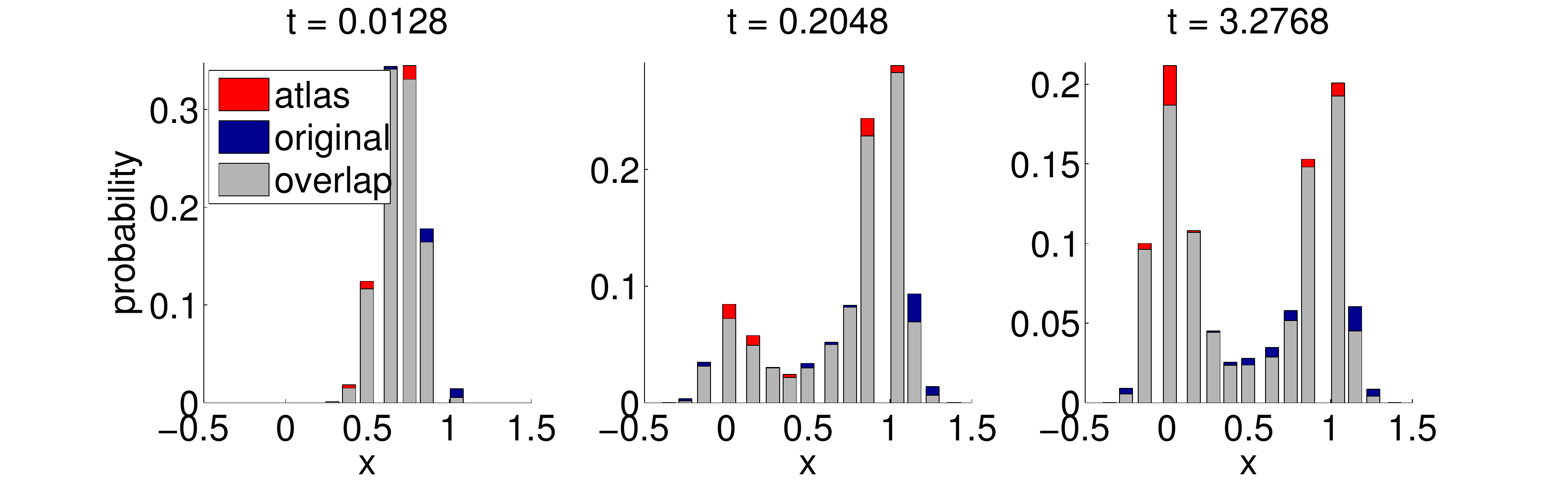}
\caption[Comparing multiple bar graphs for multiscale times.]{Comparing overlaid distributions obtained from two simulators at some multiscale times $2^k$ (originally and \ls) in example \ref{ex_dwwr}.}\label{fig_bar2}
\end{center}
\end{figure}

In order to understand motivation for how to compare simulators, we start with a 1-d example. For a fixed $t_j \leq T$, we can bin samples into equal spaced bins, and compare them. Next we would like to compare the probabilities of landing in each bin as in figure \ref{fig_bar1} by overlaying them. We can next vary $k$ (and thus $t_k$) to obtain overlaid bar graphs for multiple time scales as in figure \ref{fig_bar2}. We see that the real quantity of interest is the difference between these two histograms, and we will sum the absolute values of their difference to approximate the $L^1$ distance between the measures these histograms represent. 

Our next goal is to generalize this to high dimensional spaces. Here the bins we use can be given naturally by the \ls\ we construct. Instead of using a ``hard'' binning procedure by assigning each point to the closest bin, we will assign smooth weights to the nearest neighbors. This smooth binning procedure will help to wash out the small scale errors we make, so that we can measure the large scale errors.

The first step is to explain the smooth map which takes a distribution $\nu$ on $\{x_i\}_{i=1}^n$ to a distribution $\mu$ on a set $\{y_j\}$. We think of $\{y_j\}$ as a coarser binning of the distribution $\nu$ on $\{x_i\}$. First assign weights $w_{i,j}$ to each $(x_i,y_j)$ pair given by:
\begin{displaymath}
   w_{ij} = \left\{
     \begin{array}{cr}
       \exp\left(\frac{-|x_i - y_j|^2}{\delta^2}\right) &  |y_i - x_j| < 2\delta \\
 & \\
       0 &  \text{otherwise}
     \end{array}
   \right. \label{eqn_weights1}
\end{displaymath} 
Then we normalize the weights so that they sum to $1$ when summed over $j$.
\begin{align}
 \mu_j = \sum_i  \nu_i \frac{w_{i,j}}{\sum_j w_{i,j}} \label{eqn_weights2}
\end{align}

Fix a time slice $t_k$, then assign equal weights $\nu_i = 1/n$ to the set of samples $\{x_i\}_{i=1}^n$ given by the original simulator, and map them to a distribution $\mu$ on the net $\Gamma$ using \eqref{eqn_weights1},\eqref{eqn_weights2} and the distance function in the ambient space. Next we will assign equal weights to the samples $\{\wh x_i\}_{i=1}^n$ from the \ls\ and map them to weights $\wh \mu$ on the net $\Gamma$ using the euclidean chart distances. 

Once we have $\mu$, $\wh \mu$, we could compare them directly. However, we know that the \ls\ makes errors on this spatial scale, and so we would like to smooth these distributions out to a coarser net with $\delta_c \geq \delta$. This will also allow us to compare simulators with varying $\delta$ while keeping the number of bins fixed. For each example, we will fix a coarse grained $\delta_c$ equal to the largest $\delta$ used for that example, and obtain a net $\{z_l\}$. Then we can push $\mu$, $\wh \mu$ to distributions $p$, $\wh p$ on $\{z_l\}$ again using \eqref{eqn_weights1},\eqref{eqn_weights2} and the distance function in the ambient space. This gives us two probability distributions, one for each simulator, at time $t_k$ on the coarse net.

Given a single initial condition, we will calculate the $L^1$ distance between $p$ and $\wh p$ for each time slice $t_k$. Then we will repeat this procedure for 10 fixed initial conditions (randomly chosen) to compare the transition densities over a wide range of time scales and initial conditions. In examples where only one \ls\ is used, we plot one thin colored line for each initial condition, then a thick line representing the mean $\pm$ one standard deviation (see figures \ref{fig_comp_dwwr}, \ref{fig_comp_twwr}, \ref{fig_comp_twimg_wr}, \ref{fig_comp_fcn_dim}). In examples where we compare many \ls\ simulators, we plot only the thick line representing the mean $\pm$ one standard deviation (see figures \ref{fig_comp_dwnr}, \ref{fig_comp_twnr}, \ref{fig_comp_twimg_nr}, \ref{fig_comp_fcn}).

\subsubsection{Transition Time Comparisons}
Another quantity of interest in many stochastic dynamical systems are (expected) transition times between metastable states. Metastable states are subsets of the state space where the system spends a significant amount of time before escaping. An example includes a region near at the bottom of either well of the potential function in the example considered in section \ref{ex_onedim}, and discussed in detail below. These transitions between a metastable state (set) and another are often one of the most important characteristics of large-time dynamics in the systems we consider. The expected time between these transitions, i.e. the average time spent in a metastable state before jumping to another one - is a fundamental statistic of the system, and is a function of two given metastable states (sets).

Therefore another way to compare the \ls\ with the original simulator is to compare these expected transition times. We start by running 12 extremely long paths (100 times longer than those run in \ref{sec_probcomp}) from both simulators; these paths have on the order of thousands of transitions. Next we identify regions of interest, which requires knowledge of the problem ahead of time. Once this is done we can classify points in the original state space as belonging to region 1, region 2, region 3 (if there are three states), or none. Then all points in the long simulations are classified. Last, we scan through the list and calculate the transition times. Calculating transition times is best explained through an example. Suppose our simulation now looks like:

\begin{align}
 0, 1, 1, 0, 0, 1, 0, 2, 0, 3, 3, 0, 1
\end{align}

Denote a transition time from region $i$ to region $j$ by $\tau_{i,j}$. Start by skipping to the first time the simulator enters a region (region 1 in this case). Then it takes 6 timesteps to reach another region, region 2. Count this as a sample of $\tau_{1,2}$ equal to $6$ timesteps. Next it takes $2$ steps to go from region 2 to region 3, so count this as a sample of $\tau_{2,3}$ equal to 2 timesteps. Next we are in region 3, and it takes 3 steps before reaching region 1, so count this as a sample of $\tau_{3,1}$ equal to 3 timesteps. There are no samples of $\tau_{1,3}$ or $\tau_{3,2}$. The long paths we calculate will have many such samples, and we can average the value of these samples, and also average over the 12 paths we have run. We can then plot estimated values for $\E[\tau_{i,j}]$, the expected transition times for the original simulator, and estimated values of $\E[\wh \tau_{i,j}]$, the expected transition times for the \ls.

\subsection{One Dimensional Example}

\subsubsection{Smooth Potential} \label{ex_dwnr}
The first example presented is a simple one dimensional two-well example.  We will use the potential
\begin{align*}
 U_1(x) = 16x^2(x-1)^2
\end{align*}
and use a simulator which approximates
\begin{align*}
 dX_t = -\nabla U_1(X_t) dt + dB_t
\end{align*}
using an Euler-Maruyama scheme which takes timesteps of size 0.005. A sample path of this system is shown in figure \ref{fig_dwtraj}. The initial point set we use to generate the $\delta$-net is linearly spaced points with spacing 0.01. It is important to note that the distribution of the initial point set does not play an important role in the resulting \ls. The \ls\ algorithm performs equally well on any initial point set that has no holes of size order $\delta$. We subsample this initial point set to obtain a $\delta$-net with $\delta$ parameter 0.1 using the brute force method described in section \ref{sec_dnet}.
 
\begin{figure}[tbp]
\centering\includegraphics[width=5.5in]{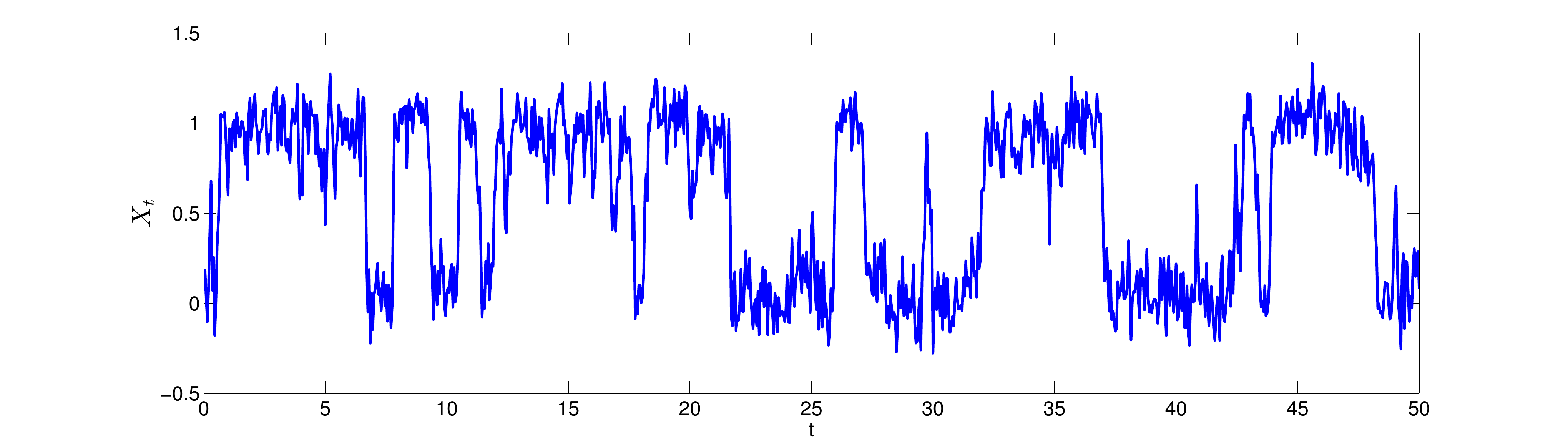}
\caption{Sample trajectory of $X_t$ for the two well example \ref{ex_dwnr}.}\label{fig_dwtraj}
\end{figure}

Once we run the \ls\ algorithm in this case, it is simple to map estimated drift vectors from the chart coordinates back to the original space. In general for an arbitrary metric space this is a hard problem, but in 1-d we need only multiply by $\pm1$ to undo MDS. In 1-d, the estimated drift vectors can easily be integrated to obtain an effective potential $\wh{U}$ for the system. We can also bring back the diffusion coefficients and see how they compare to the truth. Inverting MDS and comparing the coefficients we obtain with the true coefficients of the underlying system is a procedure we will only be able to do for this 1-d system, but it gives interesting insight into the working of the homogenizing nature of the \ls.

\begin{center}
\begin{figure}[tbp]
\centering\includegraphics[width=5.5in]{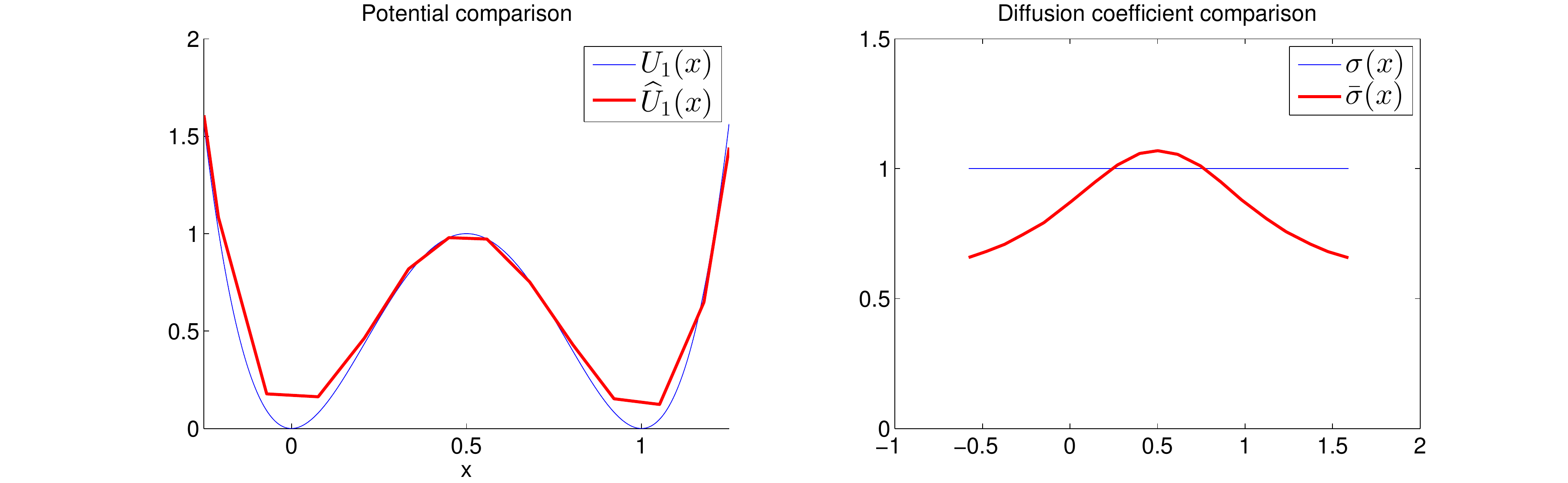}
\caption[Comparing original potential and diffusion with that of the \ls in example \ref{ex_dwnr}.]{Left: original potential $U$ (shown in blue) and effective potential of the \ls\ $\wh{U}$ (shown in red). Right: comparing original diffusion coefficient (blue) with that of the \ls (red) with $\delta = 0.1$ in example \ref{ex_dwnr}.}\label{fig_dwpot}
\end{figure}
\end{center}
See Figure \ref{fig_dwpot} showing the resulting comparisons between drift and diffusion coefficients. 

Next we generate four nets (and four \ls\ simulators) with $\delta$ values $0.05$, $0.10$, $0.15$ and $0.20$ by subsampling from a fine mesh. In each example we have used $p = 10,000$ simulations per net point, and $\patht=\delta^2$. The number of landmarks is irrelevant because as long as $m \geq 1$, there will be enough landmarks to exactly recover the local space. When simulating, we set the simulation time step $\Delta t = \delta^2/5$. Then for each of $10$ randomly chosen staring locations, we run $10,000$ long paths up to time $T = 50$. Using the simulator comparison method from section \ref{sec_simcomp}, we obtain figure \ref{fig_comp_dwnr}.
\begin{figure}[tbp]
\centering\includegraphics[width=5.5in]{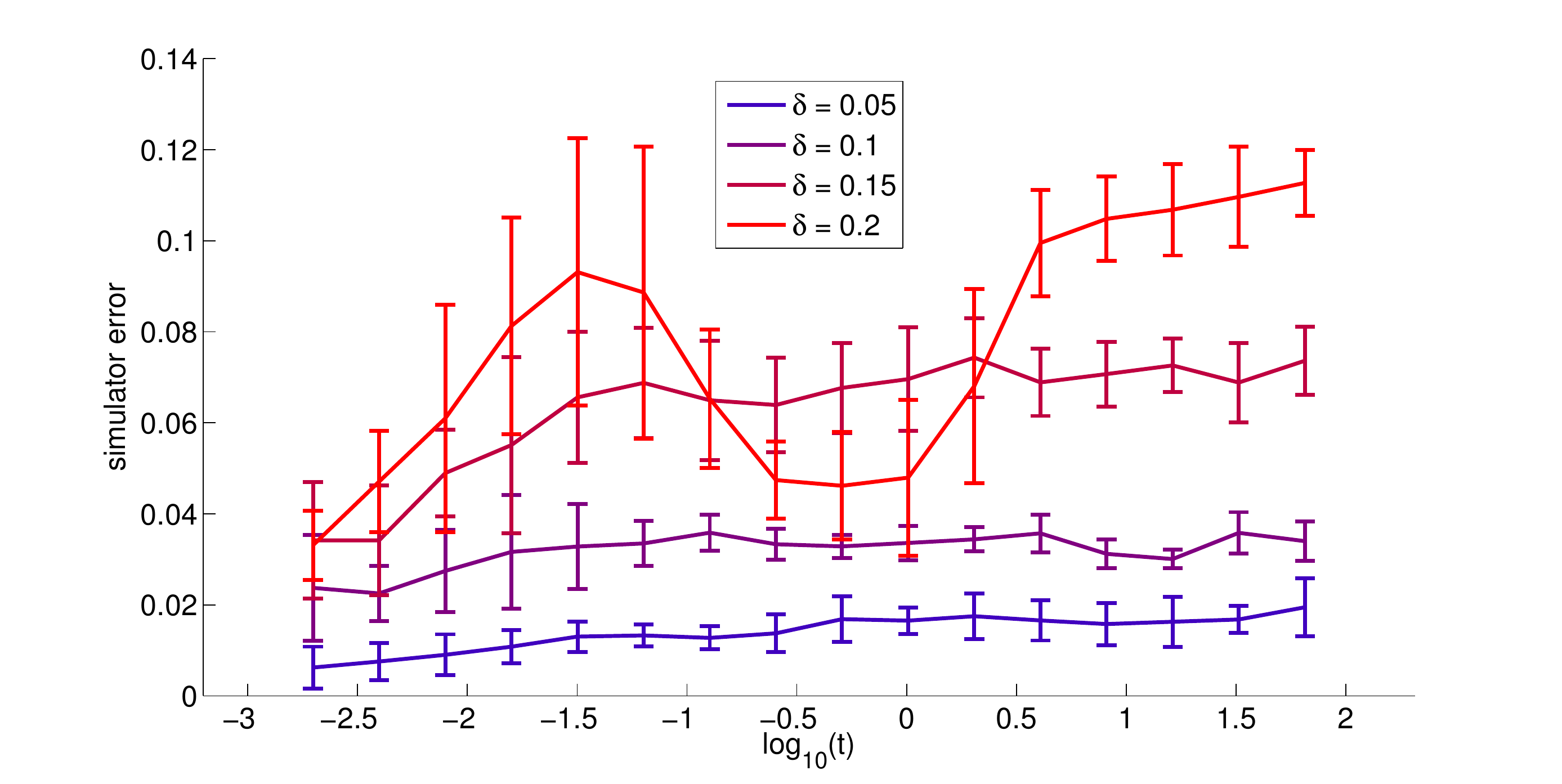}
\caption{Simulator comparison for example \ref{ex_dwnr}. Each line represents the average simulator error for a single net of the specified $\delta$ value.}\label{fig_comp_dwnr}
\end{figure} 
As we expect from theorem \ref{thm_main}, the long time error is decreasing as $\delta$ decreases. Figure \ref{fig_comp_dwnr} shows that the transition kernels are close for all time scales, which is a stronger experimental result than given by theorem \ref{thm_main}. Theorem \ref{thm_main} only tells us that the stationary distributions are $\Ord(\delta \log(1/\delta))$ far from each other. We can also compare the rates directly as seen in figure \ref{fig_rates_dwnr}.

\begin{center}
\begin{figure}[tbp]
\centering\includegraphics[width=5.5in]{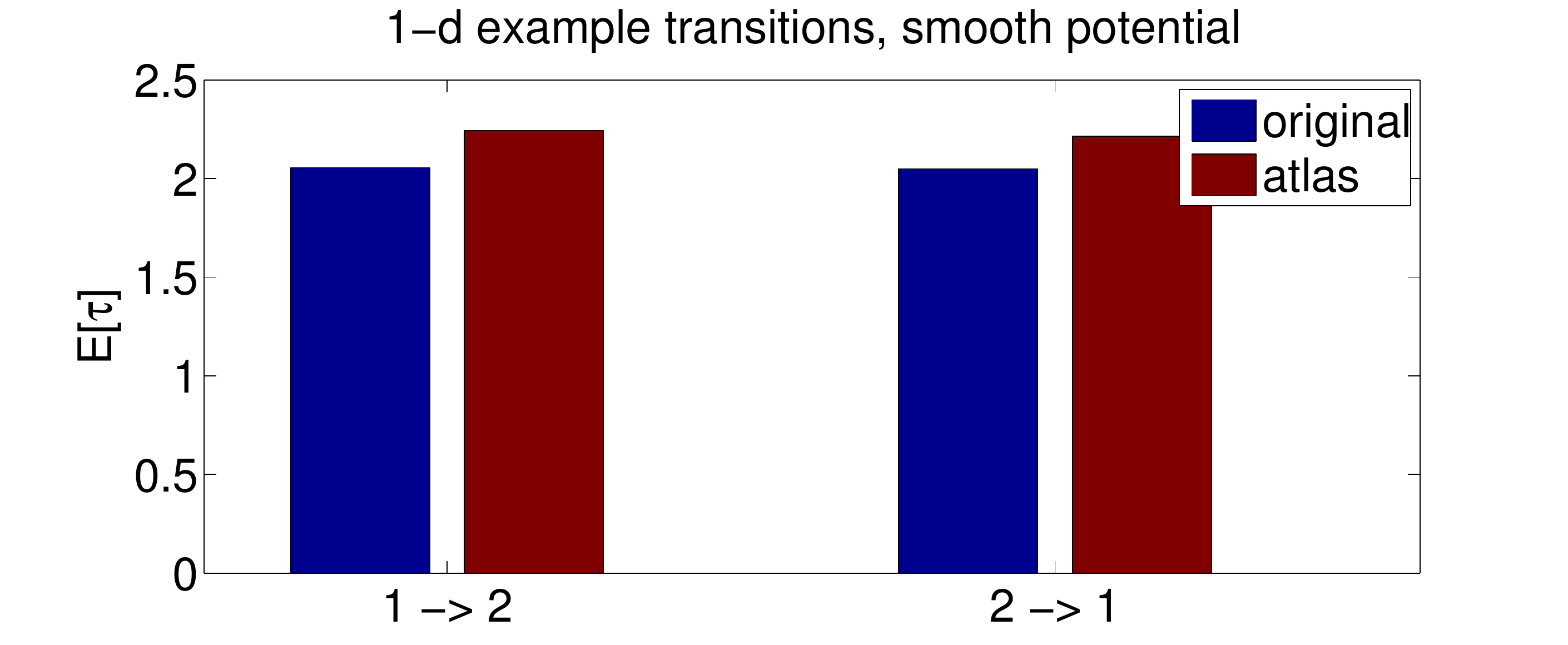}
\caption{Comparing transition times in example \ref{ex_dwnr}. Region 1 is $\{x:|x|<1/4\}$, and region 2 is $\{x:|x-1|<1/4\}$.}\label{fig_rates_dwnr}
\end{figure}
\end{center}

\subsubsection{Rough Potential} \label{ex_dwwr}
In order to make a more interesting example, we add high frequency ridges to the potential well to emulate microscale interactions. This example is a case where it is of interest to approximate a homogenized system which behaves like the original system above a certain temporal/spatial scale. Define  
\begin{align}
 V_1(x) = U_1(x) + \frac16\cos(100\pi x)
\end{align}
where $U_1(x)$ is defined in example \ref{ex_dwnr}. For our initial point set, we could again use evenly spaced grid points as in example \ref{ex_dwnr}. Since one might wonder if this is a ``fair'' input we run each grid point through the simulator for a small time $t=0.01$ to obtain our initial point set. As long as these points have no holes of size order $\delta$, the \ls\ will return a robust result with high probability.

Even though the new potential well is infinitely differentiable, the Lipschitz constant of the drift in this example is 625. In order to accurately simulate Brownian motion in this potential well, we decrease the time step to $0.00005$. These microscale interactions are determining our timestep, and thus becoming a bottleneck for running long time simulations.

If we were to apply theorem \ref{thm_main} directly to this example, it will guarantee a relatively useless error bound on the stationary distribution (since the error bound depends on the Lipschitz constant). Instead, the way we think of theorem \ref{thm_main} applying to this problem is that there is a time scale $\patht$ at which the system with potential well $V_1$ behaves like a homogenized version with smooth potential and small Lipschitz constant. Multiscale systems of this form have been studied (see \cite{pavliotis2007parameter} and references therein), and it is known that such systems behave like an SDE with smooth parameters at a large scale. If we only observe samples at time $\patht$, then we can pretend our samples come from the homogenized system rather than the microscale simulator. 

\begin{figure}[tbp]
\centering
\centering\includegraphics[width=6in]{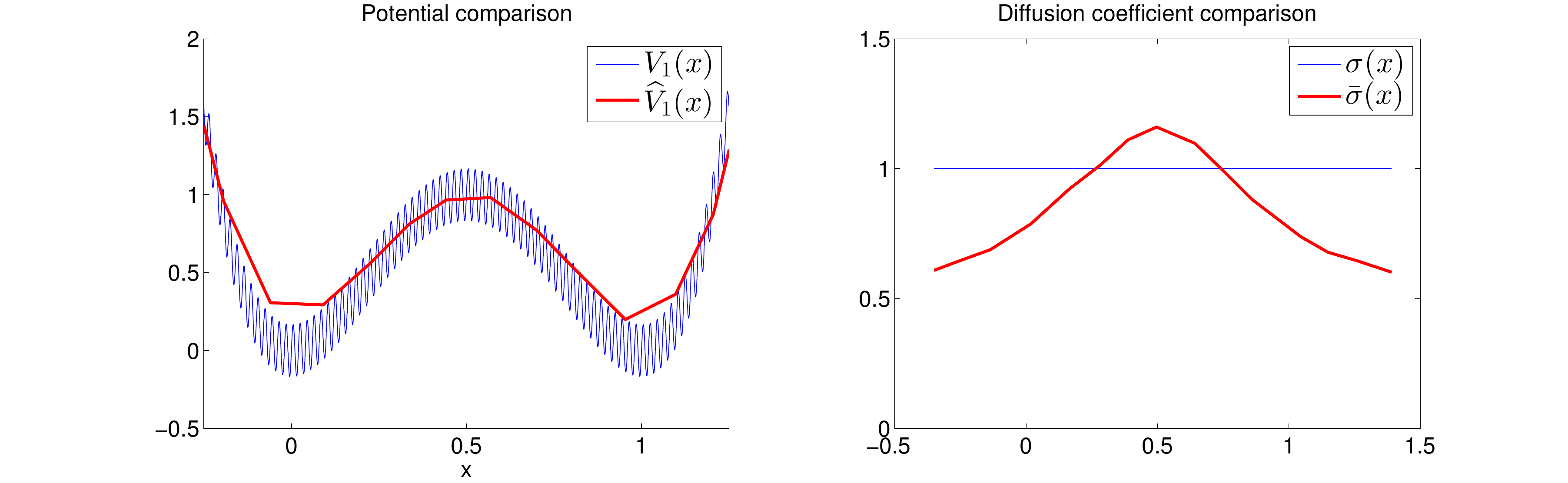}
\caption[Comparing original potential and diffusion with that of the \ls in example \ref{ex_dwwr}.]{Left: original potential $U$ (shown in blue) and effective potential of the \ls\ $\wh{U}$ (shown in red). Right: comparing original diffusion coefficient (blue) with that of the \ls (red) with $\delta = 0.1$ in example \ref{ex_dwwr}.}\label{fig_rpdw}
\label{f:Verrorhf}
\end{figure}
In this example, we learn the \ls\ using the parameters $\delta = 0.1$, $\patht = 2\delta^2 = 0.02$, $p = 10,000$, and $\Delta t = \patht/5$. Again we can map the drift and diffusion back to the original space and compare with the true simulator. Figure \ref{fig_rpdw} shows that the resulting drift is a homogenized version of the original system. The time scale the local simulator uses is $100$ times larger than that of the original system. This will result in long simulations being about $100$ times faster than using the original simulator.

Next we have run 10,000 long paths from the \ls\ with $\delta = 0.1$ shown above in figure \ref{fig_rpdw}. Figure \ref{fig_comp_dwwr} shows that again the distribution of paths is similar over multiple timescales, indicating that transition rates are preserved between states. 

\begin{figure}[tbp]
\centering\includegraphics[width=5.5in]{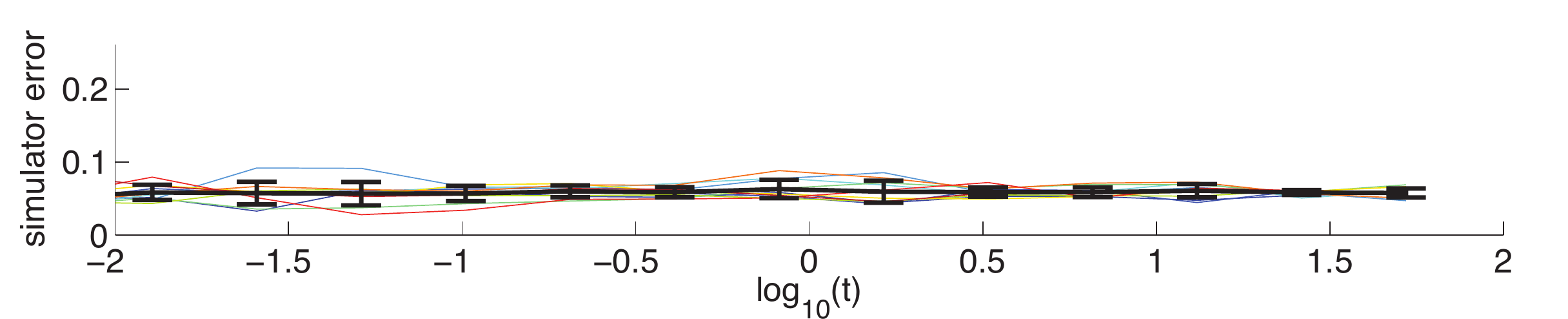}
\caption{Comparing true simulator with the \ls\ with $\delta=0.1$ on example \ref{ex_dwwr}.}\label{fig_comp_dwwr}
\label{f:simerrorhf}
\end{figure}

\begin{center}
\begin{figure}[tbp]
\centering\includegraphics[width=5.5in]{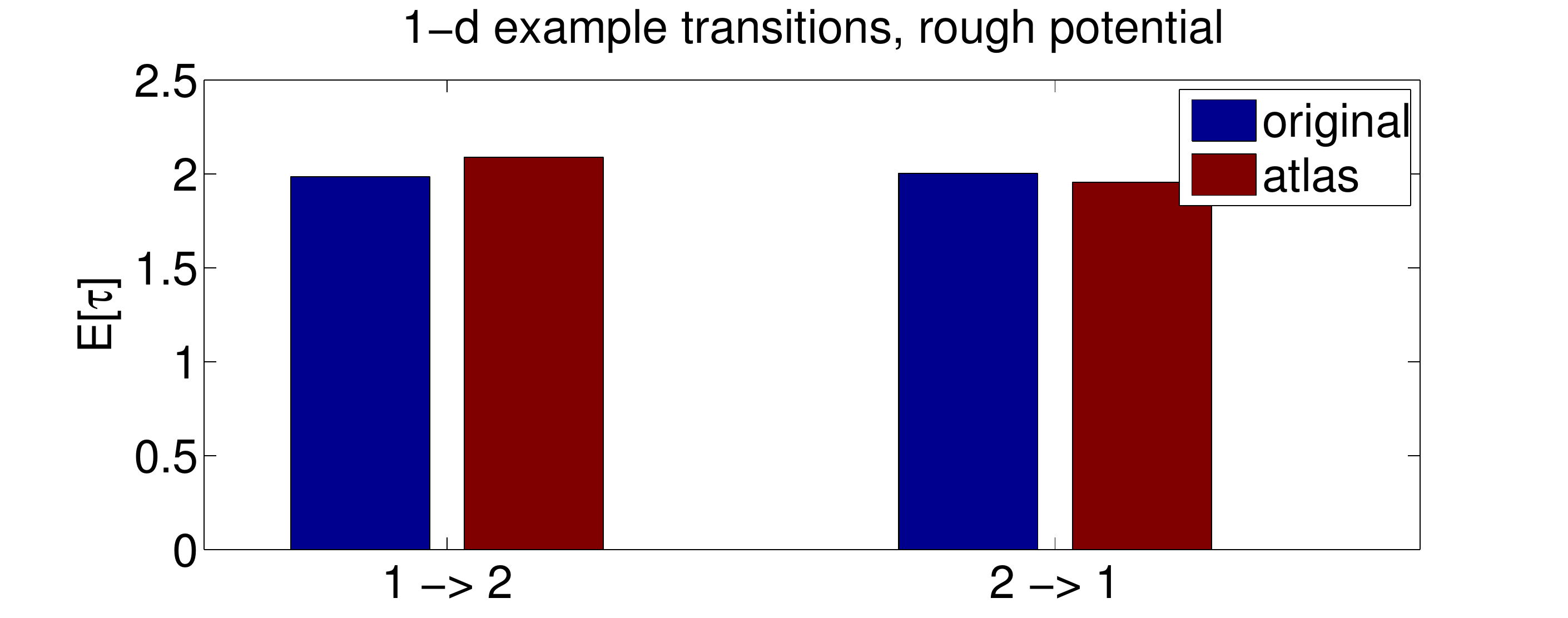}
\caption{Comparing transition times in example \ref{ex_dwwr}}\label{fig_rates_dwwr}
\end{figure}
\end{center}
In this example we only show results for $\delta = 0.1$ because that is the spatial scale where it makes sense to homogenize. For smaller values of $\delta$, the \ls\ becomes less stable as the estimated drift becomes less smooth. For larger values of $\delta$, the macroscale features of the system begin to wash out, and the two wells merge into one.

\subsection{Two Dimensional Example}
\subsubsection{Smooth Potential} \label{ex_twnr}
In this example, we consider the SDE $X_t$ in a 2-d potential well $U_2(x)$ shown below. 

\begin{align*}
dX_t &= -\nabla U_2(X_t) dt + dB_t 
\end{align*}
where
\begin{align*}
U_2(x) &= -\text{ln} \left( \text{exp}\left(\frac{-||x - p_1||^2}{c_1} \right) + \text{exp}\left(\frac{-||x - p_2||^2}{c_2} \right) + \text{exp}\left(\frac{-||x - p_3||^2}{c_3} \right) \right)\\
 p_1 &= \left[ \begin{array}{c} 0 \\ 0 \end{array} \right], \phantom{sp}  p_2 = \left[ \begin{array}{c} 1.5 \\ 0 \end{array}\right], \phantom{sp}  p_3 = \left[ \begin{array}{c} 0.8 \\ 1.05 \end{array} \right], \phantom{sp} c = \left[\frac{1}{5}, \frac{1}{5},\frac{1}{6} \right] \\
\end{align*}

The potential $U_2$ is chosen such that the stationary distribution is a mixture of Gaussians given by $\exp(-U_2/2)$. There are three clearly defined minima of $U_2$ close to $p_1,p_2,p_3$.  The parameters of the problem were chosen in such a way that the transition regions between wells lie on different level sets of the potential (see figure \ref{fig_exnet}). 
\begin{figure}[tbp]
\centering\includegraphics[width=5.5in]{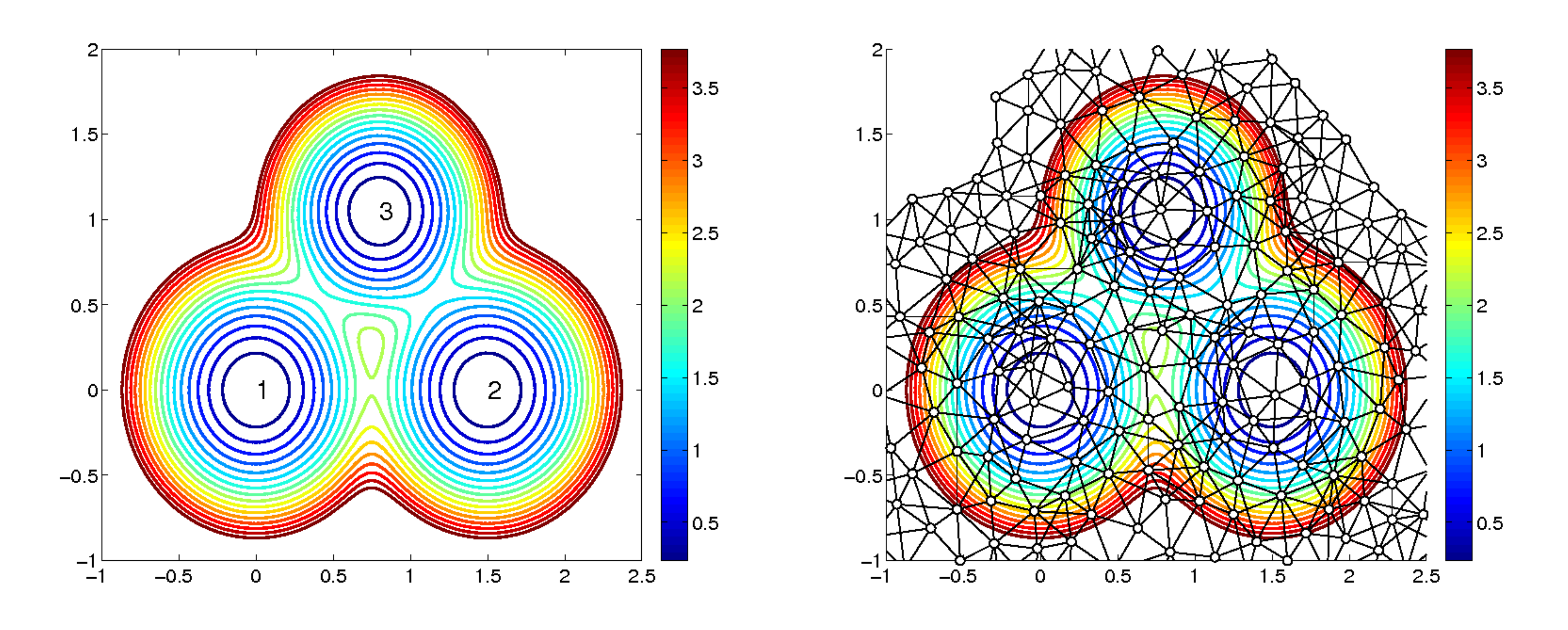}
\caption{Left: Potential for three well example. Right: $\delta=0.2$ net overlaid. Circles represent net points, black lines represent connections between net points.} \label{fig_exnet}
\end{figure}

\begin{figure}[tbp]
\centering\includegraphics[width=5.5in]{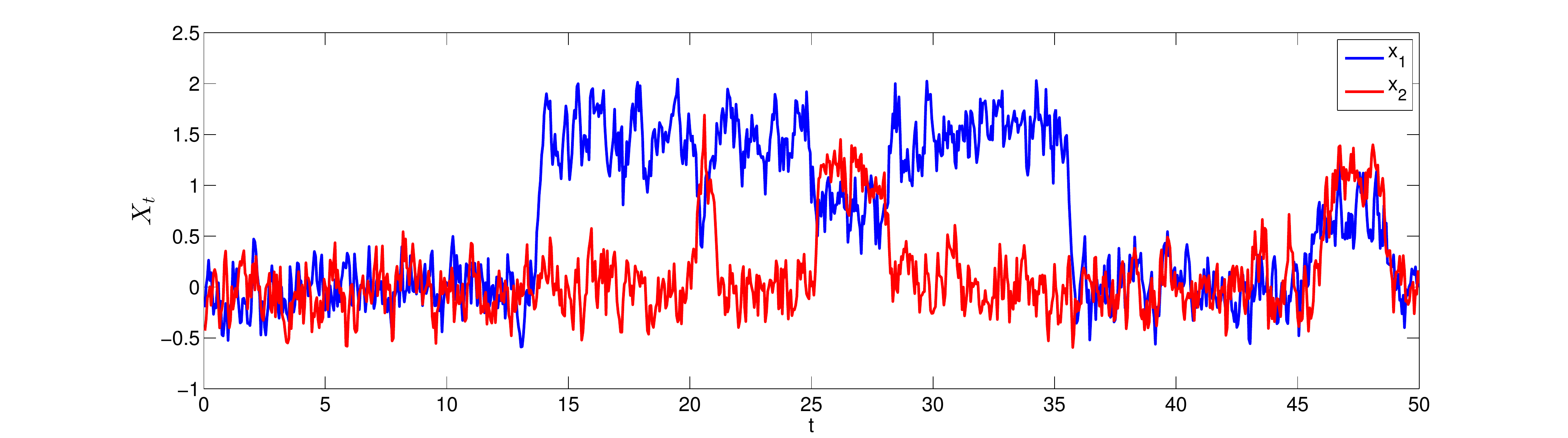}
\caption{Sample trajectory for three well example}\label{fig_twtraj}
\end{figure}
Figure \ref{fig_twtraj} shows a sample trajectory of the process $X_t$ using a simple Euler-Maruyama scheme with timestep 0.005. This is the simulator given to the \ls\ algorithm. The initial point set we use is a grid spaced by 0.01, discarding points with $U_2(x) \geq 10$. Figure \ref{fig_exnet} shows an example net for $\delta = 0.2$. 


When generating the \ls\ in this example, we use $p=10,000$, $\patht = \delta^2$, $\Delta t = \patht/5$. Again the number of landmarks does not matter since LMDS will return the exact result (up to machine precision) every time. Next for each of 10 randomly chosen starting locations we run 10,000 paths from each simulator. Then we compare them using a common coarse grained net with $\delta_c = 0.2$ as in section \ref{sec_simcomp}. The output is shown in figure \ref{fig_comp_twnr}. Again we notice that the errors are small for all times, including the range of timescales where transitions occur. 

In order to calculate transition times, we must first define the regions of interest. Region $i$ will be a ball of radius 1/4 around $p_i$. We will use these same regions for future examples stemming from this potential well. For a comparison of the transition times, see figure \ref{fig_rates_twnr}.

\begin{figure}[tbp]
\centering\includegraphics[width=5.5in]{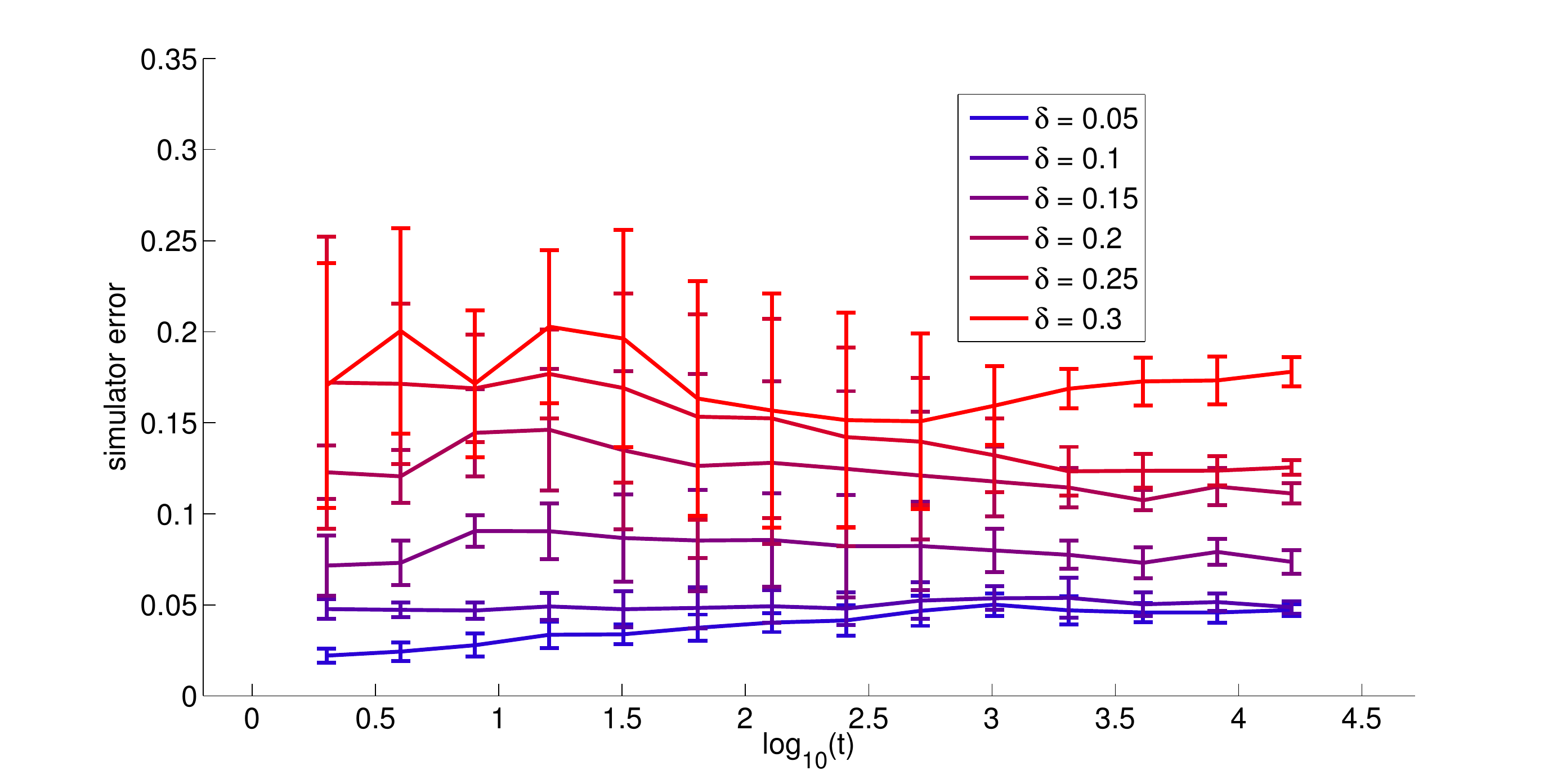}
\caption{Comparison of \ls's with original simulator in the smooth three well potential from example \ref{ex_twnr}.}\label{fig_comp_twnr}
\end{figure}

\begin{center}
\begin{figure}[tbp]
\centering\includegraphics[width=5.5in]{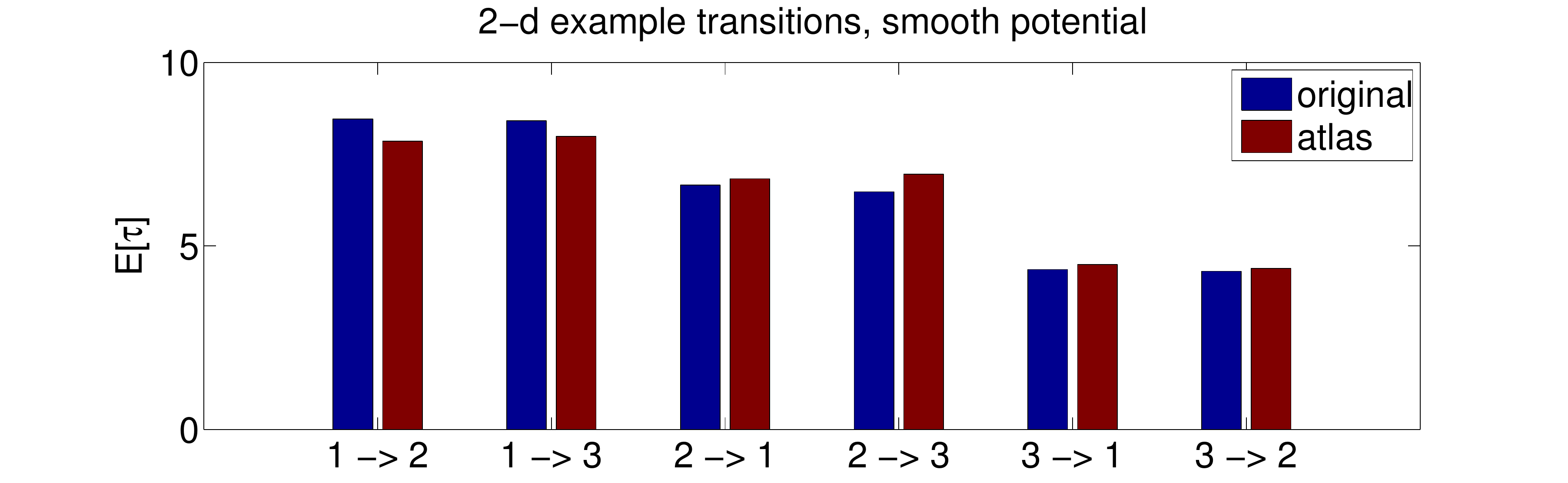}
\caption{Comparing transition times in example \ref{ex_twnr}}\label{fig_rates_twnr}
\end{figure}
\end{center}

\subsubsection{Rough Potential} \label{ex_twwr}
In the next example we take $U_2(x)$ and add a fast oscillating component to simulate small scale interactions as in example \ref{ex_dwwr}. The new potential well is 
\begin{align}
 V_2(x) = U_2(x) + \frac16\cos(100\pi x_1) + \frac16\cos(100\pi x_2).
\end{align}
And again see a simulator which approximates the process $X_t$.
\begin{align}
 dX_t = -\nabla V_2(X_t) dt + dB_t
\end{align}

As a result of the high frequency oscillations, the the timesteps will be of size 0.00005. This example will show that our algorithm is robust to fast oscillations of the potential even in a more complicated system. In this example we will again avoid using evenly spaced points as input, and run the grid points through the simulator for a short time $t=0.01$. These are samples we could obtain from running the original simulator for a long time, or using some kind of fast exploration technique. Again, the distribution of this point set is irrelevant as long as there are no holes of size $\delta$. 

For this system we will use $\delta = 0.2$ which will return $\delta$ nets with $\approx 230$ net points. We will again use use $p=10,000$, $\patht = \delta^2$, $\Delta t = \patht/5$ for consistency, even though $p$ could be chosen smaller (since $\delta$ is larger). Again, the timestep of the \ls\ is $\Delta t = 0.004$ which is over 100 times larger than the timesteps of the original simulator, and thus the \ls\ runs about 100 times faster. For the simulator comparison with this example see figure \ref{fig_comp_twwr}. Define the regions the same as in example \ref{ex_twnr}. To see the transition times, see figure \ref{fig_rates_twwr}.

\begin{figure}[tbp]
\centering\includegraphics[width=5.5in]{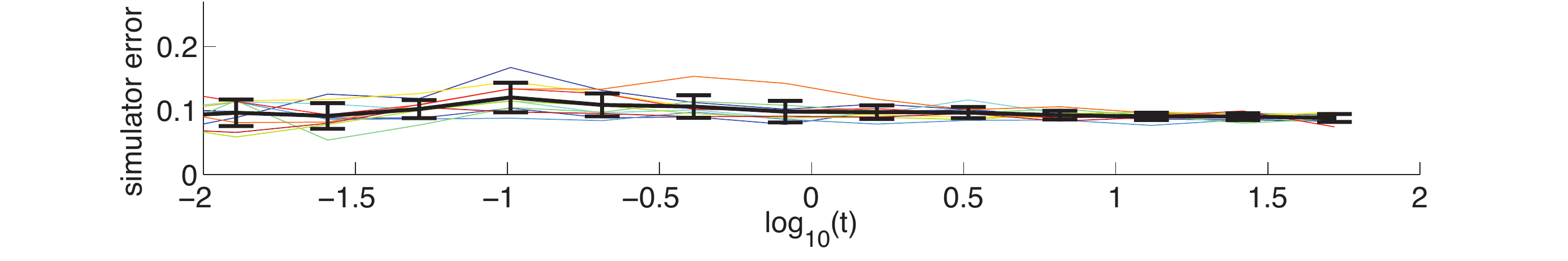}
\caption{Comparison of original simulator with the \ls\ ($\delta=0.2$) in the rough three well potential from example \ref{ex_twwr}.}\label{fig_comp_twwr}
\end{figure}

\begin{center}
\begin{figure}[tbp]
\centering\includegraphics[width=5.5in]{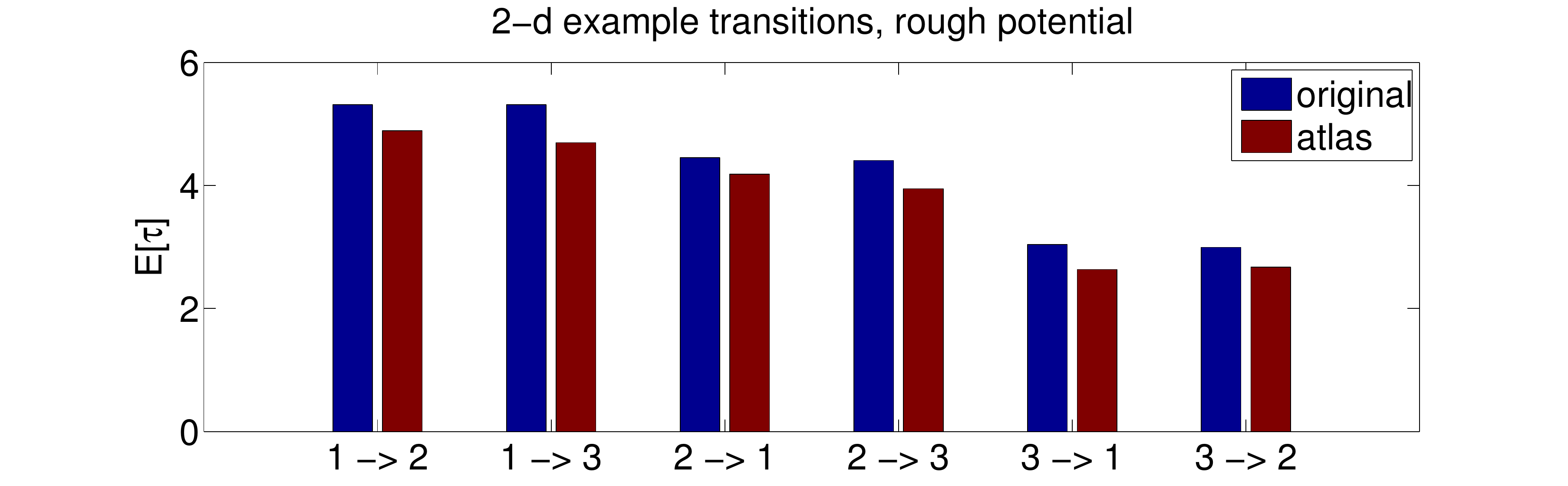}
\caption{Comparing transition times in example \ref{ex_twwr}}\label{fig_rates_twwr}
\end{figure}
\end{center}

\subsection{Random Walk on Images}\label{ex_twimg}
Next we will embed the two dimensional three well examples from sections \ref{ex_twnr} and \ref{ex_twwr} into $D=12,500$ dimensions. The high dimensional embedding is given by the following algorithm given a two dimensional point $x$:

\begin{enumerate}
 \item Generate a mesh $\{z_j\}$ on $[-1.5,3.5] \times [-1.5,2.5]$ with evenly spaced grid points and spacing 0.04.
 \item The output vector $v$ at position $j$ is 1 if $ |z_j - x| < 1/2 $ and 0 otherwise. 
\end{enumerate}

See figure \ref{fig_circimg} for an example image generated by this algorithm run on the point $(0,0)$. The natural distance to use in this space is the hamming distance, which counts the number of different entries. It induces a norm, which we call $||v||_1$ since this is the same as the 1-norm of the vector on $\R^D$. Given a binary vector $v$, we can write the ''inverse`` $\wt x$

\begin{align}
 \wt x = ||v||_1^{-1} \sum_j v_j z_j
\end{align}

This just averages the positions of the pixels $\{z_j\}$, which should roughly return the center of the circle in the image.  Any two dimensional simulator now can be mapped to a simulator on $\R^D$ in the following way:
\begin{enumerate}
 \item Given input $v \in \R^D$ and a time $\patht$, calculate the two dimensional point $\wt x$ from the approximate inverse mapping.
 \item Run the 2-d simulator for time $\patht$ with initial condition $x_0 = \wt x$.
 \item Take the output of the simulator, $X_{\patht}$ and map it to $\R^D$ with the high dimensional embedding.
\end{enumerate}

Next, we rescale the distance function by the constant $(0.04)^2/2$ so that the new norm is locally equivalent to the original distance. In so doing, we can continue using values of $\delta$ that made sense to us in the original space. This high dimensional mapping is nontrivial, and all the possible vectors $v$ we could see span the entire $12,500$ dimensional space. The space can be locally approximated by a 2-d plane for a ball of radius $r < 1/2$, and so we expect the \ls\ to find the appropriate local spaces to estimate the dynamics.  

\begin{figure}[tbp]
\begin{center}
\centering\includegraphics[width=2.5in]{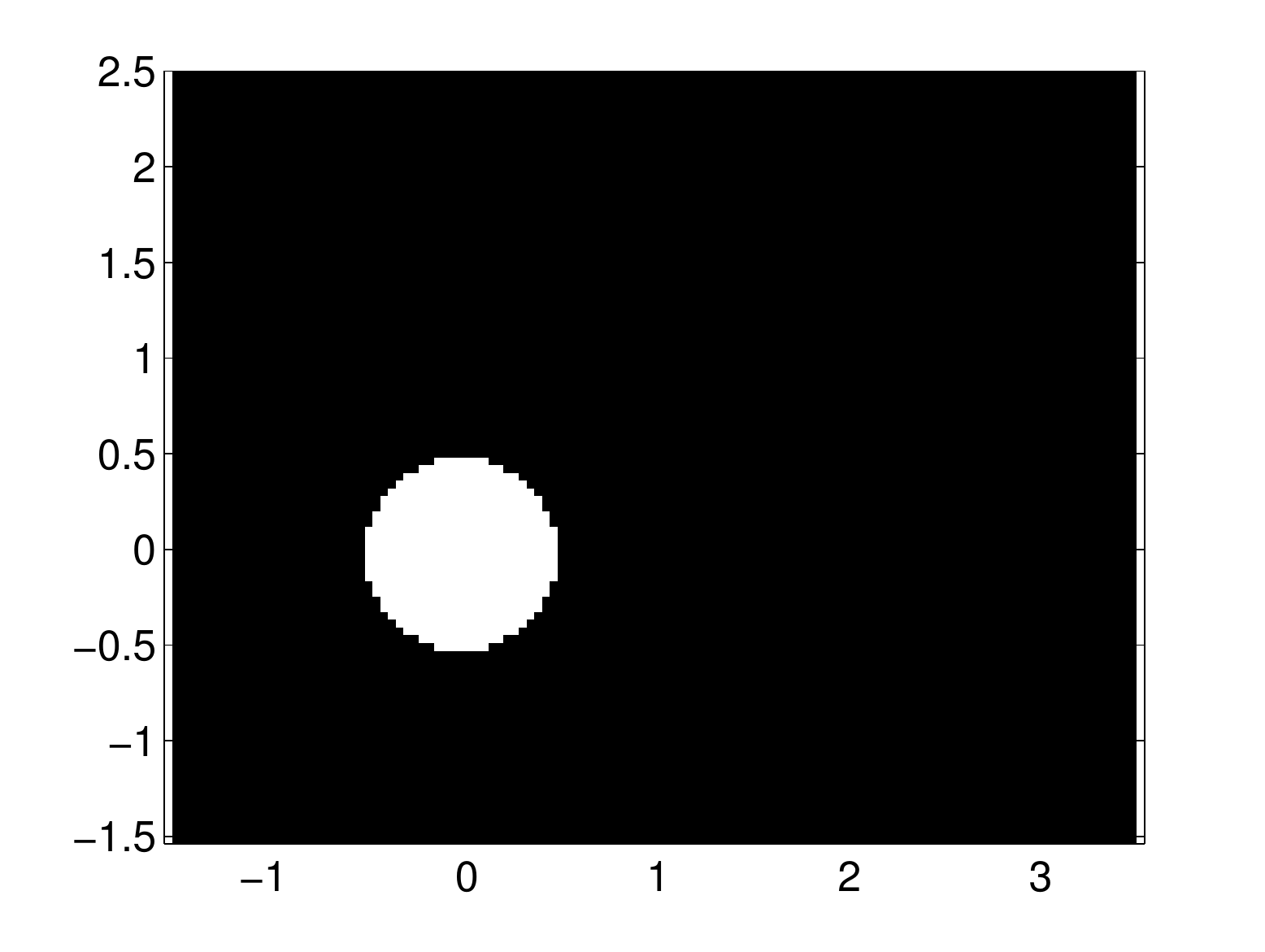}
\caption{Circle image corresponding to the point [0,0].}\label{fig_circimg}
\end{center}
\end{figure}

\subsubsection{Smooth Potential}\label{ex_twimg_nr}
First we will apply the high dimensional mapping to the simulator with smooth potential well $U_2$ from example \ref{ex_twnr}. Next we start with a set of points in $\R^D$ which cover the known state space (the same covering set from before only mapped to $\R^D$). The \ls\ algorithm is given the rescaled hamming distance function for computing distances between vectors, and it is given the simulator which takes points in $\R^D$ and a time $\patht$ and returns points in $\R^D$. Because distances are now $12,500$ times more expensive to compute, for this example we set $p = 1000$ and $m = 20$ landmarks per point. Again keep $\patht=\delta^2$ and $\Delta \patht = t/5$. 

\clearpage

\begin{center}
\begin{figure}[tbp]
\centering\includegraphics[width=5.5in]{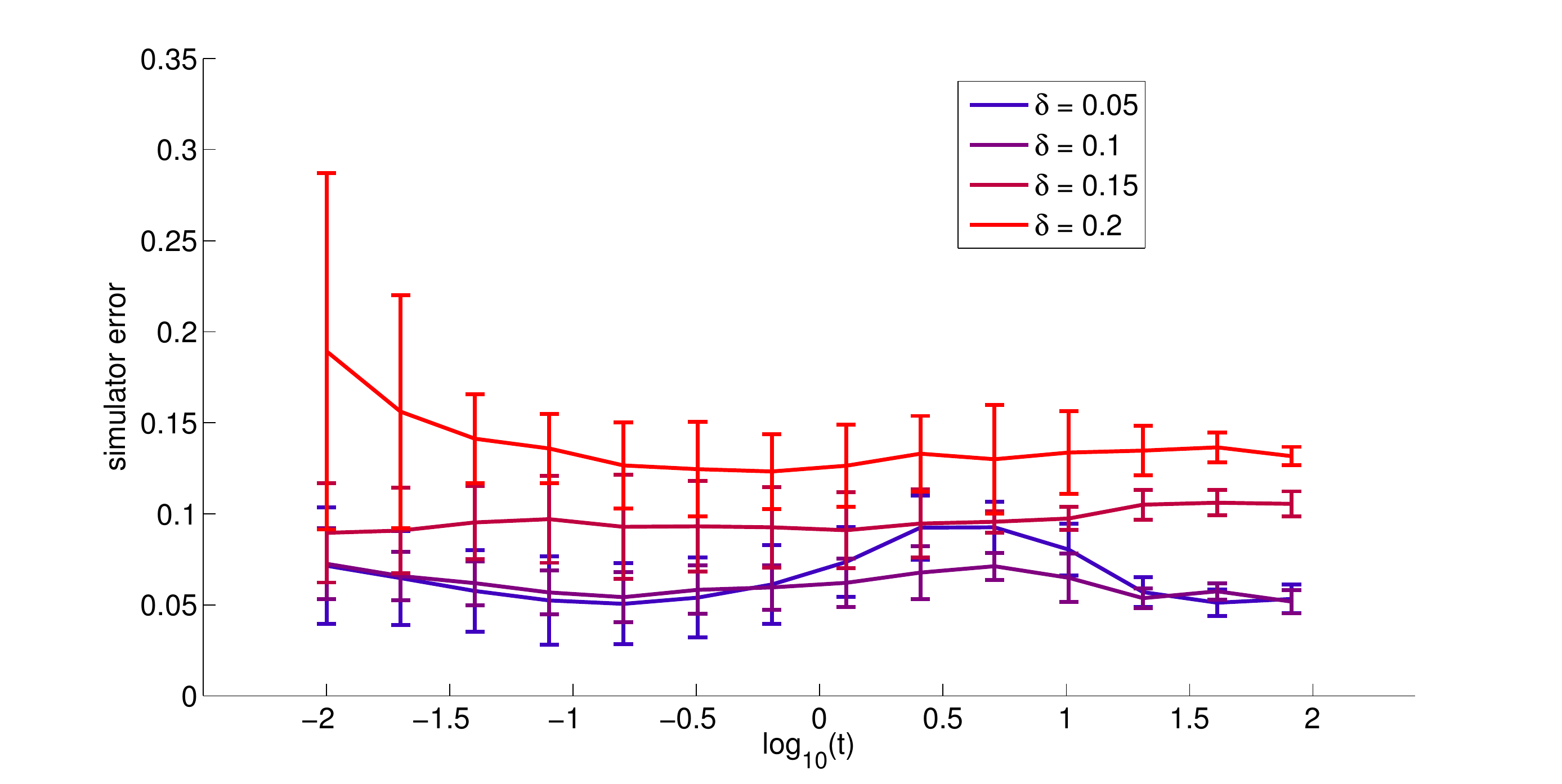}
\caption{Comparison of \ls's with original simulator for example \ref{ex_twimg_nr}.}\label{fig_comp_twimg_nr}
\end{figure}
\end{center}

After constructing multiple \ls's for varying values of $\delta$, we find that the distributions are well approximating the original given simulator. See figure \ref{fig_comp_twimg_nr} for details. The small number of samples, along with the width of the pixels limits the accuracy for small values of $\delta$. In fact we can see that $\delta = 0.05$ returns a simulator which is worse than $\delta = 0.1$.  Define the regions the same as in example \ref{ex_twnr}. To see the transition times, see figure \ref{fig_rates_imgtwnr}.

\begin{center}
\begin{figure}[tbp]
\centering\includegraphics[width=5.5in]{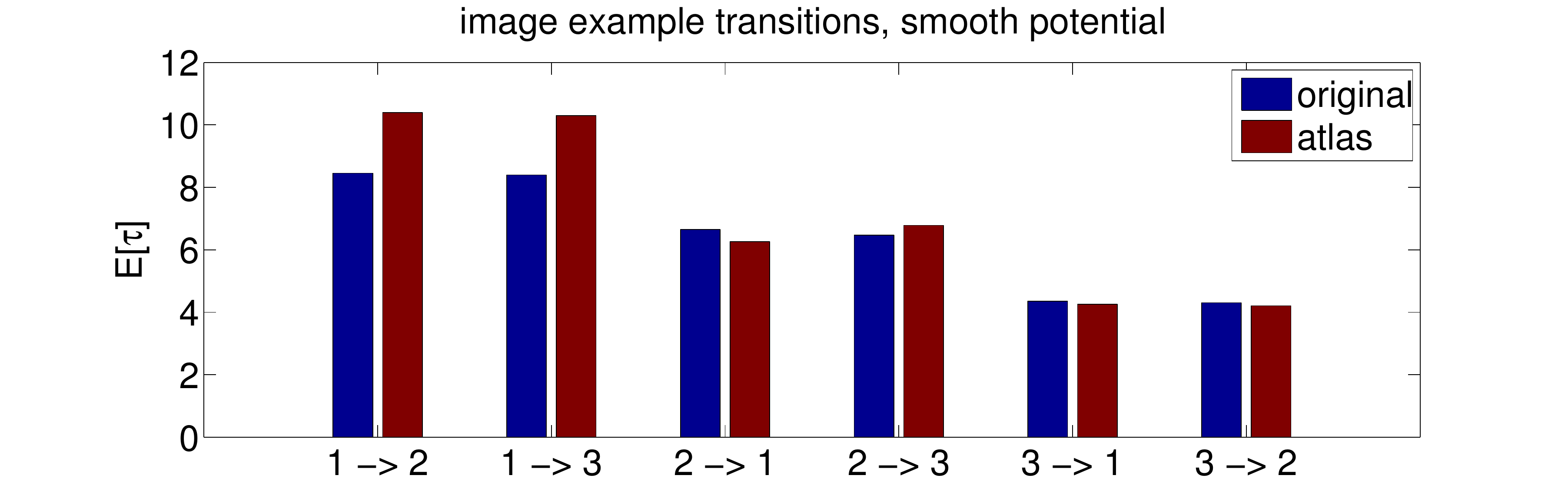}
\caption{Comparing transition times in example \ref{ex_twimg_nr}}\label{fig_rates_imgtwnr}
\end{figure}
\end{center}

\subsubsection{Rough Potential}\label{ex_twimg_wr}
In the next example of this paper, we will apply the high dimensional transformation to the rough potential well $V_2$ from example \ref{ex_twwr}. Again, we give the algorithm the same set of initial points from example \ref{ex_twwr} mapped to $\R^D$ along with the simulator using $V_2$ embedded in high dimensions. In this example we use $\delta = 0.2$, $p=2000$, $m=40$, $\patht=\delta^2$ and $\Delta t = \patht/5$. Again the simulation timescale of the local simulator is $100$ times larger than that of the original simulator. The \ls\ has a running time which depends only on the local dimensionality of the system, and so the ambient dimension only enters in the construction phase. 

\begin{center}

\begin{figure}[tbp]
\centering\includegraphics[width=5.5in]{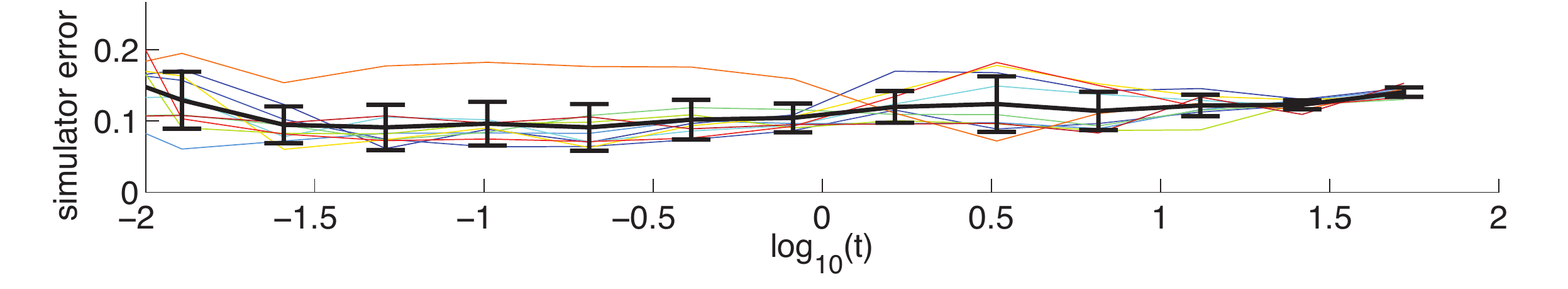}
\caption{Comparison of the \ls\ $\delta = 0.2$ with original simulator for example \ref{ex_twimg_wr}.}\label{fig_comp_twimg_wr}
\end{figure}
\end{center}
After simulating $10,000$ paths for each of $10$ different initial conditions, we can test the simulator error (see figure \ref{fig_comp_twimg_wr}). Because running the original simulator is very expensive for this system, we used the same original simulator samples (mapped to $\R^D$) for comparison as in figure \ref{fig_comp_twwr}. 
  Define the regions the same as in example \ref{ex_twnr}. To see the transition times, see figure \ref{fig_rates_imgtwwr}.

\begin{center}
\begin{figure}[tbp]
\centering\includegraphics[width=5.5in]{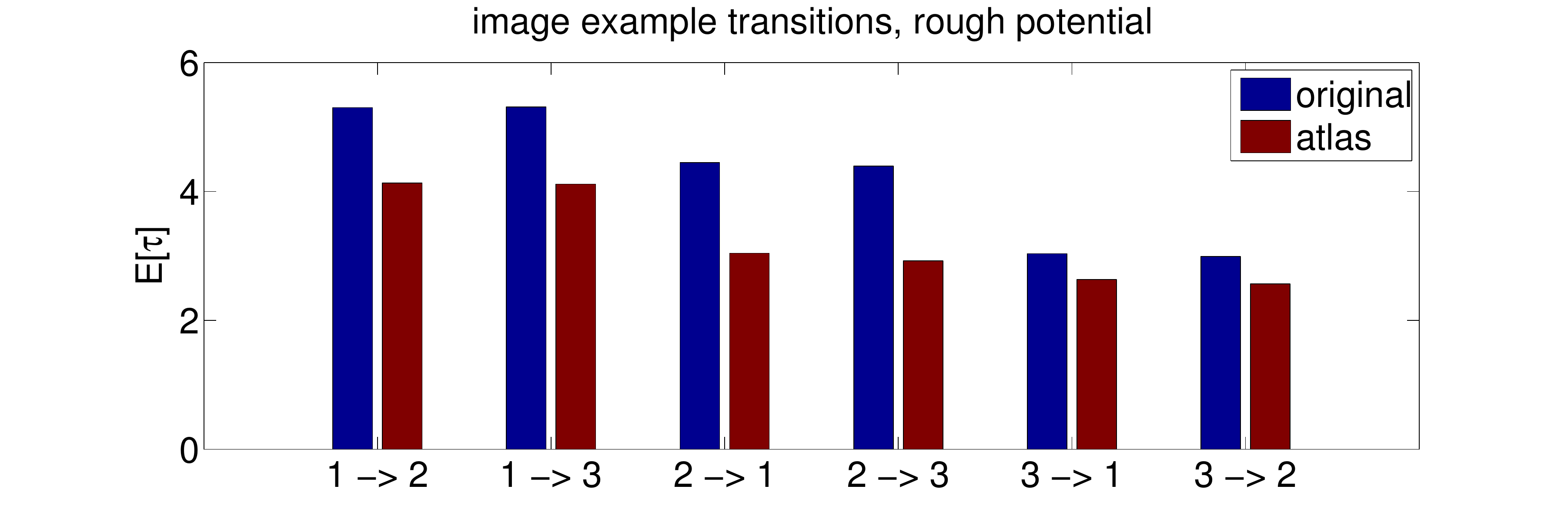}
\caption{Comparing transition times in example \ref{ex_twimg_wr}}\label{fig_rates_imgtwwr}
\end{figure}
\end{center}


\subsection{Randomly forced string}\label{ex_fcn}

\begin{figure}
\centering
    \textbf{Function Simulator}\par
\begin{minipage}{4in}
\begin{framed}
\begin{algorithmic}
 \item[] $f = \osim(f)$
 \item[] 
 \item[] \hspace{.1in} $\%$ \emph{simulate Brownian bridge}
 \item[] \hspace{.1in} $W = $ cumsum(randn(1,100)) 	
 \item[] \hspace{.1in} $W = W - W(1)$ 		
 \item[] \hspace{.1in} $W = W - x*W(100)$		
 \item[] 
 \item[] \hspace{.1in} $\%$ \emph{Add bridge to f, smooth and renormalize}
 \item[] \hspace{.1in} $f = f + (1/100)*W$ 		
 \item[] \hspace{.1in} $f = $ smooth$(f)$		
 \item[] \hspace{.1in} $f = f*(f_{\text{norm}}/\text{norm}(f))$	
\end{algorithmic}
\end{framed}
\end{minipage}
\caption{Pseudocode for a single step of the simulator used in example \ref{ex_fcn}. $f_{\text{norm}}$ is a fixed constant equal to the norm of sin($\pi x$). The function smooth is MATLAB's default smoothing algorithm.}
\label{f:fcn_sim}
\end{figure}

\begin{center}

\begin{figure}[tbp]
\centering\includegraphics[width=5.5in]{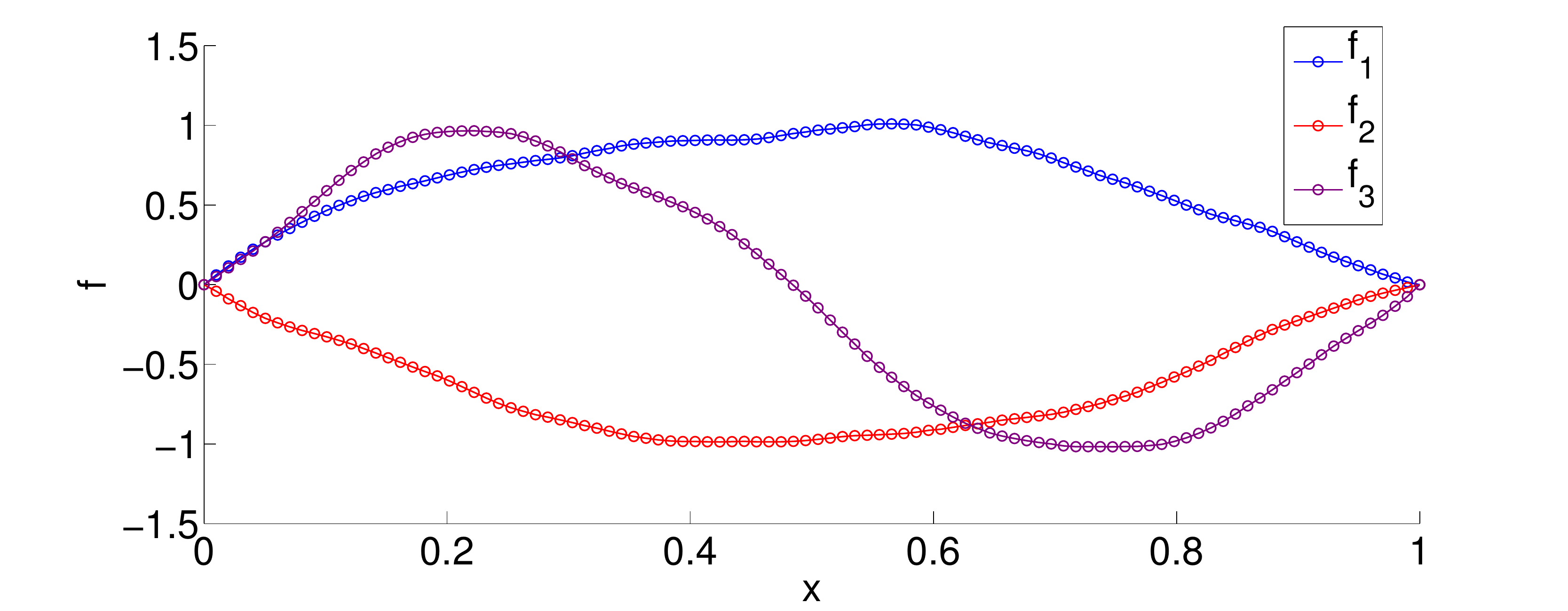}
\caption{Three typical outputs of the simulator from section \ref{ex_fcn}.}\label{fig_functions}
\end{figure}
\end{center}

In this example, we are given a dynamical system in the form of a random walk on functions on $[0,1]$ with endpoints fixed at zero. These functions are represented by values on a grid of 100 evenly spaced points (including the ends). Typical functions seen as output from the simulator are shown in figure \ref{fig_functions}. The distance we will use is euclidean distance in $\R^{100}$, rescaled by 1/100 to approximate the $L^2$ distance on functions. A single step of the simulator is done by adding a Brownian path fixed at the endpoints, then smoothing the result and renormalizing. The pseudocode is shown in figure \ref{f:fcn_sim}.

This behavior of this system in characterized by large dwelling times near the smoothest functions ($f_1$ and $f_2$ from figure \ref{fig_functions}) with rare transitions ($10^3 - 10^4$ steps) across functions like that of $f_3$ in figure \ref{fig_functions}. The three constraints $f(0)=0, f(1)=0, ||f|| = ||f_0||$ force the functions to live on $S^{97}$, a 97 dimensional sphere with radius $||f_0||$. Although we expect these functions to lie near a low dimensional submanifold $\M \subset S^{97}$ because of the smoothing step, a single step of the simulator could take us anywhere on $S^{97}$; this means the outputs of our simulator are never exactly on $\M$. This is an important aspect of this example, as real world data typically will have small noise in the ambient space.

One can think of this simulator as a discretization of the SDE on $S^{97}$
\begin{align}
 dX_t = F(X_t)dt + \sigma(X_t) dW_t \label{eqn_fsde}
\end{align}
For an appropriate choice of $F,\sigma$. One can also think of this as a discretization scheme for a stochastic partial differential equation (SPDE) of the form
\begin{align}
 \frac{\partial}{\partial t} f_t = \frac{\partial^2}{\partial x^2} f_t + b(f_t) + \sum_{j=1}^\infty g_j(f_t) dW^j_t \label{eqn_spde}
\end{align}
for an appropriate choice of drift $b$ and orthogonal functions $\{g_j\}$. One can think of \eqref{eqn_spde} as an infinite dimensional analogue to \eqref{eqn_fsde} with each coordinate $X^j_t = \langle f_t, g_j \rangle$ being driven by a one dimensional brownian motion. 

In order to generate the \ls, first we must generate an initial sampling of the space. In order to do this, we start with $50,000$ renormalized Gaussian vectors, the uniform distribution on $S^{97}$. Next we want to ''heal`` these samples by running them through the simulator. One can see with some observation that 250 steps is large enough that the noise is killed; samples with 250 steps of ``healing'' are similar to those with 500 steps of ``healing''.

Next we wish to select parameters $\delta, \patht$. We expect that the system may be homogenized at a time scale of $\patht = 250$ steps for the following reasons: $\patht$ is an order of magnitude below the scale of major events of the system, $\patht$ is an order of magnitude above the scale of the noise (since even the most noisy inputs have been smoothed by time $\patht$). The parameter $\delta$ is closely tied to the choice of $\patht$. We measure the average distance moved by paths of length $\patht$ starting from our healed samples to be $0.3 = \delta$. Next we choose the minor parameters $p,m,d$. In these experiments, we use $p=5000$ and $m=40$. As discussed in section \ref{sec_LMDS}, we can choose $d$ based upon the singular values obtained through LMDS. Choosing a cutoff of $(\delta/4)^2$ for the eigenvalues yields $d=3$ over $99\%$ of the time. Using $d=3$ and comparing with the original simulator in the usual way yields figure \ref{fig_comp_fcn}. 
\begin{center}
\begin{figure}[tbp]
\centering\includegraphics[width=5.5in]{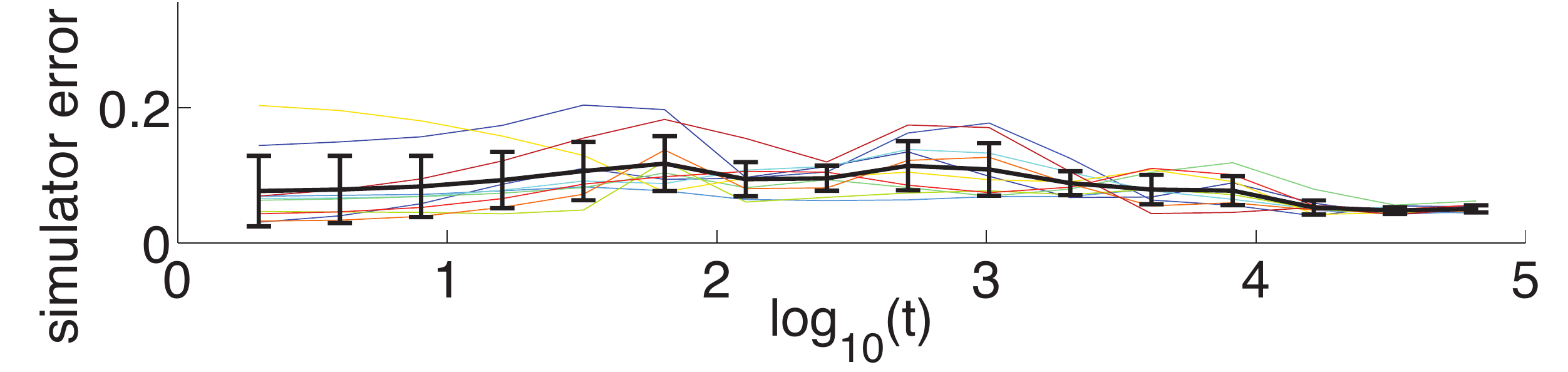}
\caption{Comparison of the \ls\ with original simulator for example \ref{ex_fcn}.}\label{fig_comp_fcn}
\end{figure}
\end{center}
In general, it is better to overestimate $d$ than underestimate; underestimating $d$ may lose important degrees of freedom causing failure, while overestimating $d$ will only affect the computational cost mildly. In fact the algorithm is robust to the choice of $d$, provided $d$ is large enough to capture the important degrees of freedom.  See figure \ref{fig_comp_fcn_dim} to see results for varying values of the choice of $d$. 

\begin{center}
\begin{figure}[tbp]
\centering\includegraphics[width=5.5in]{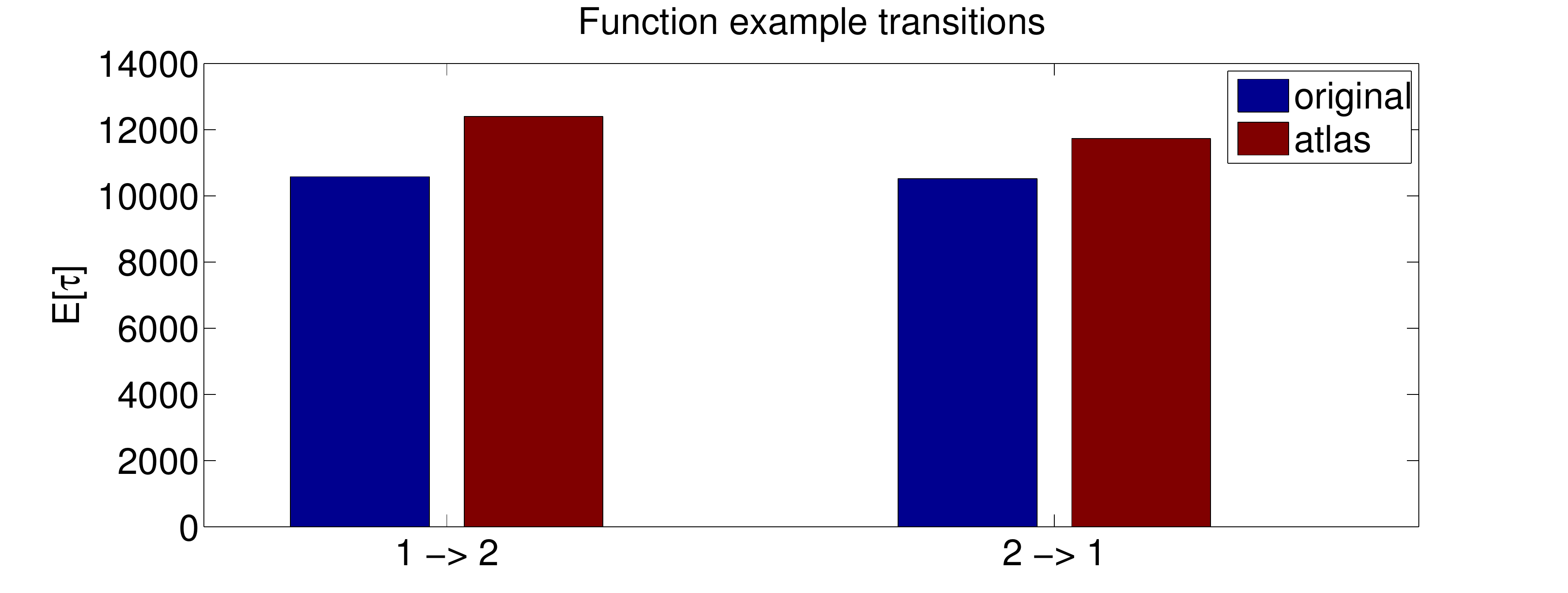}
\caption{Comparing transition times in example \ref{ex_fcn}}\label{fig_rates_fcn}
\end{figure}
\end{center}

Next we wish to compare the transition times between states. In order to do this, we define region 1 to be a ball of radius 1/4 around sin($\pi x$), and region 2 to be a ball of radius 1/4 around -sin($\pi x$).  To see a comparison of transition times between these regions, see figure \ref{fig_rates_fcn}.

\begin{center}
\begin{figure}[tbp]
\centering\includegraphics[width=5.5in]{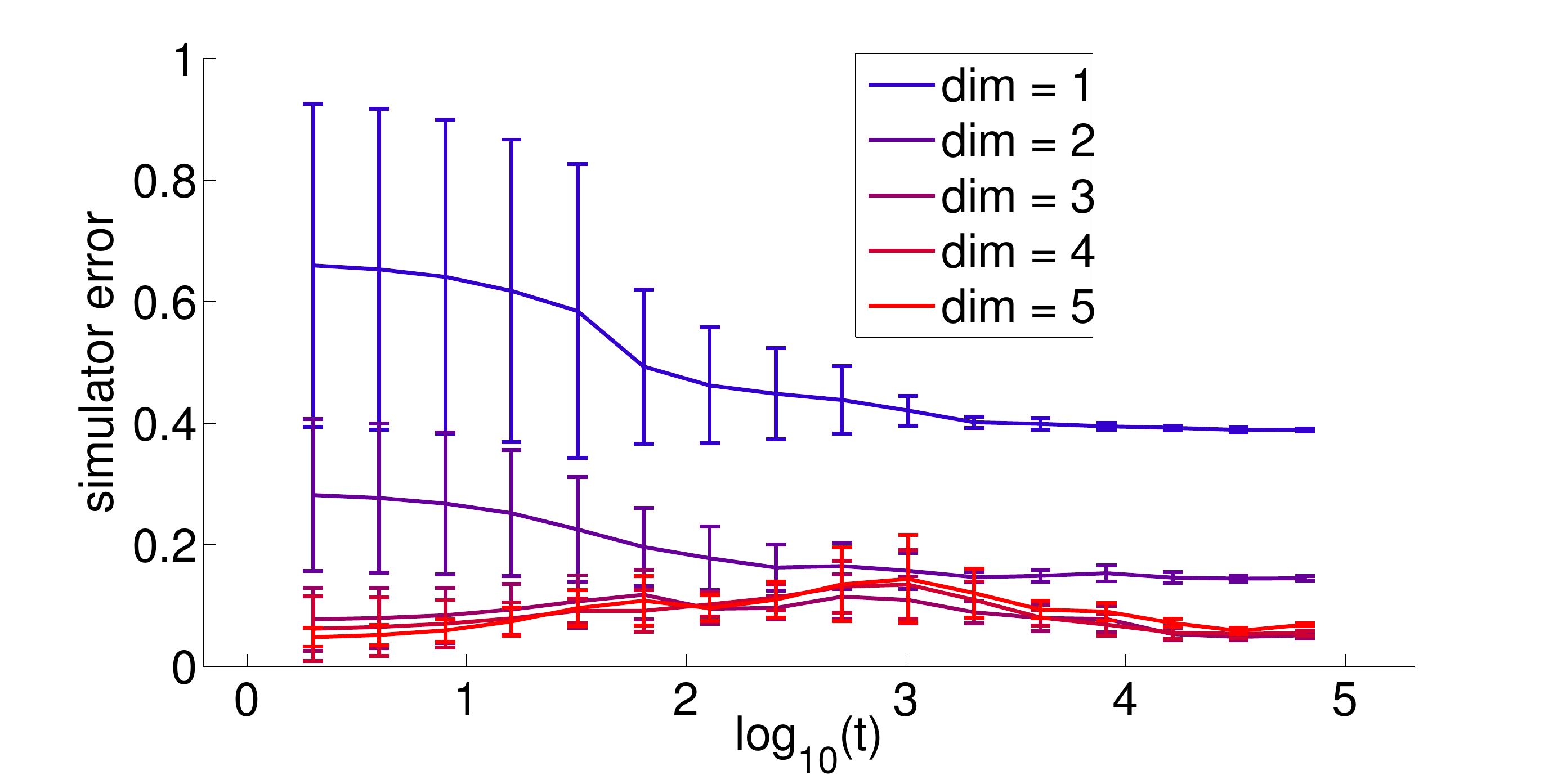}
\caption{Comparison for example \ref{ex_fcn} varying $d$, the dimension of the \ls.}\label{fig_comp_fcn_dim}
\end{figure}
\end{center}

The \ls\ constructed for this example again captures the important aspects of the original simulator. The \ls\ is again faster in this example due to two factors: decreased dimensionality and increased timestep. The dimensionality of the \ls\ is 3 as compared to the original 100, and the timestep of the simulator is equivalent to 50 of the original steps (250/5 since $\Delta t = \patht/5$).

\subsection{Chaotic ODE system} \label{ex_ode}
In many real world systems, noise arises from ensembles of deterministic chaotic processes. In this example we apply our algorithm to a multiscale ODE driven by small scale deterministic chaos. Consider the set of multiscale ODEs with a scale parameter $\eps$:
\begin{equation}
	\begin{cases}
		\dot{ X }_{t}^{\eps} = \eps f(X_t^\eps) + g(Y_t), & X_0^\eps = x\\
		\dot{Y}_t = h(Y_t) & Y_0 = y \label{eqn_ODE}
	\end{cases}
\end{equation}
Systems of this form (although slightly more general) are studied in \cite{vanden2003fast}. Suppose the dynamics for $Y_t$ alone have an invariant measure $\mu$, and $\E_\mu[f] = \Ord(\eps)$. Then the system \eqref{eqn_ODE} behaves like the SDE

\begin{equation}
 dX_s = b(X_s)ds + \sigma(X_s)dB_s
\end{equation}
on the timescale $s=\eps t$ in the limit as $\eps \rightarrow 0$. For fixed $\eps$, such systems are difficult to simulate directly due to the timescale separation. 

We start by choosing functions $f,g,h$ and scale parameter $\eps$. Start by choosing $Y_t$ to be the Lorenz '96 system with 80 dimensions and $F= 8$ (thus fixing $h$). Each coordinate $Y_i(t)$ is governed by equation \ref{eqn_lorenz}, where indices wrap around (so $Y_{-1} = Y_{79}, Y_0 = Y_{80}, Y_{81} = Y_1$). 
\begin{equation}
\dot{Y}_i = -Y_{i-2}Y_{i-1} + Y_{i-1}Y_{i+1} - Y_i + F\label{eqn_lorenz}
\end{equation}

We fix $\eps = 0.01$ and let $f(X_t)$ be the cartesian coordinate version of the system in equations \ref{eqn_radsystem1},\ref{eqn_radsystem2}:
\begin{align}
\dot{r} &= -(r - 3/4)(r - 3/2)(r - 2) \label{eqn_radsystem1} \\
\dot{\theta} &= r - 3/2\label{eqn_radsystem2}
\end{align}

Last, we choose $g = [g_1(y), g_2(y)]$ to be
\begin{align}
 g_1(y) = \frac{1}{320}\sum_{i \in I_1}y_i - 0.2925 \quad,\quad
 g_2(y) = \frac{1}{320}\sum_{i \in I_2}y_i - 0.2925 
\end{align}
with $I_1 = [1:10, 21:30, 41:50, 61:70]$ and $I_2 = I_1^c$. This choice of $g$ was made in order that $g$ is approximately mean zero and variance 1 with respect to the measure $\mu$. The system is then solved using the Runge-Kutta method with timestep 0.05. Last, we multiply $Y_t$ by a small constant $10^{-4}$ in a post processing stage in order that these directions do not overpower the interesting states of the system (the first two coordinates). We use $10^{-4}$ since a typical value of each coordinate of $Y_t \approx 10$, there are $80$ variables, so this ensures the norm of $Y_t$ is lower than $\eps$. The system then behaves like
\begin{equation}
 d\bar{X}_s = f(\bar{X}_s) + dB_s \label{eqn_limitingsde}
\end{equation}
on the timescale $s = \eps t$. The ODE system associated with $f$ (i.e. without the stochastic term $dB_s$) has two stable attracting limit cycles, one at $r = 3/4$ and the other at $r = 2$. A slope field for $f$ is shown, along with a sample path of $X_t$, in figure \ref{fig_odepath}.

\begin{center}
\begin{figure}[tbp]
\centering
\centering\includegraphics[width=5.5in]{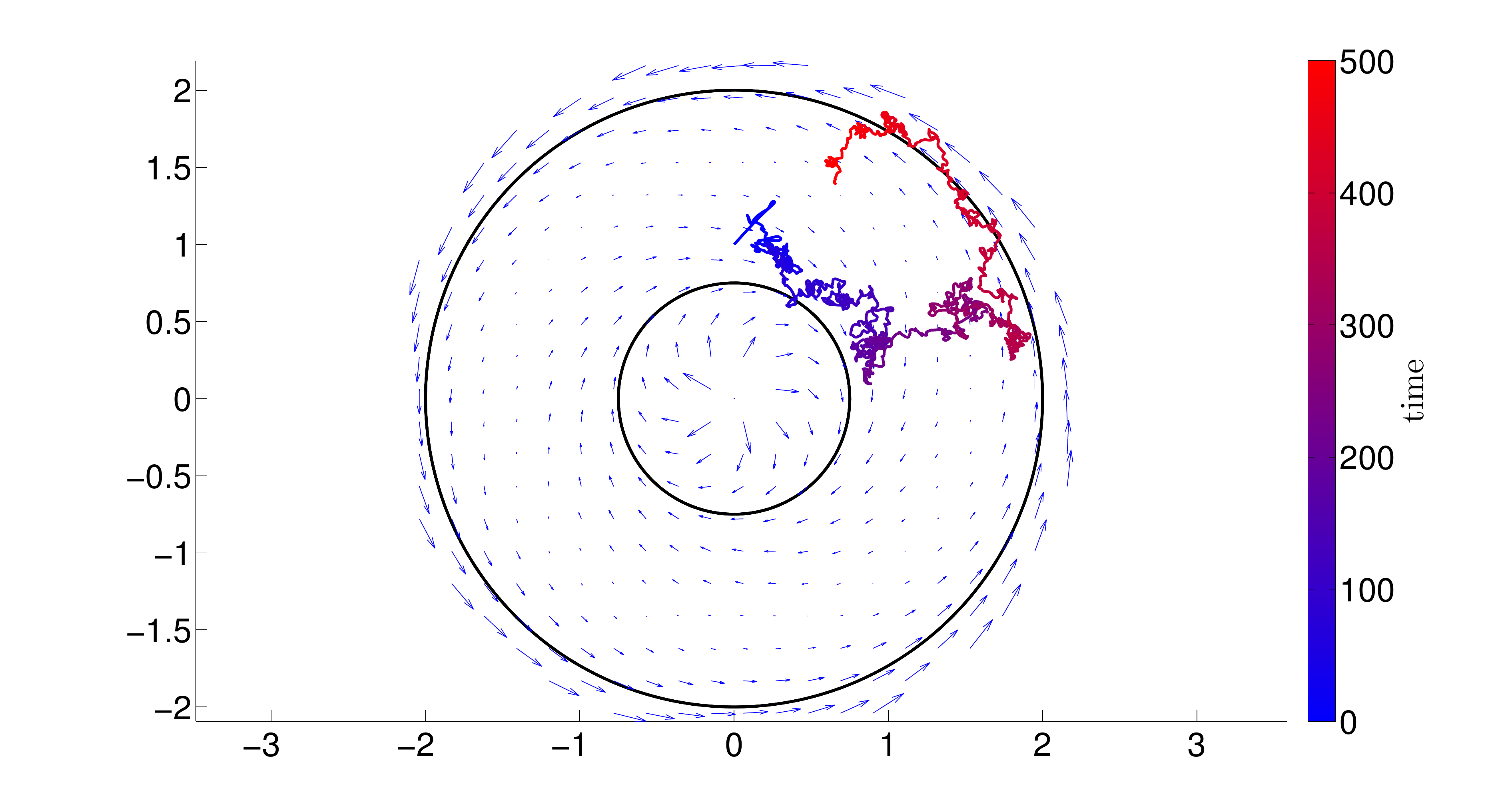}
\caption{Colored path: example simulation for example \ref{ex_ode}. Black circles represent the stable limit cycles, and blue vectors represent $f$ at that location.}\label{fig_odepath}
\end{figure}
\end{center}

We cannot apply directly our algorithm to the above, since multiple runs will yield the same result (and have covariance zero). For this reason we will add small noise ($10^{-5}$ times a random normal) to the initial condition we give as input, and pretend our simulator is of the form \eqref{eqn_limitingsde}. Because of the chaos in the system, this small perturbation propagates quickly through the system, yielding us a different ''realization`` of the chaos $Y_t$.

\begin{center}
\begin{figure}[tbp]
\centering
\centering\includegraphics[width=5.5in]{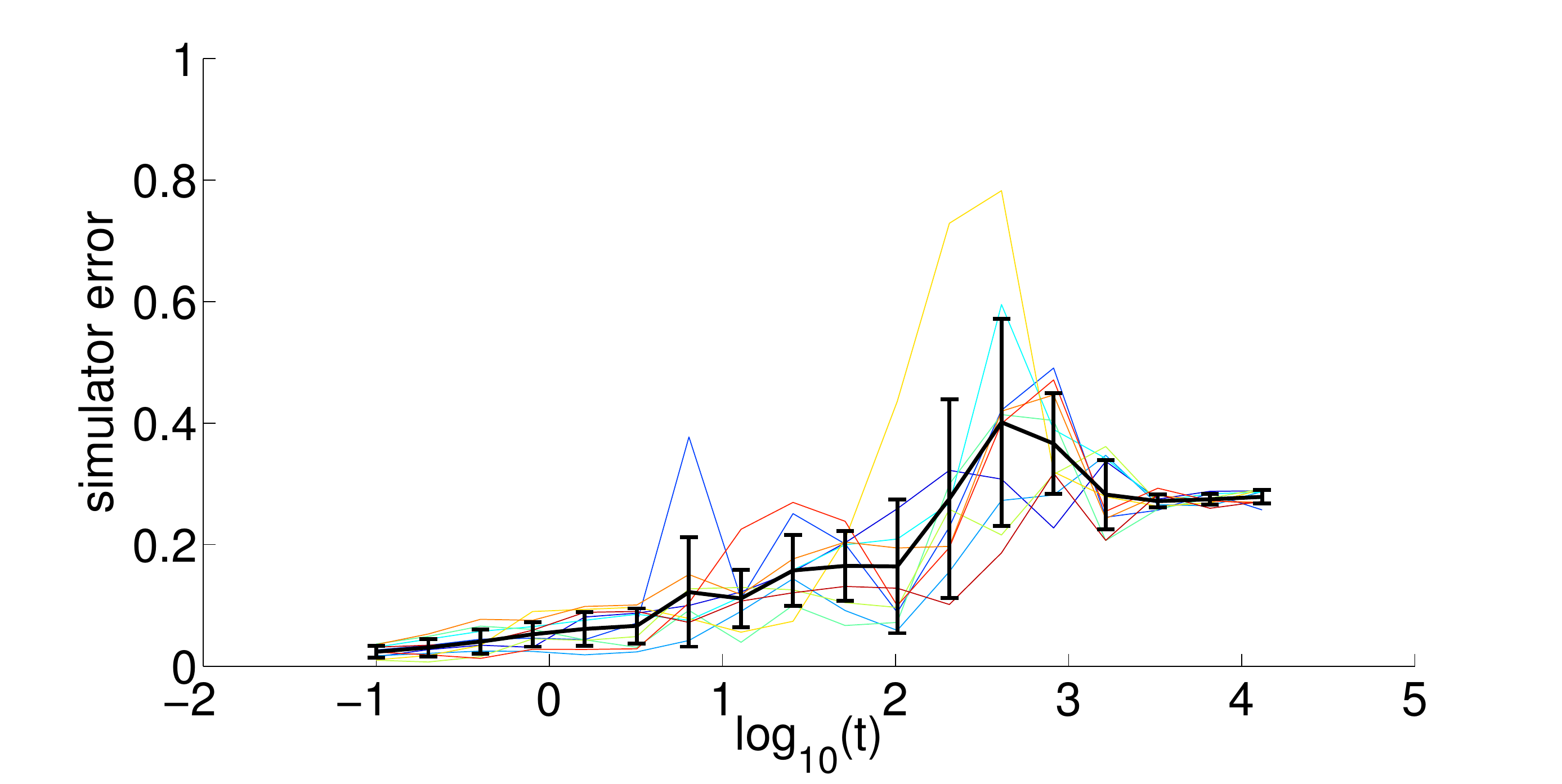}
\caption{Comparison of the \ls\ with original simulator for example \ref{ex_ode}.}\label{fig_comp_ode}
\end{figure}
\end{center}

\begin{center}
\begin{figure}[tbp]
\centering
\centering\includegraphics[width=5.5in]{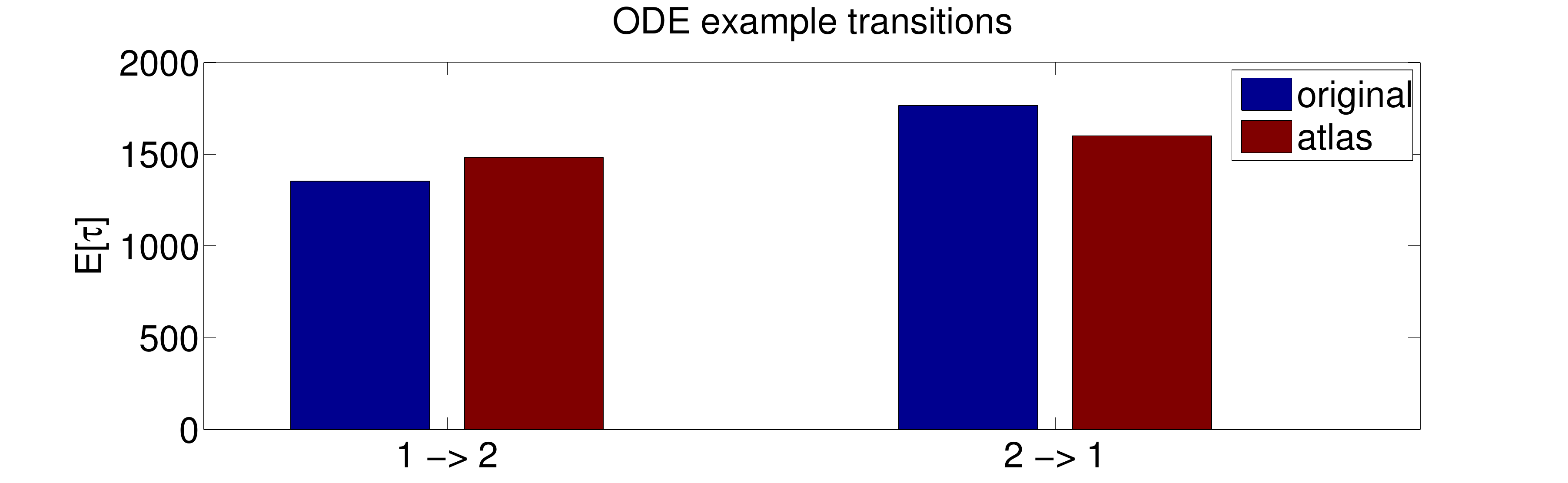}
\caption{Comparing transition times in example \ref{ex_ode}}\label{fig_rates_ode}
\end{figure}
\end{center}

Next we need to generate a large set of input points for the $\delta$-net. In this case we use 82 dimensional random normal vectors with the last 80 directions multiplied by $10^{-4}$ (in order to be similar in size to typical outputs of the simulator). We run these random samples through the simulator for a short time in order to ''heal`` them so that the points given to us are actual outputs of the true simulator we are given. Next we must choose $t_0 = \Ord(\eps^{-1})$, since this is the timescale on which the SDE dynamics occur. The exact value we use is $t_0 = 25 = (1/4)\eps^{-1}$. In this amount of time, the system travels roughly $\delta = 0.3$. The \ls\ we construct compared with the true simulator is shown in figure \ref{fig_comp_ode}. Here we choose region 1 to be everywhere $r < 1$, and region 2 to be everywhere $r > 7/4$. To see a comparison of the transition times between these regions see figure \ref{fig_rates_ode}. Notice that the rates of transition between states are accurate, but the comparison is not very accurate for intermediate times. This is because the speed at which the system travels around the limit cycles is not well approximated.


\section{
Extensions and Discussion
} \label{sec_disc}

There are many open problems related to this work, some of which we mention here.

Theorem \ref{thm_main} reveals that the local learning algorithm works well on compact SDEs with Lipschitz drift and diffusion. 
We consider only bounded domains in the proof to make thing simpler, although the same framework can be applied to the unbounded case with tight transition density. In this case, one has to worry about parts of the space which are unexplored, but seldomly reached. Indeed we see that some of our examples have unbounded state spaces, and the algorithm performs as desired.

The framework we introduced may be generalized to richer families of local simulators, enabling the approximation of larger classes of stochastic systems. Proving large time accuracy may be difficult for such systems, so it is an open problem how much one is allowed to change these local simulators.
Many molecular dynamics (MD) systems remember the velocity of atoms and so do not follow an SDE of the form \eqref{eqn_Ydef} which is memoryless. 
A subject of ongoing research is to use more complex models locally to be able to capture dynamics of typical MD systems.

Another subject of future work is efficient computation of the function $G$, which is the inverse MDS mapping. One can always approximate this function up to order $\delta$ via a piecewise constant function (returning the chart center). In some cases, such as when $\metric$ is the root mean square distance (RMSD), it is possible to create an inverse mapping which has error of order $\delta^2$ via local linear approximations, using ideas from \cite{CM:MGM2}.

Using the \ls\ as a basis for generating samples from the stationary distribution is useful for quickly computing diffusion maps for these systems. A subject of interest is to understand how the errors made by the \ls\ propagate through diffusion maps. How similar do diffusion maps look generated by samples from the \ls\ as compared to diffusion maps generated directly from the original simulator?

In some problems, choosing $\delta$ and $t_0$ is difficult.
Another subject of ongoing research is a robust way of choosing these parameters based on short simulations.
For simplicity in this paper we have assumed that $\delta$ and $t_0$ is constant for each $k \in \Gamma$, but it is possible to have these parameters depend on the location $y_k$ (and perhaps statistics of short sample paths). 

Last but not least, this construction as described here still requires a large number of steps to sample rare events and reach stationarity, i.e. it does not address the problem of accelerating the sampling of rare events or overcoming energy barriers. In many important applications, e.g. molecular dynamics, such barriers force the simulations to be extremely long (e.g. $10^{12}-10^{14}$ time-steps is common). The point of this work is to produce a simulator that is much faster (in real world time) than the original fine scale simulator. It is important to note that any of the many techniques developed over the years to attempt to overcome this problem may be used in conjunction with our construction, i.e. it can be run on our \ls, instead of the original expensive fine scale simulator. This yields a double gain in simulation speed, combining the gains of a faster simulator with those of an importance sampler that efficiently samples rare events.

\section{
Acknowledgement
}
The authors gratefully acknowledge partial support from NSF CAREER DMS-0847388, NSF CHE-1265920 and ONR N00014-12-1-0601.
We thank J. Mattingly for useful discussions, and Y. Kevrekidis for introducing us to the questions around model reduction and equation free methods, and many discussions on these topics over the years.
Finally, we would like to thank the anonymous reviewer, whose comments on the first version of this manuscript greatly helped us improve the presentation and readability of this work.

\begin{figure}
\centering
    \textbf{Notation and Definitions}\par
\begin{framed}

\textbf{Notation:}
\begin{itemize}
 \item For $f,g$ functions on the same domain, we say that $f=\Ord(g)$ if there is a constant $C>0$ such that for all $x$ in the domain of $f$ and $g$ we have $f(x)\le Cg(x)$.
 \item If $F:M \rightarrow N$ is a measurable function and $\mu$ is a measure on $M$ then we define the push forward measure of $\mu$ through $F$ for any $A\subset N$ by $F_*\mu(A) = \int_M \mathbbm{1}_A(F(x))d\mu(x)$.
\end{itemize}

\textbf{Definitions:}
\begin{itemize}
 \item $\Gamma = \{y_k\}$: Set of points in the net.
 \item $i \sim j$: Neighbor connections on net indices \eqref{sec_dnet}.
 \item $\MDS_i$: Mapping created by LMDS associated with chart $i$, see section \ref{sec_LMDS}.
 \item $\csp$: Atlas constituted by the collection of $\{(x,i) | x \in \R^d, |x|<2\delta, i \in \Gamma\}$.
 \item $S_{i,j}$: Chart transition map $S_{i,j}(x) = (x - \mu_{i,j})T_{i,j} + \mu_{j,i}$, see equation \eqref{eqn_switchingmap}.
 \item $i'$: Shorthand denoting the new chart index chosen starting from $(x,i) \in \csp$, see equation \eqref{eqn_ip}.
 \item $W$: Wall function which keeps the \ls\ within $2\delta$ of the chart center, see equation \eqref{e:W}.
 \item $q$: density (with respect to volume measure on $\mathcal{M}$) of the stationary distribution for the original process $Y_t$. 
 \item $\mu$: stationary distribution of the \ls\ process $Z_k$.
 \item $\wh q$: 'almost stationary' distribution for $\wh X_t$ (see \eqref{eqn_qhatdef}).
 \item $C^k, ||\cdot||_{C^k}$: H\"older space of order $k$.
\end{itemize}

\end{framed}
\caption{List of notation and definitions used for reference}
\label{sec_notdef}
\end{figure}

\bibliography{LLDS}
\bibliographystyle{plain}

\end{document}